\documentclass[reqno, 12pt, a4paper, twosided, english]{amsart}

\usepackage{tikz}
\usetikzlibrary{matrix,arrows,calc,patterns,decorations.markings, patterns.meta, decorations.pathmorphing}

\usepackage[mathscr]{eucal}
\usepackage{amscd}
\usepackage{amsfonts}
\usepackage{amsmath, amsthm, amssymb, bbm, bm}
\usepackage{stmaryrd}
\usepackage{lscape}
\usepackage{hyperref}
\usepackage{setspace}

\usepackage{latexsym}
\usepackage{graphics}

\usepackage[normalem]{ulem}

\usepackage{tikz-cd, mathtools}
\usepackage[utf8]{inputenc}
\usepackage[english]{cleveref}
\usepackage[english]{babel}
\usepackage{enumitem, float, comment}
\usepackage[multiple]{footmisc}
\usepackage{xspace}
\let\oldFootnote\footnote
\newcommand\nextToken\relax

\renewcommand\footnote[1]{%
    \oldFootnote{#1}\futurelet\nextToken\isFootnote}

\newcommand\isFootnote{%
    \ifx\footnote\nextToken\textsuperscript{,}\fi}

\usetikzlibrary{shadows,arrows.meta,positioning,backgrounds,fit,shapes,arrows,chains,matrix,decorations.pathreplacing,calc,calligraphy}
\usepackage{contour}
\contourlength{0.8pt}

\theoremstyle{plain}
\newtheorem{thm}{Theorem}[section]
\newtheorem{cor}[thm]{Corollary}
\newtheorem{lem}[thm]{Lemma}
\newtheorem{prop}[thm]{Proposition}
\theoremstyle{definition}

\newtheorem{defn}[thm]{Definition}
\newtheorem{ex}[thm]{Example}

\newtheorem{Setup}[thm]{Setup}
\newtheorem{Notation}[thm]{Notation}
\theoremstyle{remark}
\newtheorem{rem}[thm]{Remark}
\newtheorem{Claim}[thm]{Claim} %

\Crefname{thm}{Theorem}{Theorems}
\crefname{thm}{Theorem}{Theorems}
\Crefname{defn}{Definition}{Definitions}
\crefname{defn}{Definition}{Definitions}
\Crefname{Notation}{Notation}{Notations}
\crefname{Notation}{notation}{notations}
\Crefname{cor}{Corollary}{Corollaries}
\crefname{cor}{corollary}{corollaries}
\Crefname{lem}{Lemma}{Lemmas}
\crefname{lem}{lemma}{lemmas}
\Crefname{rem}{Remark}{Remarks}
\crefname{rem}{remark}{remarks}

\crefformat{footnote}{#2\footnotemark[#1]#3}

\newtheorem*{Key Statement}{Key Statement}

\newcommand{\ul}[1]{\underline{#1}}

\DeclareMathAlphabet{\mathpzc}{OT1}{pzc}{m}{it}

\DeclareMathOperator{\Perf}{\mathsf{Perf}}

\DeclareMathOperator{\Sing}{\mathsf{Sing}}

\newcommand{\GP}{\opname{GP}}

\DeclareMathOperator{\rad}{\mathsf{rad}}
\DeclareMathOperator{\tors}{\mathsf{tors}}

\DeclareMathOperator{\coker}{\mathsf{coker}}

\renewcommand{\ker}{\mathsf{ker}}
\newcommand{\im}{\mathsf{im}}

\renewcommand{\dim}{\mathsf{dim}}

\DeclareMathOperator{\Coh}{\mathsf{coh}}

\DeclareMathOperator{\MCM}{\mathsf{MCM}}

\newcommand{\SMAL}{\mathfrak{S}}

\DeclareMathOperator{\prdim}{\mathsf{pr.dim}}
\DeclareMathOperator{\injdim}{\mathsf{inj.dim}}
\DeclareMathOperator{\gldim}{\mathsf{gl.dim}}

\DeclareMathOperator{\add}{\mathsf{add}}

\DeclareMathOperator{\Mod}{\mathsf{Mod}}

\newcommand{\ind}{\opname{ind}}

\DeclareMathOperator{\Hom}{\mathsf{Hom}}

\DeclareMathOperator{\Ext}{\mathsf{Ext}}

\DeclareMathOperator{\End}{\mathsf{End}}

\DeclareMathOperator{\Spec}{\mathsf{Spec}}

\newcommand{\opname}[1]{\operatorname{\mathsf{#1}}}

\newcommand{\cok}{\opname{cok}\nolimits}
\renewcommand{\ker}{\opname{ker}\nolimits}

\renewcommand{\mod}{\opname{mod}\nolimits}

\newcommand{\proj}{\opname{proj}\nolimits}
\newcommand{\Proj}{\opname{Proj}\nolimits}

\renewcommand{\Mod}{\opname{Mod}\nolimits}

\renewcommand{\add}{\opname{add}\nolimits}
\newcommand{\op}{^{\mathrm{op}}}

\newcommand{\ca}{{\mathcal A}}
\newcommand{\cb}{{\mathcal B}}
\newcommand{\cc}{{\mathcal C}}

\newcommand{\ce}{{\mathcal E}}
\newcommand{\cf}{{\mathcal F}}
\newcommand{\cg}{{\mathcal G}}

\newcommand{\ci}{{\mathcal I}}

\newcommand{\ck}{{\mathcal K}}
\newcommand{\cl}{{\mathcal L}}
\newcommand{\cm}{{\mathcal M}}

\newcommand{\co}{{\mathcal O}}
\newcommand{\cp}{{\mathcal P}}

\newcommand{\cs}{{\mathcal S}}
\newcommand{\ct}{{\mathcal T}}
\newcommand{\cu}{{\mathcal U}}

\newcommand{\cx}{{\mathcal X}}

\newcommand{\Z}{\mathbb{Z}}
\newcommand{\N}{\mathbb{N}}
\newcommand{\Q}{\mathbb{Q}}
\newcommand{\C}{\mathbb{C}}

\renewcommand{\P}{\mathbb{P}}

\newcommand{\id}{\mathbf{1}}

\input xy
\xyoption{all}

\newcommand{\kk}{\Bbbk}
\newcommand{\bsm}{\begin{smallmatrix}}
\newcommand{\esm}{\end{smallmatrix}}

\newcommand{\sD}{D}

\newcommand{\GG}{\mathbb{G}}

\newcommand{\ZZ}{\mathbb{Z}}

\newcommand{\cA}{\mathcal{A}}

\newcommand{\set}[2]{\left\{\,{#1}\,\mid\,{#2}\,\right\}}

\renewcommand{\epsilon}{\varepsilon}
\newcommand{\KSOD}{KSOD\xspace}
\newcommand{\KSODs}{KSODs\xspace}
\newcommand{\SOD}{SOD\xspace}

\newcommand{\SODs}{SODs\xspace}  %
\newcommand{\cSODs}{SODs\xspace} %
\newcommand{\absorption}{algebraic categorical absorption\xspace}
\newcommand{\absorptions}{algebraic categorical absorptions\xspace}

\renewcommand{\subset}{\subseteq}
\renewcommand{\supset}{\supseteq}
\newcommand{\sg}{\mathrm{sg}}
\newcommand{\hf}{\mathrm{hf}}

\newcommand{\Dsg}{\opname{D}^{\mathrm{sg}}}
\newcommand{\Dfd}{\opname{D}^{\mathrm{fd}}}
\newcommand{\Db}{\opname{D}^{\mathrm{b}}}
\newcommand{\Kb}{\opname{K}^{\mathrm{b}}}

\title[Obstructions to semiorthogonal decompositions II]{Obstructions to semiorthogonal decompositions for singular projective varieties II: Representation theory}
\author{Martin Kalck}
\address{
    Institut für Mathematik und Wissenschaftliches Rechnen\\
    Universität Graz\\
	8010 Graz\\
	Austria
 }
\email{martin.kalck@uni-graz.at}
\author{Carlo Klapproth}
\address{
    Institut for Matematik\\
	Aarhus Universitet\\
	8000 Aarhus C\\
	\mbox{Denmark}
 }
\curraddr{
	Institut für Algebra und Zahlentheorie (IAZ)\\
	Universität Stuttgart\\
 	70569 Stuttgart\\
 	Germany
}
\email{carlo.klapproth@mathematik.uni-stuttgart.de}
\author{Nebojsa Pavic}
\address{
	Institut für Mathematik und Wissenschaftliches Rechnen\\
    Universität Graz\\
	8010 Graz\\
	Austria
 }
\email{nebojsa.pavic@uni-graz.at}

\usepackage{todonotes}

\begin{document}



\begin{abstract}
    We show that odd-dimensional projective varieties with \emph{tilting objects} and only ADE-hypersurface singularities are nodal, i.e.\ they only have $A_1$-singularities. This is a very special case of more general obstructions to the existence of semiorthogonal decompositions for projective Gorenstein varieties. More precisely, for many isolated \emph{hypersurface} singularities, we show that 
    Kuznetsov--Shinder's \emph{categorical absorptions of singularities} cannot contain tilting objects.

    The key idea is to compare singularity categories of projective varieties to singularity categories of finite-dimensional associative Gorenstein algebras. The former often contain special generators, called \emph{cluster-tilting objects}, which typically have loops and $2$-cycles in their quivers. In contrast, quivers of cluster-tilting objects in the latter categories, can never have loops or $2$-cycles.
\end{abstract}

\maketitle
\tableofcontents

\section{Introduction}
Derived categories of coherent sheaves on algebraic varieties are rich structures, that serve as bridges to different mathematical areas -- for example, to symplectic geometry, representation theory and mathematical physics. Applications include the construction of new hyperkähler varieties via Bridgeland stability \cite{BL+}, a new approach to the study of the Yang--Baxter equation \cite{Polishchuk, BurbanKreussler},   
and a deeper understanding of curve counting invariants \cite{BrownWemyss} and threefold flops \cite{WemyssFlops}.

A key tool to gain insights into the complicated structures of derived categories are \emph{semiorthogonal decompositions (\SODs)}.  
An important special case are full exceptional sequences, which severely restrict the possible Hodge structures  \cite{MT} and are also expected to imply generic semi-simplicity of big quantum cohomology (Dubrovin's conjecture).  

However, for singular projective (Gorenstein) varieties $X$, full exceptional sequences cannot exist, cf.\ \cite{KS22, KPS19}! Together with recent works of Kawamata \cite{Kaw19} and Karmazyn--Kuznetsov--Shinder \cite{KKS} this motivates the study of \emph{Kawamata type semiorthogonal decompositions (\KSODs)}. These are admissible \SODs that naturally generalize exceptional sequences to the singular case \cite{ KPS19}. More precisely, we consider admissible \SODs of $\Db(X) \coloneqq \Db(\Coh X)$ of the form
\begin{align*}
\Db(X) = \langle \mathcal{P}, \Db(R_1), \ldots, \Db(R_n) \rangle, \text{ where  $\mathcal{P} \subseteq \Perf(X)$ and $\Db(R_i) \coloneqq \Db(\mod R_i)$}
\end{align*}
for some finite-dimensional $\C$-algebras $R_i$, `capturing the singular information of $X$'.
Here, $\Perf(X) \subset \Db(X)$ is the subcategory of perfect complexes on $X$.
For $n=1$ and $\mathcal{P}=0$, \KSODs specialize to \emph{tilting equivalences}
\begin{align}\label{def:tilting}
    \Db(X) \cong \Db(\End_{\Db(X)}(\ct)),
\end{align}
where $\ct \in \Perf(X)$ is a \emph{tilting object}, that is $\smash{\Hom_{\Db(X)}(\ct, \ct[i])=0}$ for all $i \neq 0$ and for all $\cf \in \Db(X)$, if $\smash{\Hom_{\Db(X)}(\ct, \cf[i])=0}$ for all $i \in \ZZ$ then $\cf=0$, see for example \cite[Theorem 7.6(2)]{HilleVandenBergh}.

Examples of \KSODs for nodal varieties and varieties with quotient singularities are studied in \cite{KKS, Kaw19, KPS19,xieNdp, ps, Kaw21, KKP}. It is natural to ask for further examples of singular varieties admitting \KSODs, in particular, for hypersurface singularities. The goal of this article is to show that there are strong obstructions to the existence of such decompositions. 

We also introduce the more general concept of an \emph{\absorption of closed a subspace $\mathsf{S} \subset \Sing(X)$}. These are \KSODs of admissible subcategories $\cA \subset \Db(X)$ which `capture the singular information of $\mathsf{S}$'. 
In \Cref{subsec:roadmap} we give an overview how this concept relates to \KSODs and tilting and singular equivalences. A roadmap to this article can also be found there.

\subsection{Obstructions to \KSODs}
In \cite{KPS19}, using (negative) K-theory, `global' obstructions to the existence of \KSODs are studied with main focus on nodal singularities.

In this article, we focus on `complete local' obstructions, i.e.\ obstructions that come from the analytic type of the singularities.
%
%
Surprisingly, our first result (Theorem \ref{T:GoalNEW}) rules out \KSODs for odd-dimensional varieties with the `mildest' \emph{non-nodal} isolated singularities: 
Arnol'd's \emph{simple hypersurface singularities} also known as \emph{ADE-hypersurface singularities}, whose definition is recalled in \Cref{rem:explicit}. This is a very special case of our main result \Cref{T:Main}.%

\begin{thm}\label{T:GoalNEW}
  Let $X$ be an odd-dimensional projective Gorenstein variety over $\C$ such that $\Db(X)$ admits a \KSOD.\nopagebreak

  If $s \in \Sing(X)$ is an ADE-hypersurface singularity,
  then $s$ is already nodal, that is an $A_1$-singularity $\smash{\widehat{\co}_{X,s}} \cong \C\llbracket z_0, \ldots, z_{d}\rrbracket/(z_0^2 + \cdots + z_{d}^2 )$ with $d = \dim X$.
  \end{thm}

Theorem \ref{T:GoalNEW} 
allows us to prove non-existence of \KSODs in new cases.
We illustrate this in the following example: we 
construct two explicit cubic threefolds having $A_3$ respectively $A_{11}$-singularities (hence `complete local' obstructions), and explain that they have no `global' obstructions as studied in \cite{KPS19}.
%
%
%
%
%
%

%
%

%
%
\begin{ex}~\label{E:A_11}
\begin{itemize}
    \item[(a)] Let $X$ be a threefold with only $A_{2m}$-singularities, then Theorem \ref{T:GoalNEW} shows that $X$ does not admit \KSODs. Because $\mathrm{K}_{-1}(X)$ vanishes, see for example \cite[Corollary 3.8]{KPS19}, the global obstructions do not apply in this case. 

      An explicit example is the cubic threefold $X$ in $\mathbb{P}^4$ given by
    \[
    z_{4} ( z_0^2 + z_1^2 + z_{2}^2 ) + z_0^3 + z_1^3  + z_{2}^3 + z_{3}^3 = 0.
    \]
    It has a singularity of type $A_2$ at $p=[0:0:0:0:1]$ and is smooth away from $p$. 

    \item[(b)] Let us consider the cubic threefold $X \subset \mathbb{P}^4$ given by the equation
\[
z_4 ( z_3^2 - z_0 z_2 ) + z_2^3 + z_1^2 z_0 - 2z_1 z_2 z_3 + z_0^3 = 0 .
\]
This cubic  has an $A_{11}$-singularity at $q= [0:0:0:0:1]$ and  is smooth away from $q$ (see e.g.\ \cite[Section 2]{All03}).
By Theorem \ref{T:GoalNEW} we have that $X$ does not admit a \KSOD.

Also in this example, we do not have global obstructions to the existence of \KSODs. We show that $\mathrm{K}_{-1}(X) \otimes \Q$ vanishes -- with a more detailed computation, which we might include in another paper, one can prove that $\mathrm{K}_{-1}(X)=0$.
Indeed, since $X$ contains the 
rational cubic scroll\footnote{By a change of coordinates given by $z_0' = z_1 - i z_0$ and $z_1' = z_1 + i z_0$, we see that $Z$ is given by the equations $z_2^2 = z_3 z_0'$, $z_2 z_4 = z_0' z_1'$ and $z_4 z_3 = z_2 z_1'$ in $\mathbb{P}^4$, describing an embedding of a rational cubic scroll inside $\mathbb{P}^4$.}
\[
Z = \{ z_2^2 = z_3 ( z_1 - i z_0  ) , \, z_2 z_4 = z_1^2 + z_0^2 , \, z_4 z_3 = z_2 ( z_1 + i z_0 ) \} \subset \mathbb{P}^4,
\]
we see by \cite[Theorem 1.2]{MV23} that $X$ has non-vanishing defect, see for example \cite[Definition 3.9]{KPS19}, and by \cite[Corollary 3.8]{KPS19} we further conclude that the defect of $X$ is maximal.
Hence, by \cite[Remark 3.11]{KPS19} we see that $\mathrm{K}_{-1}(X) \otimes \mathbb{Q}$ vanishes. 
\end{itemize}

\end{ex}
Our main result \Cref{T:Main}, generalizes \Cref{T:GoalNEW} in two ways.

\begin{enumerate}[label={(\alph*)}]
    \item\label{intro:a} By allowing more general \SODs, e.g.\ certain `categorical absorptions of singularities' in the sense of Kuznetsov--Shinder. 
    \item\label{intro:b} By extending the class of isolated hypersurface singularities beyond ADE-hypersurface singularities.
\end{enumerate}

We start by describing the class of hypersurface singularities in \ref{intro:b}. To do this, recall the definition of the \emph{(Buchweitz--Orlov) singularity category} of a variety $X$ and of a two-sided Noetherian ring $R$:
\begin{align}\label{eq:singularitycategory}
     \Dsg(X)=\Db(X)/\Perf(X) \quad \text{  and  } \quad \Dsg(R)=\Db(R)/\Kb(R).
\end{align}
Here, $\Kb(R) \coloneqq \Kb(\proj R)$ is the homotopy category of bounded complexes of finitely generated projective $R$-modules.
\begin{defn} \label{D:frakS}
    {Let $(S, \mathfrak{m})$ be a complete local $\C$-algebra with an isolated singularity and residue field $\C \cong S/\mathfrak{m}$. 
    A \emph{resolution of singularities} $\pi : Y \to \Spec(S)$ of $S$ is a projective morphism from a regular scheme $Y$ to $\Spec (S)$ such that $\pi$ is an isomorphism away from the singular point $\mathfrak{m}$.
    For a $3$-dimensional Gorenstein $\mathbb{C}$-algebra $(S , \mathfrak{m})$ we say that a resolution of singularities $Y \to \Spec(S)$ is \emph{small}, if all fibres have dimension at most one.}

 We denote the class of singular $3$-dimensional %
 complete local 
 Gorenstein $\mathbb{C}$-alge\-bras $(S , \mathfrak{m})$ with a small resolution by $\SMAL^{(3)}$ and define $\SMAL$ as the class of complete local singularities $S'$, such that there is an $S \in \SMAL^{(3)}$ and a $\C$-linear triangle equivalence of singularity categories $\Dsg(S') \cong \Dsg(S)$.

\end{defn}

The following remark shows that there are many singularities in $\SMAL$.

\begin{rem}
    By work of Reid \cite{Reid83}, the singularities of type $\SMAL^{(3)}$ are hypersurface singularities, cf.\ \Cref{Reid}. They are the bases of flops between smooth minimal models of threefolds. Knörrer periodicity \cite{Knoerrer} yields triangle equivalences between singularity categories of  hypersurface singularities\footnote{This also allows us to reduce the case of odd-dimensional ADE-hypersurface singularities to the case of ADE-threefold singularities in our proofs.} defined by $f \in \C\llbracket z_0, \ldots, z_d \rrbracket$ and $f + z_{d+1}^2 + z_{d+2}^2 \in \C\llbracket z_0, \ldots, z_{d+2} \rrbracket$, respectively. 
    In particular, for $\C\llbracket z_0, \ldots, z_3 \rrbracket/(f)$ of type $\SMAL^{(3)}$, the singularities $\C\llbracket z_0, \ldots, z_{2l+1} \rrbracket/(f+z_4^2+\cdots+z_{2l+1}^2)$ are of type $\SMAL$. 
    
    Conversely, using work of Gulliksen \cite{Gulliksen} all singularities of type $\SMAL$ are hypersurface singularities. We give a more explicit description in \Cref{S:AppendixDetailsOnSing}.
\end{rem}

  \subsection{Obstructions to tilting for categorical absorption of singularities}
  Let $X$ be a projective variety over $\C$. We explain a refinement of Kuznetsov--Shinder's recently introduced concept of a `categorical absorption of singularities' \cite{KS22a} and the notion of \KSODs in \cite{KPS19}. This needs the following preparation.

\begin{defn}[{Cf.\ \cite[Definition 1.6]{orl6}}]\label{defn:homfiniteobj}
    Let $\ca$ be a triangulated category.
    We say an object $M \in \ca$ is \emph{homologically finite} if for any $N \in \ca$ we have $\ca(M,N[i]) = 0$ for all but finitely many $i \in \Z$.
     We denote the triangulated subcategory of homological finite objects of $\ca$ by $\ca^{\hf}$.
\end{defn}

\begin{defn} \label{D:CatAbs} 
   Let $\mathsf{S}$ be a closed subspace of $\Sing(X)$ and $\ca \subseteq \Db(X)$ be an admissible triangulated subcategory. In particular, there is an induced embedding
   \begin{align*}
    \ca/\ca^{\hf} \subseteq \Dsg(X) \subseteq \overline{\Dsg(X)}
    \end{align*}
    into the idempotent completion of the singularity category of $X$, see \eqref{eq:singularitycategory} as well as \Cref{lem/def:asg}.
Let $\ca^{\sg}$ be the \emph{essential} image of this embedding. 
Then we call $\ca \subseteq \Db(X)$ a \emph{categorical absorption of the singularities of $\mathsf{S}$} if
\begin{align*}
   \Dsg_{\mathsf{S}}(X) \subseteq \ca^{\sg} \subseteq \overline{\Dsg(X)}.
\end{align*}
Here, $\Dsg_{\mathsf{S}}(X) \coloneqq \Db_{\mathsf{S}}(X) / \Perf_{\mathsf{S}}(X)$ is the Verdier quotient of the triangulated subcategories $\Db_{\mathsf{S}}(X) \subset \Db(X)$ and $\Perf_{\mathsf{S}}(X) \subset \Perf(X)$ of complexes with cohomology supported at $\mathsf{S}$, which is a triangulated subcategory of $\Dsg(X)$ by \cite[Lemma 2.6]{Orlov11}.
\end{defn}

  \begin{rem}
     The inclusion $\smash{\Dsg_{\Sing(X)}(X)} \subseteq \Dsg(X)$ can be strict. 
    However, the idempotent completions of both categories coincide, cf.\ \cite{Orlov11}. %
  \end{rem}

\begin{ex}[Kuznetsov--Shinder]\label{E:KuzShi}
Let $X$ be a projective Gorenstein variety. A \emph{categorical absorption of singularities} is an admissible subcategory $\ca \subseteq \Db(X)$ such that 
$\ca^\perp \subseteq \Perf(X)$, cf.\ \cite[Definition 1.1]{KS22a}.

Then $\ca^{\sg} = \Dsg(X)$. Hence, $\ca$ is a categorical absorption of $\mathsf{S}=\Sing(X)$ in the sense of \Cref{D:CatAbs}. 
\end{ex}

\begin{defn}
An \emph{\absorption of the singularities of $\mathsf{S}$} is a categorical absorption of the singularities of $\mathsf{S}$, where the subcategory $\mathcal{A}$ admits an admissible \SOD of the form
\begin{align*}
\ca = \langle \Db(R_1), \ldots, \Db(R_n) \rangle,
\end{align*}
where the $R_i$ are associative $\C$-algebras. If $n=1$, then $\ca$ has a tilting object, cf.\ the tilting equivalence in \eqref{def:tilting}.
\end{defn}

\begin{ex}[\KSOD vs.\ \absorption] \label{E:KSOD}
By the same reasoning as in \Cref{E:KuzShi}, every \KSOD gives rise to an \absorption of $\Sing(X)$.
But the following example shows the concept of \absorptions of singularities is more general. 

Consider the Gorenstein weighted projective space 
\[Y=\P(1^3, 3)=\Proj(\C[z_0, \ldots, z_3]),\] 
where $\deg z_3=3$ and $\deg z_i=1$ else.
It has a tilting bundle $\ct$ by work of Kawamata \cite{Kaw21}. 
Let $C \subseteq Y$ be a curve with an $A_3$-singularity that does not meet the singular point $[0:0:0:1]$ of $Y$, e.g.\ $C=V(z_0^4 - z_1^2z_2(z_2+z_1), z_3)$.
Let $X$ be the blow-up of $Y$ in $C$.
By Orlov's blow-up formula (cf.\ e.g.\ \cite[Theorem 5.1]{KPS19} for a formulation in the singular setting) and Kawamata's tilting result, we have an admissible \SOD
\begin{align*}
    \Db(X)=\langle \Db(Y), \Db(C) \rangle \cong \langle \Db(R), \Db(C) \rangle
\end{align*}
for a finite-dimensional $\C$-algebra $R=\End_Y(\ct)$. These equivalences induce equivalences of singularity categories
\begin{align*}
\Dsg(\widehat{\co}_{X, s}) \oplus  \Dsg(\widehat{\co}_{X, s'})   \cong \overline{\Dsg(X)} \cong \Dsg(R) \oplus \Dsg(C), 
\end{align*}
where $s$ is the strict transform of $[0:0:0:1] \in Y$ and $s'$ is the $A_3$-singularity of $X$. The essential image of $\Dsg(R)$ in $\smash{\overline{\Dsg(X)}}$ is equivalent to $\smash{\Dsg(\widehat{\co}_{X, s})}$. This shows that $s$ admits an \absorption of singularities.
 On the other hand, \Cref{T:Main} will show that odd-dimensional $A_3$-singularities do \emph{not} admit an \absorption of singularities. In particular, $s'$ does not admit an \absorption of singularities and therefore $X$ cannot have a \KSOD.
\end{ex}

\begin{thm}\label{T:Main}
   Let $X$ be a projective Gorenstein variety over $\C$.
   Assume that $d = \dim(X)$ is odd and that $X$ contains an \emph{isolated} singularity $s$, which is an ADE-hypersurface singularity or satisfies $\smash{\widehat{\co}_{X,s} \in \SMAL}$.

  If there is an \absorption $\ca = \langle \Db(R_1), \dots, \Db(R_n) \rangle$ of $s$, then $s$ is a nodal singularity, i.e.\ an $A_1$-singularity $\smash{\widehat{\co}_{X,s}} \cong \C\llbracket z_0, \ldots, z_{d} \rrbracket/(z_0^2 + \cdots + z_{d}^2 )$.
\end{thm}

We have the following more explicit description of the singularities appearing in \Cref{T:Main}, cf.\ also \Cref{S:AppendixDetailsOnSing}.

\begin{rem}\label{rem:explicit}%
An \emph{ADE-hypersurface singularity} %
is a singularity $S$ isomorphic to $\C\llbracket z_{0}, \ldots, z_{d}\rrbracket/(f),$ where $f$ is one of the following polynomials
\begin{itemize}
\item[$(A_{n})$] \quad  $z_0^{n+1}+ z_1^2 + z^2_{2} + \cdots + z_{d}^2 \quad  \,\, \, \, \, ( n \geq 1 )$,
\item[$(D_{n})$] \quad $z_0^2z_1 + z_1^{n-1} + z^2_{2} + \cdots + z_{d}^2 \quad ( n \geq 4 )$,
\item[$(E_{6})$] \quad $z_0^3 + z_1^4 + z^2_{2} + \cdots + z_{d}^2$,
\item[$(E_{7})$] \quad $z_0^3 + z_0z_1^3 + z^2_{2} + \cdots + z_{d}^2$,
\item[$(E_{8})$] \quad $z_0^3 + z_1^5 + z^2_{2} + \cdots + z_{d}^2$.
\end{itemize}
The $A_1$-singularities are also known as \emph{nodal singularities}.
\Cref{T:Main} applies to all isolated singularities $s \in \Sing(X)$ such that $\smash{\widehat{\co}_{X,s}}$ is an odd-dimensional ADE-hypersurface singularity and also to singularities whose singularity categories are equivalent to those of $3$-dimensional singularities $s$ where $\smash{\widehat{\co}_{X,s}}$ admits a small resolution of singularities. The latter $3$-dimensional singularities are a subclass of the \emph{compound Du Val-singularities (cDV-singularities)}, see Theorem \ref{Reid},
that is singularities of the form $\smash{\widehat{\co}_{X,s}} \cong  \C\llbracket z_{0}, \ldots, z_{3}\rrbracket/(f)$ where $f$ is a polynomial 
\begin{itemize}
\item[(cDV)] \quad $g + z_3 h$, \quad $(g \in \C\llbracket z_0, z_1, z_2 \rrbracket$ of type ADE and $h \in \C\llbracket z_0, z_1, z_2, z_3 \rrbracket)$.
\end{itemize}
\end{rem}

\begin{rem}
    For an irreducible Gorenstein variety $X$, such that its base locus (i.e.\ the base locus of the canonical bundle $\omega_X$) is empty or consists of a finite set of points, it is known that there are no non-trivial \SODs (see e.g.\ \cite[Corollary 3.3]{Spence} together with \cite[Corollary 6.6]{KS22}).
If the base locus is empty, we know that there are no tilting objects (see \cite[Example 4.6]{KPS19}). If the base locus is non-empty, this seems to be 
less clear. 
Note that in this case, our results still obstruct the existence of tilting objects.
\end{rem}

As a direct consequence of \Cref{T:Main}, we have the following more general statement, which implies \Cref{T:GoalNEW}, by \Cref{E:KSOD}.

\begin{cor}\label{cor:main}
   Let $X$ be a projective Gorenstein variety over $\C$ and let $\mathsf{S}$ be a closed subspace of $\Sing(X)$.
   Suppose $d = \dim(X)$ is odd and $\mathsf{S}$ contains an \emph{isolated} singularity $s$, which is an ADE-hypersurface singularity or satisfies $\smash{\widehat{\co}_{X,s} \in \SMAL}$.

   If there is an \absorption of the singularities of $\mathsf{S}$, then $s$ is a nodal singularity, i.e.\ an $A_1$-singularity $S \cong \C\llbracket z_0, \ldots, z_d \rrbracket/(z_0^2 + \cdots + z_d^2 )$.
\end{cor}

\subsection{Obstructions to singular equivalences}
The following result lies at the heart of our proof of  \Cref{T:Main}. It suggests that singular equivalences between odd-dimensional isolated hypersurface singularities and finite-dimensional algebras might only exist for nodal singularities.
  \begin{thm}\label{T:localmain}
  Let $S = \C\llbracket z_0, \ldots, z_d \rrbracket/(f)$ be a complete local $d$-dimensional ADE-hypersurface singularity or $S \in \SMAL$. 
  
  If $d$ is odd and there is a $\C$-linear triangle equivalence $\Dsg(S) \cong \Dsg(R)$
  for a finite-dimensional Gorenstein $\C$-algebra $R$, then $S$ is nodal, i.e.\ an $A_1$-singularity $S \cong \C\llbracket z_0, \ldots, z_d \rrbracket/(z_0^2 + \cdots + z_d^2 )$.
  \end{thm}

  We prove this theorem as part of the proof of \Cref{T:Main} in \Cref{SubS:ProofofTheoremMain} -- in particular, the part starting with equation \eqref{E:FirstReduction}.

\subsection{Notations and conventions}
All our commutative $\C$-algebras are of the form $\C\llbracket z_0, \ldots, z_n \rrbracket/I$. 
By a variety, we mean a separated noetherian reduced, but not necessarily irreducible, scheme of finite type over $\C$. 
In particular, the local completion of a $\mathbb{C}$-variety at a closed point is again a commutative $\mathbb{C}$-algebra of the form $\C\llbracket z_0, \ldots, z_n \rrbracket/I$, by Cohen's structure theorem.
We denote the Krull dimension of such a commutative $\C$-algebra $S$ by $\dim \, S$.
If the variety $X$ is clear from context we may write $\co_s$ instead of $\co_{X,s}$.
For an isolated singularity $s \in \Sing(X)$ we may write $\Dsg_s(X)$ instead of $\smash{\Dsg_{\{s\}}(X)}$.

We assume all subcategories to be closed under isomorphism, unless stated otherwise. We also assume all functors to be covariant unless stated otherwise.
The shift functor of all our triangulated categories is denoted by $\Sigma$, unless explicitly stated otherwise.
For an abelian category $\cA$ with enough projective (or injective) objects we define its \emph{global dimension} $\gldim \cA$ as the smallest $n \in \mathbb{N}$ such that $\Ext^{n+1}_{\cA}(-,-) \cong 0$ as a functor if it exists and otherwise $\infty$. We call $\gldim \Mod \Lambda$ the \emph{global dimension} of a ring $\Lambda$ and denote it by $\gldim \Lambda$.

\subsection{Roadmap}\label{subsec:roadmap}
The left side of \Cref{fig:roadmap}
\begin{figure}[h]\scalebox{.89}{\begin{tikzpicture}
\pgfdeclarelayer{b1}
\pgfdeclarelayer{b2}
\pgfdeclarelayer{b3}
\pgfdeclarelayer{b4}
\pgfsetlayers{b1, b2, b3, b4, background, main}

\tikzstyle{blue} = [
   rectangle, align=center, rounded corners=2pt, text width=16em, text opacity = 1, draw=blue!75!black, fill=blue!20
];
\tikzstyle{thm} = [
   rectangle, align=center, rounded corners=2pt, text width=10em, draw=blue!75!black, fill=blue!20, inner sep = .5em
]

\tikzstyle{thmenum} = [
    rectangle, align=center, rounded corners=2pt, text width=7em, draw=red!75!black, fill=red!20
]

\node[blue, text width= 16em] (SE) {\small Singular equivalences:\\$\Dsg(\widehat{\mathcal{O}}_{X,s})\cong\Dsg(R)$\\\tiny (with $R$ finite dimensional and Gorenstein)};
\node[blue, text width= 16em, above=4.5em of SE] (TE) {\small Tilting equivalences:\\$\Db(X) \cong \Db(R)$};
\node[text width= 16em, above=2em of TE, align=center] (KSOD) {\small \KSODs of $\Db(X)$:\\$\Db(X) = \langle \cp, \Db(R_1), \dots, \Db(R_n)\rangle$};
\begin{pgfonlayer}{b4}
    \node[blue, fit=(KSOD) (TE)] (KSODTE){};
\end{pgfonlayer}
\node[text width= 16em, above=3em of KSOD, align=center] (ACA) {\small Algebraic categorical absorptions of $\mathsf{S}$:\\$\cA = \langle \Db(R_1), \dots, \Db(R_n)\rangle \subset \Db(X)$};
\begin{pgfonlayer}{b3}
    \node[blue, fit=(ACA) (KSODTE)] (ACAKSODTE) {};
\end{pgfonlayer}

\node[above of=ACA, anchor=south] (SODsDesc) {\SODs (of subcategories) of $\Db(X)$:};
\begin{pgfonlayer}{b2}
    \node[fill=red!20, rounded corners=2pt, draw=red!75!black, fit=(SODsDesc) (ACAKSODTE)] (SODs)     {};
\end{pgfonlayer}

\node[right=10.5em of ACA, anchor=center] (RightOfACA) {};
\node[thm, above=.3em of RightOfACA] (corB) {\small \Cref{cor:main}};
\node[thm, below=.3em of RightOfACA] (thmB) {\small \Cref{T:Main}\\\tiny (Cf.\ also \Cref{fig:flowchart})};

\node[thm, right=10.5em of KSOD, anchor=center, yshift=-1em] (thmA) {\small \Cref{T:GoalNEW}};

\node[thm, right=10.5em of TE, anchor=center] (clm) {\small\Cref{clm:special}};

\node[right=10.5em of SE, anchor=center] (RightOfSE) {};

\node[thm, above=.5em of RightOfSE] (thmC) {\Cref{T:localmain}\\\tiny (Cf.\ also \Cref{fig:flowchart})};

\node[thm, below=.5em of RightOfSE] (thmD) {\Cref{P:ClusterTiltingObject}};

\node at ([yshift=.5em]KSODTE.north) {\tiny Special case, where $\mathsf{S} = \Sing(X)$ and $\cp \coloneqq {}^\perp \ca$:};

\node at ([yshift=.5em]TE.north) {\tiny Special case, where $n = 1$ and $\cp = 0$:} ;

\draw[->, shorten >=1pt, shorten <=1pt] (KSODTE.east) -| ++(+1em,-4em) |- node[above right, align=center, yshift=-2.3em] {\rotatebox{90}{\tiny For $n=1$ and $\Sing(X) = \{s\}$, cf.\ \Cref{P:RinKSOD}}}(SE.east);

\draw[->,  shorten >=1pt,shorten <=1pt] (thmA.south) -- (clm.north);
\draw[->,  shorten >=1pt,shorten <=1pt] (thmB.south) -- (thmA.north);
\draw[<->, shorten >=1pt,shorten <=1pt] (thmB.north) -- (corB.south);

\draw[->,  dashed, shorten >=1pt,shorten <=1pt] (thmC.west) -- ([xshift=-1em] thmC.west) -- node[left] {\rotatebox{90}{\tiny Generalisation: Using results listed in Sections \ref{subsec:prep}--\ref{subsec:AS} and \ref{subsec:homresults}}}([xshift=-1em]thmB.west) -- (thmB.west);

\draw[->, shorten >=1pt,shorten <=1pt] (thmD.north) -- node[right] {\tiny \Cref{T:PropertiesDsgR}}(thmC.south);

\draw [decorate,decoration={calligraphic brace,amplitude=5pt,raise=.5em}] (SODs.south west) -- (SODs.north west) node (desc1) [midway, xshift=-2.1em, rotate=90, text width = 12em, align=center] {\tiny Projective varieties $X$\\ and their derived categories};

\draw [decorate,decoration={calligraphic brace,amplitude=5pt,raise=.5em}] (SODs.west |- thmD.south) -- ($(SODs.west |- SODs.south west) - (0,0.5em)$) node [midway, xshift=-2.3em, rotate=90, text width = 12em, align=center] {\tiny Complete local singularities $\widehat{\co}_{X,s}$ and their singularity categories};

\node (A) at ([yshift=1.3em]SODs.north) [align =center] {\bf (Hierarchy of) structures};

\node at ($ (thmA |- A.north) $) [align =center, anchor=north] {\bf Our obstructions\\ \bf for non-nodal singularities};

\end{tikzpicture}}
\caption{Roadmap}
\label{fig:roadmap}
\end{figure}%
gives an overview how \absorptions, \KSODs and tilting and singular equivalences relate to each other.
On the right side is an overview of our main results, which show that for singularities of type ADE or $\SMAL$ the corresponding structure on the same height on the left can \emph{only} exist for nodal singularities, that is for type $A_1$.

\Cref{fig:flowchart} explains how the singularities appearing in \Cref{T:Main,T:localmain} relate%
\begin{figure}[h]\scalebox{.89}{\begin{tikzpicture}
\tikzstyle{blue} = [
   rectangle, align=center, rounded corners=2pt, text opacity = 1, draw=blue!75!black, fill=blue!20
];
\tikzstyle{thm} = [
   rectangle, align=center, rounded corners=2pt, draw=green!75!black, fill=green!40!black!20, inner sep = .5em
]

\tikzstyle{thmenum} = [
    rectangle, align=center, rounded corners=2pt, draw=red!75!black, fill=red!20, minimum height = 2.5em, minimum width = 4.1em
]
\tikzstyle{thmenumii} = [
    rectangle, align=center, rounded corners=2pt, draw=red!75!black, fill=red!20, minimum height = 5.9em, minimum width = 4.1em
]
\pgfdeclarelayer{b1}
\pgfdeclarelayer{b2}
\pgfdeclarelayer{b3}
\pgfdeclarelayer{b4}
\pgfsetlayers{b1, b2, b3, b4, background, main}

      \node[thm] at (0,0) (Aone) {$A_1$};
      \node[left = 1em of Aone] (Atwomplusone) {$A_{2m+1}$};
      \begin{pgfonlayer}{b4}
         \node[thmenum, fit = (Aone) (Atwomplusone)] (allAtwomplusone) {};
      \end{pgfonlayer}
      \node[thmenum, left = 1em of allAtwomplusone] (Dtwomplusone) {$D_{2m+1}$};
      \node[thmenum, left = 1em of Dtwomplusone] (Esix) {$E_{6}$};
      \node[thmenum, below = 1.3em of Atwomplusone] (Dtwom) {$D_{2m}$};
      \node[thmenum, left = 1em of Dtwom] (Eseven) {$E_7$};
      \node[thmenum, left = 1em of Eseven] (Eeight) {$E_8$};
      \node[thmenumii] at ($(Esix.center)!.5!(Eeight.center) - (5.3em,0)$) (Atwom) {$A_{2m}$};

      \draw[dashed] ([xshift=.5em, yshift=-.5em]Eseven.south east) -- ([xshift=.5em, yshift=-.5em]Dtwom.south east) -- ([xshift=.5em, yshift=+.3em]Dtwom.north east) -- ([xshift=.5em, yshift=-.5em]allAtwomplusone.south east) -- ([xshift=.5em, yshift=+.5em]allAtwomplusone.north east) -- ([xshift=-.5em, yshift=+.5em]Atwom.north west) -- ([xshift=-.5em, yshift=-.5em]Atwom.south west) -- ([xshift=.5em, yshift=-.5em]Atwom.south east) -- cycle;
      \draw[dashed]([xshift=.5em, yshift=.5em]Dtwomplusone.north east) -- ([xshift=.5em, yshift=-5em]Eseven.south east) -- ([xshift=.5em, yshift=-5em]Dtwom.south east) -- ([xshift=.5em, yshift=-.5em]Dtwom.south east);
      \draw[dashed]([xshift=.5em, yshift=.5em]Atwom.north east) -- ([xshift=.5em, yshift=-.5em]Atwom.south east);
      \begin{pgfonlayer}{b3}
         \fill[pattern=north west lines, pattern color=blue!20] ([xshift=.5em, yshift=-.5em]Eseven.south east) -- ([xshift=.5em, yshift=-.5em]Dtwom.south east) -- ([xshift=.5em, yshift=+.3em]Dtwom.north east) -- ([xshift=.5em, yshift=-.5em]allAtwomplusone.south east) -- ([xshift=.5em, yshift=+.5em]allAtwomplusone.north east) -- ([xshift=-.5em, yshift=+.5em]Atwom.north west) -- ([xshift=-.5em, yshift=-.5em]Atwom.south west) -- ([xshift=.5em, yshift=-.5em]Atwom.south east) -- cycle;
         \fill[pattern=north east lines, pattern color=blue!20] ([xshift=.5em, yshift=.5em]Dtwomplusone.north east) -- ([xshift=.5em, yshift=-5em]Eseven.south east) -- ([xshift=.5em, yshift=-5em]Dtwom.south east) -- ([xshift=.5em, yshift=+.3em]Dtwom.north east) -- ([xshift=.5em, yshift=-.5em]allAtwomplusone.south east) -- ([xshift=.5em, yshift=+.5em]allAtwomplusone.north east) -- cycle;
      \end{pgfonlayer}

      \draw [decorate,decoration={calligraphic brace,amplitude=5pt,raise=.5em}] ([xshift=-.4em, yshift=.5em]Atwom.north west) -- ([xshift=.4em, yshift=.5em]Atwom.north east); 
      \draw [decorate,decoration={calligraphic brace,amplitude=5pt,raise=.5em}] ([xshift=-.4em,yshift=.5em]Dtwom.west |- Esix.north) -- ([xshift=+.4em,yshift=.5em]Dtwom.east |- Esix.north); 
      \draw [decorate,decoration={calligraphic brace,amplitude=5pt,raise=.5em}] ([xshift=-.4em, yshift=.5em]Esix.north west) -- ([xshift=.4em, yshift=.5em]Dtwomplusone.north east); 
      \draw [decorate,decoration={calligraphic brace,amplitude=5pt,raise=.5em}] ([xshift=.5em, yshift=.4em]allAtwomplusone.north east) -- ([xshift=.5em, yshift=-.4em]Dtwom.south east -| allAtwomplusone.north east); 
      \draw [decorate,decoration={calligraphic brace,amplitude=5pt,raise=.5em}] ([xshift=.4em, yshift=-5em]allAtwomplusone.south east |- Dtwom.south) -- ([xshift=-.4em, yshift=-5em]allAtwomplusone.south west |- Dtwom.south); 

      \draw[->,  shorten >=1pt,shorten <=1pt] (Esix.east) -- (Dtwomplusone.west);
      \draw[->,  shorten >=1pt,shorten <=1pt] (Eseven.east) -- (Dtwom.west);
      \draw[->,  shorten >=1pt,shorten <=1pt] (Eeight.east) -- (Eseven.west);
      \draw[->,  shorten >=1pt,shorten <=1pt] (Dtwomplusone.east) -- (allAtwomplusone.west);

      \node[text width=5em, align=center, anchor= south] at ([yshift=1.35em]Atwom.north) {\tiny Loop(s) in the AR-quiver};
      \node[text width=5em, align=center, anchor= south] at ([yshift=1.35em]Dtwom |- Esix.north) {\tiny 2-Cycle(s) in the AR-quiver};
      \node[text width=12em, align=center, anchor= south] at ([yshift=1.35em, xshift=.5em]Esix.north east) {\tiny Reduction to the case(s) of \\ type $\SMAL$ using \Cref{L:ReduceType}};
      \node[text width=12em, align=center, anchor= north] at ([yshift=-5.85em] allAtwomplusone.center |- Dtwom.south) {\tiny Singularities of type $\SMAL$};
      \node[text width=12em, align=center, anchor= north, rotate=90] at ([xshift= 1.35em]allAtwomplusone.east |- Atwom.center) {\tiny Singularities of type ADE};
\end{tikzpicture}}
\caption{Singularities and obstructions in \Cref{T:Main,T:localmain}}
\label{fig:flowchart}
\end{figure}
and shows what the obstructions for \absorptions or singular equivalences to finite dimensional Gorenstein $\C$-algebras for each non-nodal singularity are. Recall that in dimension $3$ all singularities considered in \Cref{fig:flowchart} are also cDV-singularities, see \Cref{rem:explicit,Reid}. 

The Rest of this paper is organized as follows: in \Cref{sec:main} we first give an idea how to prove \Cref{clm:special} and \Cref{T:Main}.
We then give a summary of the results which are shown in \Cref{sec:prep,sec:homresults,sec:aussolred,sec:fromsmallres,sec:gor+sing} and give a proof our main result \Cref{T:Main} under assumption of these. In \Cref{sec:aussolred,sec:prep,sec:gor+sing} we show the results needed to reduce the proof of \Cref{T:Main} to a slightly stronger version of \Cref{T:localmain} which is implied by our main geometric result \Cref{P:ClusterTiltingObject} and our main homological result \Cref{T:PropertiesDsgR}.
\Cref{sec:homresults,sec:fromsmallres} are dedicated to proving the latter two results, respectively.

%

%
%
%
%
%
%
%
%
%

\section{Summary of the main tools and a proof of \Cref{T:Main}}\label{sec:main}
In this section, we prove our main Theorem \ref{T:Main}, after collecting the results that are needed in its proof. We refer to the later sections for details and proofs. 

First however, we give an overview by outlining the proof of a special case for tilting equivalences, cf.\ \Cref{clm:special}, and the general case of \Cref{T:Main}.

\subsection*{The proof in the special case of one singular point}
Let $X$ be a projective threefold with a unique singular point $s $ of type $ \SMAL^{(3)}$, that is $\Spec(\widehat{\mathcal{O}}_s)$ admits a small resolution of singularities. For example, $s$ could be an $A_{2m+1}$-hypersurface singularity. We first explain the steps of the proof of a special case of Theorem \ref{T:Main}.

\begin{Claim}\label{clm:special}
 If $\Db(X)$ admits a tilting object, then $s$ is nodal (i.e.\ an $A_1$-singularity).   
\end{Claim} 
\begin{proof}
A tilting object $\ct$ in 
$\Db(X)$ induces a triangle equivalence
\begin{align}\label{E:tiltingEq}
    \Db(X) \cong \Db(R),
\end{align}
where $R=\End_{\Db(X)}(\ct)$ is a finite-dimensional associative Gorenstein $\C$-algebra, cf.\ \Cref{P:KuznetsovShinder}.

The equivalence \eqref{E:tiltingEq} induces an equivalence of singularity categories
\begin{align}\label{E:IndSingEq}
    \Dsg(X) \cong \Dsg(R).
\end{align}
Since the latter category is idempotent complete \cite[Corollary 2.4]{XWChen}, there is a triangle equivalence by work of Orlov, cf.\ \Cref{prop:decomposition}
\begin{align}\label{E:SingEq2}
    \Dsg(X) \cong \Dsg(\widehat{\mathcal{O}}_s).
\end{align}
The combination of \eqref{E:IndSingEq} and \eqref{E:SingEq2} is Step \ref{I:2.1} in the general outline below.

By Buchweitz's \Cref{T:BuchweitzIntro}, there is a triangle equivalence
\begin{align*}
     \Dsg(R) \cong \ul{\GP}_{\proj R}(R):= \frac{\GP(R)}{\proj R},
\end{align*}
where the Frobenius exact category $\GP_{\proj R}(R)$ is a noncommutative generalization of maximal Cohen--Macaulay modules. 
This is Step \ref{I:2.2} below. Summing up, we have an equivalence of triangulated categories
\begin{align}\label{E:Summary}
     \Dsg(\widehat{\mathcal{O}}_s) \cong \ul{\GP}_{\proj R}(R).
\end{align}
Since $\Spec(\widehat{\mathcal{O}}_s)$ admits a small resolution of singularities, $\widehat{\mathcal{O}}_s$ admits a noncommutative crepant resolution in the sense of Van den Bergh, cf.\ \Cref{T:VdB2}. 
By works of Iyama--Reiten, cf.\ \Cref{T:Iyama}, this implies that the triangulated category $\Dsg(\smash{\widehat{\mathcal{O}}_s})$ admits a special generator, which is a \emph{cluster-tilting object}, see \Cref{defn:ctsmall}. 
The endomorphism algebra of cluster-tilting objects $T$ is finite-dimensional and can thus be described by a quiver $Q_T$ with relations. 
We show that if the quivers $Q_T$ of all cluster-tilting objects $T$ contain neither loops nor $2$-cycles, then $s$ is nodal, cf.\ \Cref{P:ClusterTiltingObject} and Step \ref{I:2.4}.

By work of Buan--Iyama--Reiten--Scott \cite{BIRSc}, the quiver of a cluster-tilting object in $\ul{\GP}_{\proj R}(R)$ never contains loops or $2$-cycles. This corresponds to Step \ref{I:2.5}, for which we need to work considerably harder in the general setting below.

The triangle equivalence \eqref{E:Summary} sends cluster-tilting objects to cluster-tilting objects. In combination with the discussion above, this shows that $s$ is indeed nodal and finishes the proof of the claimed special case of Theorem \ref{T:Main}, cf.\ Step \ref{I:2.6}.
\end{proof}

\subsection*{Outline of the proof in general}
Here is a brief outline of the subsections below, which provide a guideline for the proof of Theorem \ref{T:Main}.

\begin{enumerate}[label={(2.\theenumi)},leftmargin=*]
   \item \label{I:2.1}We show that there exists $1 \leq i \leq n$ and a triangulated category $\cc$ such that 
   \begin{align} \label{E:DirSummand}
       \Dsg(\widehat{\co}_s) \oplus \cc = \Dsg(R_i).
   \end{align}
 We will use the equivalence \eqref{E:DirSummand} to produce a new\footnote{We can skip Step \ref{I:2.3}\ref{item:2.3a} if $\cc = 0$ and Step \ref{I:2.3}\ref{item:2.3b} if $s$ is of type $A_{2m}$ or $\widehat{\co}_s \in \SMAL$. In particular, if both $\cc = 0$ and $s$ is of type $A_{2m}$ or $\widehat{\co}_s$ in $\SMAL$, the equivalence  \eqref{E:desiredEq} is just \eqref{E:DirSummand}, with $S':=\widehat{\co}_s$.} equivalence %
    \begin{align}\label{E:desiredEq}
       \Dsg(S') \cong \ul{\GP}_{{\cm}}(R_i).
    \end{align}
    in Steps \ref{I:2.2} -- \ref{I:2.3}. This will be our analogue of the key equivalence \eqref{E:Summary} in the proof of the special case above -- here, $\GP_{{\cm}}(R_i)$ is a new $\Hom$-finite Frobenius exact category with a larger class of projective-injective objects ${\cm}$ and 
    $\Spec(S')$ admits a small resolution of singularities or is of type  $A_{2m}$. %
   \item \label{I:2.2}We show that $R_i$ is a Gorenstein algebra and thus there is a $\Hom$-finite Frobenius exact category $\GP_{\proj R_i}(R_i)$ with stable category $\Dsg(R_i) \cong \underline{\GP}_{\proj R}(R_i)$
   \item \label{I:2.3}
   \begin{enumerate}[label={(\alph*)},leftmargin=*]
   \item\label{item:2.3a} In order to remove the summand $\cc$ in \eqref{E:DirSummand}, we change the exact structure on the additive category
   $\GP_{\proj R_i}(R_i)$ to obtain a new Frobenius exact category $\GP_{\cp}(R_i)$ with larger class of projective-injective objects ${\cp}$  and stable category 
   \begin{equation} 
        \smash{\Dsg(\widehat{\co}_s)} \cong \underline{\GP}_{\cp}(R_i) := \frac{\GP(R_i)}{\cp} \label{E:SingToStabFrob}
   \end{equation}
   using Auslander--Solberg reduction, cf.\ \Cref{subsec:AS}.
   \item\label{item:2.3b} If $\widehat{\co}_{s} \in \SMAL$ or $s$ is of type $A_{2m}$ (with $m >0$), then \eqref{E:SingToStabFrob} is the desired equivalence \eqref{E:desiredEq} with $S' \coloneqq \smash{\widehat{\co}_s}$ and ${\cm} \coloneqq \cp$.
   Otherwise, by our assumptions, $s$ is an ADE-hypersurface singularity of type $D_{2m+1}$ with $m >1$ or $E_m$ with $m \in \{6,7,8\}$.
   To obtain \eqref{E:desiredEq} from \eqref{E:SingToStabFrob}, we apply Auslander--Solberg reduction on both sides of \eqref{E:SingToStabFrob}: on the left side we change $\Dsg(\smash{\widehat{\co}_s})$ into $\Dsg(S')$ with $S' \in \SMAL$ a singularity of type $A_3$ or $D_4$ and on the right side we change $\underline{\GP}_{\cp}(R_i)$ into $\underline{\GP}_{{\cm}}(R_i)$ with $\cp \subsetneq {\cm}$. 
   \end{enumerate}
   \item \label{I:2.4} \label{I:Sing3fold} We show that for a 
   threefold singularity $S' $ of type $ \SMAL^{(3)}$, which is \emph{not} an $A_{1}$-singularity, $\Dsg(S')$ has a cluster--tilting object $T$ with loops or $2$-cycles, in its quiver.
   \item\label{I:2.5}\label{I:NoLoopsand2cycles} We show that the triangulated category $\underline{\GP}_{{\cm}}(R_i)$ in \eqref{E:desiredEq} can neither
   \begin{enumerate}[label={(\alph*)},leftmargin=*]
        \item have loops in its AR-quiver, nor \label{item:loops}
        \item contain a cluster-tilting object $T$ with loops or $2$-cycles in its quiver.\label{item:2cycles}
   \end{enumerate}
   \item \label{I:2.6} In view of the triangle equivalence \eqref{E:desiredEq}, the existence of loops in the AR-quiver \eqref{E:A2n} of $\Dsg(A_{2m})$ and \ref{I:NoLoopsand2cycles}\ref{item:loops} imply \Cref{T:Main} for singularities of type $A_{2m}$ for $m \geq 1$.
   Items \ref{I:Sing3fold} and \ref{I:NoLoopsand2cycles}\ref{item:2cycles} imply Theorem \ref{T:Main} for threefold singularities that are not of type $A_{2m}$. The general case reduces to the case of threefolds, since the singularity categories of all singularities that we consider are equivalent to singularity categories of these threefold singularities (for ADE-hypersurface singularities this follows from Knörrer's periodicity \cite{Knoerrer} and for $\SMAL$ this holds by \Cref{D:frakS}).
   \end{enumerate}

\subsection{Main tools}\label{sec:auxiliary}  We continue with the collection of the results which are needed in the proof of \Cref{T:Main}. These are organized by the remaining sections of this paper, where their proofs, more general versions and additional information can be found.

\subsubsection{\nameref{sec:prep} (See \Cref{sec:prep})}\label{subsec:prep}

The following result reduces an \absorption of an isolated singularity with connected singularity category to one component.
\begin{prop}[{See \Cref{prop:reduceabsorption2}}]\label{prop:reduceabsorption}
   Let $X$ be a projective Gorenstein variety. 
   Suppose $s \in \Sing(X)$ is an isolated singularity such that $\Dsg(\smash{\widehat{\co}_s})$ is connected. 
   If 
   \[\ca = \langle \Db(R_1), \dots, \Db(R_n) \rangle \subseteq \Db(X)\] 
   is an \absorption of $s$ then $\Dsg(\smash{\widehat{\co}_s}) \subset \Dsg(R_i)$ is a triangulated direct summand for some $1 \leq i \leq n$.
   In particular, $\Db(R_i)$ is already an \absorption of $s$.
\end{prop}
By the following and \Cref{S:AppendixDetailsOnSing} the requirements of \Cref{prop:reduceabsorption} are satisfied for isolated singularities of type $\SMAL$ and type ADE.

\begin{prop}[See \Cref{P:Takahashi}]\label{Pi:Takahashi}
Let $S=\C\llbracket z_0, \ldots, z_d\rrbracket/(f)$ be an isolated singularity. Then the singularity category $\Dsg(S)$ is connected.
\end{prop}
We also need the following recognition result for nodal singularities.

\begin{lem}[See \Cref{L:Yoshino}]\label{L:nodal}
Let $S$ be a complete local Gorenstein $\C$-algebra of Krull 
dimension $d$. If $\Dsg(S) \cong  \Dsg(\C\llbracket z_0,\dots,z_3 \rrbracket/(z_0 z_1-z_2 z_3))$, then $d$ is odd and the singularity $S \cong \C\llbracket z_0, \dots, z_d \rrbracket/(z_0^2 + \ldots + z_d^2)$ is nodal.
\end{lem}

\subsubsection{\nameref{sec:gor+sing} (See \Cref{sec:gor+sing})}
In this section we consider the following important class of (noncommutative) rings.

\begin{defn}\label{defn:Goring}
A two-sided Noetherian ring $\Lambda$ satisfying $\injdim _{\Lambda}\!\Lambda < \infty$ and 
$\injdim \Lambda_{\Lambda} < \infty$ is called \emph{Gorenstein ring}\footnote{In the noncommutative setting these rings are also often called \emph{Iwanaga--Gorenstein rings}.}. 
\end{defn}

We use the following generalization of \cite[Theorem 4.4]{KPS19} which follows from \cite{KS22}.

\begin{prop}[{See \Cref{P:KuznetsovShinder}}]\label{P:RinKSOD}
   Let $X$ be a projective Gorenstein variety, %
   $\ca  \subseteq \Db(X)$ be an admissible triangulated subcategory and $\ca = \langle \Db(R_1), \dots, \Db(R_n) \rangle$ be an admissible \SOD for some $\C$-algebras $R_1, \dots, R_n$.
   
   Then $R_1, \dots, R_n$ are finite-dimensional Gorenstein $\C$-algebras.
\end{prop}

The singularity category of Gorenstein rings has a nice Frobenius model by the following definition and the following theorem.
\begin{defn}\label{D:GProj}
Let $\Lambda$ be an Gorenstein ring. We define the subcategory
\[
\GP(\Lambda):=\{X\in\mod\Lambda\mid \Ext^i_\Lambda(X,\Lambda)=0\mbox{ for any }i>0\}
\]
of $\mod\Lambda$. This category is called the category of \emph{Gorenstein projective} $\Lambda$-modules\footnote{If $\Lambda$ is a commutative local Gorenstein ring, then the Gorenstein projective $\Lambda$-modules are precisely the maximal Cohen--Macaulay $\Lambda$-modules.}. 
The \emph{stable category} $\underline{\GP}_{\proj \Lambda}(\Lambda)$ of $\GP_{\proj \Lambda}(\Lambda)$ is defined as the additive quotient
\[
   \underline{\GP}_{\proj \Lambda}(\Lambda):=\frac{\GP(\Lambda)}{\proj \Lambda},
\]
where $\proj \Lambda$ is the category of finitely generated projective $\Lambda$-modules, see also \Cref{Not:frobenius}. 
The stable category $\underline{\GP}_{\proj \Lambda}(\Lambda)$ naturally carries a triangulated structure by work of Happel, see \cite{Happel} and \Cref{Not:frobenius}.
\end{defn}

The next result is due to Buchweitz \cite{{Buchweitz87}}.
\begin{thm}[{See \cite[Theorem 4.4.1]{Buchweitz87}}]\label{T:BuchweitzIntro}
Let $\Lambda$ be a Gorenstein ring. 

Then there is a triangle equivalence $ \underline{\GP}_{\proj \Lambda}(\Lambda) \to \Dsg(\Lambda)$ which is induced by the canonical inclusion $\GP(\Lambda)\subseteq \mod(\Lambda) \rightarrow \Db(\Lambda)$.
\end{thm}

The following statement is well-known.

\begin{prop}[{See \Cref{cor:Serre}}]\label{P:SerreFunctor}
    Let $R$ be a finite-dimensional Gorenstein $\C$-algebra and $\GP_\cp(R)$ be a Frobenius exact structure on $\GP(R)$.
    
    Then the triangulated category $\underline{\GP}_{\cp}(R)$ has a Serre functor.
\end{prop}

\subsubsection{\nameref{sec:aussolred} (See \Cref{sec:aussolred})}\label{subsec:AS}
The following reduction results follow from their corresponding results in \Cref{sec:aussolred} using \Cref{P:SerreFunctor}

\begin{lem}[See \Cref{C:RemoveFiniteTypeComponents}]\label{L:AuslanderSolbergSummand}
   Let $R$ be a finite-dimensional Gorenstein $\C$-algebra and $\GP_{\cp}(R)$ be a Frobenius exact structure on $\GP(R)$. 
   Assume $\ul{\GP}_\cp(R) = \cc_1 \oplus \cc_2$ is a direct sum of triangulated subcategories.

   Then there is a Frobenius exact structure $\GP_\cm(R)$ on $\GP(R)$ such that there is a triangle equivalence $\ul{\GP}_{\cm}(R) \cong \cc_1$. 
\end{lem}

\begin{lem}[See \Cref{C:ReduceType}] \label{L:ReduceType}
Let $d$ be odd and $S$ be a complete local $d$-dimensional ADE-hypersurface singularity such that there is a finite-dimensional Gorenstein $\C$-algebra $R$ and a Frobenius exact structure $\GP_\cp(R)$ which satisfies $\underline{\GP}_{\cp}(R) \cong \Dsg(S)$.

Then there is a Frobenius exact structure $\GP_\cm(R)$ on $\GP(R)$ such that there is a triangle equivalence $\underline{\GP}_{\cm}(R) \cong \Dsg(S')$, where $S'$ is a complete local $3$-dimensional ADE-hypersurface singularity
\begin{enumerate}[label={(\alph*)}]
\item of type $A_3$, if $S$ is of type $D_{2m+1}$ (with $m >1$) or of type $E_6$,
\item of type $D_4$, if $S$ is of type $E_7$ of $E_8$.
\end{enumerate}
In particular, $\Dsg(S')$ is $2$-Calabi--Yau and $\Spec S'$ has a small resolution.
\end{lem}

\subsubsection{\nameref{sec:fromsmallres} (See \Cref{sec:fromsmallres})}
Recall the notion of cluster-tilting objects, which we will use in \Cref{sec:fromsmallres,sec:homresults}.
\begin{defn}[{Cf.\ \Cref{defn:ct}}]\label{defn:ctsmall}
   Let $\cc$ be a $2$-Calabi--Yau triangulated category.
   We say $T \in \cc$ is a \emph{cluster-tilting object} if
\begin{enumerate}[label={{(\alph*)}}]
   \item $\cc(T, \Sigma T)=0$, and
   \item if $X$ is an object such that $\cc(T, \Sigma X)=0$ then $X \in \add T$.
\end{enumerate}
\end{defn}

The following result is the key ingredient on the geometric side. 
It follows from works of Iyama \cite{Iyama07, Iyama07a}, Iyama--Reiten \cite{IyamaReiten}, Keller--Reiten \cite{KellerReiten08} and Van den Bergh \cite{VandenBergh04, NCCR}.

\begin{thm}[{See \Cref{P:ClusterTiltingObject2}}]\label{P:ClusterTiltingObject}
   Let $S \cong \C\llbracket z_1, \ldots, z_n \rrbracket/I$ be a $3$-dimensional Gorenstein $\C$-algebra with an isolated singularity and assume that $\Spec S$ has a small resolution.\nopagebreak

   Then $\Dsg(S)$ contains a cluster-tilting object. Moreover, if the quiver of every cluster-tilting object in $\Dsg(S)$ contains no loops 
   or $2$-cycles, then $S$ is an $A_1$-singularity.
\end{thm}

\subsubsection{\nameref{sec:homresults} (See \Cref{sec:homresults})}\label{subsec:homresults}

The following result is our key ingredient on the representation theoretic side.
To show part \ref{item:noloops} we first prove a categorical generalization of the `No-loop Theorem' for finite-dimensional algebras of finite global dimension, see \cite{Igusa, IgusaLiuPaquette}. 
We use this and a generalization of results from \cite{BIRSc} to deduce part \ref{item:noloops2cycles}.

\begin{thm}[{See \Cref{T:ARnoLoops} and \Cref{C:Noloops2cycles}}] \label{T:PropertiesDsgR}
   Let $R$ be a finite-dimen\-sional Gorenstein $\C$-algebra and $\GP_\cp(R)$ be a Frobenius exact structure on $\GP(R)$. 
\begin{enumerate}[label={{(\alph*)}}]
   \item Then there are no loops in the Auslander-Reiten quiver of $\underline{\GP}_{\cp}(R)$.\label{item:noloops}
   \item If $\underline{\GP}_{\cp}(R)$ is a $2$-Calabi--Yau category then the quiver of any cluster-tilting object
   $T \in \underline{\GP}_{\cp}(R)$ has neither loops nor $2$-cycles.\label{item:noloops2cycles}
\end{enumerate}
\end{thm}

\subsection{Proof of \Cref{T:Main}} \label{SubS:ProofofTheoremMain}
   Let $\ca = \langle \Db(R_1), \dots, \Db(R_n) \rangle$ be an \absorption of 
   $s$.
   Since $s$ is a hypersurface singularity, $\smash{\Dsg(\widehat{\co}_s)}$ is connected, see \Cref{Pi:Takahashi}.
   By \Cref{prop:reduceabsorption} there are $1 \leq i \leq n$ and a triangulated subcategory $\smash{\cc \subseteq \overline{\Dsg(X)}}$ such that $\smash{\Dsg(\widehat{\co}_s) \oplus \cc = \Dsg(R_i)}$ as triangulated categories.
   By \Cref{T:BuchweitzIntro} we have $\Dsg(R_i) \cong \underline{\GP}_{\proj R_i}(R_i)$.
   By \Cref{L:AuslanderSolbergSummand} there is a Frobenius exact structure $\GP_{\cp}(R_i)$ on $\GP(R_i)$ such that 
   \begin{align}\label{E:FirstReduction}
   \underline{\GP}_{\cp}(R_i) \cong \Dsg(\widehat{\co}_s).
   \end{align}
  
We distinguish the following cases:
 \begin{enumerate}[label={\textup{(\alph*)}}]
 \item $s $ is of type $ \SMAL$, \label{I:smallres}
 \item $s$ is an odd-dimensional ADE-hypersurface singularity and $\widehat{\co}_s \notin \SMAL$. Then $s$ is of one of the following types (cf.\ e.g.\ \cite{BIKR}):
    \begin{enumerate}[label={\textup{(\arabic*)}}]
       \item $A_{2m}$,  \label{I:A2n}
       \item  $D_{2m+1}$ or $E_m$. \label{I:D2m+1En} 
    \end{enumerate}
 \end{enumerate} 
We begin with the case \ref{I:smallres}. Combining \eqref{E:FirstReduction} with the definition of $\SMAL$, there is a complete local Gorenstein threefold singularity $\Spec(S)$ admitting a small resolution of singularities such that
\begin{align}\label{E:SingEqCommNC}
\Dsg(S) \cong \Dsg(\widehat{\co}_s) \cong \underline{\GP}_{\cp}(R).
\end{align}  
As $R$ is a finite-dimensional Gorenstein algebra, there are neither loops nor $2$-cycles in the quivers of cluster-tilting objects in $\underline{\GP}_{\cp}(R)$, by \Cref{T:PropertiesDsgR}\ref{item:noloops2cycles}. 
Hence, \eqref{E:SingEqCommNC} and \Cref{P:ClusterTiltingObject} show that $S$ is nodal.
Therefore, $\smash{\widehat{\co}_s}$ is nodal by \Cref{L:nodal}.

We continue with the case \ref{I:A2n}, i.e.\ $s$ is of type $A_{2m}$.
The AR-quiver of $\Dsg(\smash{\widehat{\co}_s})$ is given by 
\[ \label{E:A2n}
     \includegraphics{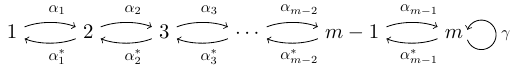}
\]
using Kn{\"o}rrer periodicity, see \cite{Knoerrer} and \cite{DieterichWiedemann86}.
By \Cref{T:PropertiesDsgR}\ref{item:noloops}, the AR-quiver of $\underline{\GP}_{\cp}(R_i) = \Dsg(\smash{\widehat{\co}_s})$ contains no loops, contradicting the existence of $\gamma$. 
This shows that $s$ cannot be of type $A_{2m}$ and hence case \ref{I:A2n} cannot occur.

Finally, we can apply \Cref{L:ReduceType} to reduce the case \ref{I:D2m+1En} to the case \ref{I:smallres}.
Indeed, \Cref{L:ReduceType} yields a new Frobenius exact structure $\GP_{\cm}(R_i)$ on $\GP(R_i)$ such that 
$\underline{\GP}_{\cm}(R_i) \cong \Dsg(S')$, where $S'$ is a ADE-threefold singularity of type $A_3$ or $D_4$. 
But following the argument in case \ref{I:smallres} shows that $S'$ is of type $A_1$. So case \ref{I:D2m+1En} cannot occur.
Under assumption of the auxiliary statements from \Cref{sec:auxiliary}, which are shown in the rest of this paper, the proof is finished.
\hfill \qedsymbol

\section{Preparations}\label{sec:prep}
We recommend recalling the notion of \SODs, cf.\ e.g.\ \cite{Bondal-Kapranov}, and \Cref{D:CatAbs,defn:homfiniteobj} before reading this section.

\subsection{Additional information regarding \Cref{D:CatAbs}}
In this section we provide additional information about the category $\ca^{\sg}$ from \Cref{D:CatAbs}.

First, recall that any admissible subcategory of $\Db(X)$ for a projective variety $X$ is $\infty$-admissible in the sense of \cite[Defintion 1.8]{Bondal-Kapranov}.
\begin{lem}\label{lem:inftyadmissible}
    Let $X$ be a projective variety and $\ca \subset \Db(X)$ be an admissible triangulated subcategory. 
    Then $\langle \ca, {}^\perp \ca \rangle$ and $\langle {}^\perp \ca, \ca \rangle$ are admissible \SODs.
\end{lem}
\begin{proof}
    It follows from \cite[Proposition 1.5]{Bondal-Kapranov} that $\langle \ca, {}^\perp \ca \rangle$ and $\langle {}^\perp \ca, \ca \rangle$ are \SODs.
    As a result of $\Db(X)$ being saturated by work of Neeman, cf.\ \cite[Theorem 2.7]{KPS19}, also ${}^\perp \ca$ and $\ca^{\perp}$ are admissible, using \cite[Corollary 2.9.2]{Bondal-Kapranov}.
\end{proof}

The following specializes to \cite[Theorem 4.4(3)]{KPS19} and shows existence of the induced embedding $\ca/\ca^{\hf} \to \Dsg(X)$ used in \Cref{D:CatAbs}. 

\begin{lem}\label{lem/def:asg}
    Let $X$ be a projective variety and let $\ca \subset \Db(X)$ be an admissible subcategory. Assume that
$\ca=\langle \Db(R_1), \dots, \Db(R_n) \rangle$ is an admissible \SOD, where the $R_i$ are 
$\C$-algebras.
    Then there is an induced embedding $\ca/\ca^{\hf} \to \Dsg(X)$ and $\ca^{\sg} = \langle \Dsg(R_1), \dots, \Dsg(R_n) \rangle$ is an admissible \SOD of its essential image $\ca^{\sg}$.
\end{lem}
\begin{proof}
    As $\Db(X)$ is $\Hom$-finite, $R_i$ is finite-dimensional over $\C$ for $1 \leq i \leq n$.
    Hence, $\Db(R_i)^{\hf} \cong \Kb(R_i)$, by \cite[Propition 2.18]{iyama-wemyss}, and $\Db(R_i)/\Kb(R_i) \cong \Dsg(R_i)$.
    Similarly, $\Db(X)/\Db(X)^{\hf} \cong \Dsg(X)$, by \cite[Proposition 1.11]{orl6}.
    By \Cref{lem:inftyadmissible} we have admissible \SODs
    \[ \Db(X) = \langle \ca, {}^{\perp} \ca \rangle \text{ and } \ca = \langle \Db(R_1), \dots, \Db(R_n) \rangle. \]
    Now, \cite[Proposition 1.10]{orl6} yields \SODs of the form $\Dsg(X) = \langle \ca/\ca^{\hf}, ({}^{\perp}\ca)/({}^{\perp}\ca)^{\hf} \rangle$ and $\ca/\ca^{\hf} = \langle \Dsg(R_1), \dots, \Dsg(R_n) \rangle$, where the former \SOD induces the claimed embedding $\ca/\ca^{\hf} \to \Dsg(X)$.
\end{proof}

\subsection{Local consequences of categorical absorptions}
Let $X$ be a projective Gorenstein variety and $s \in \Sing(X)$ be an isolated singularity.
In this section, we show that if $\ca = \langle \Db(R_1), \dots, \Db(R_n) \rangle \subset \Db(X)$ is an \absorption of $s$ then $\Dsg(\smash{\widehat{\co}_s})$ is a direct summand of $\Dsg(R_i)$ for some $1 \leq i \leq n$, provided $\Dsg(\smash{\widehat{\co}_s})$ is a connected triangulated category.
This and \Cref{P:RinKSOD} not only reduce the proof of \Cref{T:Main} to the easier case of \Cref{T:localmain} but also allows us to use \Cref{L:AuslanderSolbergSummand} to obtain a nice Frobenius model for $\Dsg(\smash{\widehat{\co}_s})$. 
We will achieve our goal as a consequence of the following general lemma.

\begin{lem}\label{lem:SODintersection}
   Suppose $\ca_1, \dots, \ca_n \subset \cc$ are idempotent complete, triangulated subcategories such that $\cc_1 \subseteq \langle \ca_1, \dots, \ca_n \rangle \subset \cc_1 \oplus \cc_2$ for some triangulated subcategories $\cc_1, \cc_2 \subset \cc$. Let $\ca_{i,j} \coloneqq \cc_j \cap \ca_i$ for $1 \leq i \leq n$ and $j = 1,2$.
   \begin{enumerate}[label={(\alph*)}]
       \item We have $\ca_i = \ca_{i,1} \oplus \ca_{i,2}$ as triangulated categories for $1 \leq i \leq n$.\label{item:SODintersectionSum}
       \item Furthermore, we have $\cc_1 = \langle  \ca_{1,1}, \dots, \ca_{n,1} \rangle$.\label{item:SODintersection}%
       \item If $\cc_1$ is $\kk$-linear over a field $\kk$, connected and $d$-Calabi--Yau for some $d \in \N$   then $\cc_1 \subseteq \ca_i$ is a triangulated direct summand for some $1 \leq i \leq n$.\label{item:SODintersectionCY}%
   \end{enumerate}
\end{lem}
\begin{proof}

   \ref{item:SODintersectionSum}: It is clear that the intersection of triangulated subcategories is again a triangulated subcategory and that $\ca_{i,1}$ and $\ca_{i,2}$ are orthogonal for each $1 \leq i \leq n$.
   As $\ca_i$ is idempotent complete, it is closed under summands in $\cc$.
   In particular, if $X = X_1 \oplus X_2$ is a decomposition of $X \in \ca_i$ for some $1 \leq i \leq n$ with $X_j \in \cc_j$ then $X_j \in \ca_{i,j}$.
   As $\ca_i \subset \cc_1 \oplus \cc_2$ this shows $\ca_i = \ca_{i,1} \oplus \ca_{i,2}$ as triangulated categories.

   \ref{item:SODintersection}: We have $\ca \coloneqq \langle \ca_1, \dots, \ca_n \rangle = \langle \ca_{1,1}, \ca_{1,2}, \dots, \ca_{n,1}, \ca_{n,2}, \rangle$ as we have in particular $\ca_i = \langle \ca_{i,1}, \ca_{i,2} \rangle $ for $1 \leq i \leq n$ by \ref{item:SODintersectionSum}.
   Rearranging
   the terms shows that $\ca = \langle \langle \ca_{1,1}, \dots, \ca_{n,1} \rangle, \langle \ca_{1,2}, \dots, \ca_{n,2} \rangle \rangle$ and this decomposition is actually orthogonal as any subcategory of $\cc_1$ is orthogonal to any subcategory of $\cc_2$.
   Note that rearranging the SOD this way does not require that the components $\ca_{i,j}$ are admissible since $\ca_{i,1}$ and $\ca_{j,2}$ are orthogonal to each other for $1 \leq i,j \leq n$.
   It follows from $\cc_1 \subset \ca$ that $\cc_1 \subset \langle \ca_{1,1}, \dots, \ca_{1,n} \rangle$.
   The other inclusion is trivial as $\cc_1 \subset \cc$ is triangulated.

    \ref{item:SODintersectionCY}: We have $0 = \Hom_{\cc_1}(\Sigma^{-d} Y, X)^\ast \cong \Hom_\kk(\Hom_{\cc_1}(X, Y), \kk)$ for any $X \in \ca_{i,1} $ and any $Y \in \ca_{j,1}$ with $1 \leq i < j \leq n$ as $\cc_1$ is $d$-Calabi--Yau and because the decomposition $\cc_1 = \langle  \ca_{1,1}, \dots, \ca_{n,1} \rangle$ is semiorthogonal.
    This shows that the given decomposition is orthogonal, i.e.\ $\cc_1 = \ca_{1,1} \oplus \dots \oplus \ca_{n,1}$, and hence $\cc_1 = \ca_{i,1} \subset \ca_i$ for some $1 \leq i \leq n$ as $\cc_1$ is connected.
    Thus, we have $\cc_1 \oplus \ca_{i,2} = \ca_i$ by \ref{item:SODintersectionSum}.
\end{proof}

\begin{ex}
   \Cref{lem:SODintersection}\ref{item:SODintersection} is not true in general if there is no $\cc_2 \subset \cc$ with $ \langle \ca_1, \ca_2 \rangle \subset \cc_1 \oplus \cc_2$.
   For example let $\cc = \Db(\mathbb{P}^1) = \langle \ca_1, \ca_2 \rangle$, where $\ca_1 = \langle \co \rangle$, $\ca_2 = \langle \co(1) \rangle$ and $\cc_1 = \Db(\tors\mathbb{P}^1)$. 
   Then $\cc_1 \cap \ca_1 = 0 = \cc_1 \cap \ca_2$.
\end{ex}

In order to apply \Cref{lem:SODintersection}\ref{item:SODintersectionCY} we need the following proposition.
Notice, as $s$ is an isolated singularity, $\{s\} \subset \Sing(X)$ is a connected component.

\begin{prop}\label{prop:decomposition}
   Let $X$ be a quasi-projective variety.
   \begin{enumerate}[label={(\alph*)}]
        \item If $\Sing(X) = \bigsqcup_{i = 1}^n C_i$ is a decomposition into connected components then we have an induced decomposition $ \smash{\overline{\Dsg(X)}} = \bigoplus_{i=1}^n \overline{\Dsg_{C_i}(X)}$.\label{item:decomposition}%
        \item If $s \in \Sing X$ is an isolated singularity then $\overline{\Dsg_{s}(X)} = \Dsg(\widehat{\co}_s)$.\label{item:equal}%
   \end{enumerate}
\end{prop}
\begin{proof}
   \ref{item:decomposition}: By \cite[Proposition 2.7]{Orlov11} we have $\smash{\overline{\Dsg(X)}  = \overline{\Dsg_{\Sing(X)}(X)}}$.
   Now, both the bounded derived category of coherent sheaves and the category of perfect complexes supported at $\Sing(X)$ decompose into orthogonal sums according to the decomposition of the support $\Sing(X) = \bigsqcup_{i=1}^n C_i$.
   Hence, the same holds for $\smash{\Dsg_{\Sing(X)}(X)}$, and therefore also for its idempotent completion.

   \ref{item:equal}: This follows from \mbox{\cite[Proposition 3.3]{Orlov11}} applied to the closed subspace $Z = \{s\}$ of $X$ and the morphism $\Spec(\widehat{\co}_s) \to X$, using the notation of loc.\ cit.\ together with the fact that $\Dsg( \widehat{\co}_s )$ is idempotent complete (Henselian local rings have vanishing $\mathrm{K}_{-1}$ by \cite[Theorem 3.7]{Drinfeld}, which, together with e.g.\ \cite[Lemma 1.11]{pavicshinder}, implies that their singularity category is idempotent complete).
\end{proof}

By means of \Cref{prop:decomposition}\ref{item:equal} we may identify $\Dsg(\widehat{\co}_s)$ with $\overline{\Dsg_{s}(X)}$ from now on.
We can now proof the main result of this section.

\begin{prop}\label{prop:reduceabsorption2}
   Let $X$ be a projective Gorenstein variety. 
   Suppose $s \in \Sing(X)$ is an isolated singularity such that $\Dsg(\smash{\widehat{\co}_s})$ is connected. 
   If 
   \[\ca = \langle \Db(R_1), \dots, \Db(R_n) \rangle \subseteq \Db(X)\] 
   is an \absorption of $s$ then $\Dsg(\smash{\widehat{\co}_s}) \subset \Dsg(R_i)$ is a triangulated direct summand for some $1 \leq i \leq n$.
   In particular, $\Db(R_i)$ is already an \absorption of $s$.
\end{prop}
\begin{proof}\label{proof:prop2.1}
   We have $\ca^{\sg} = \langle \Dsg(R_1), \dots, \Dsg(R_n) \rangle \subseteq \Dsg(X) \subseteq \overline{\Dsg(X)}$ by \Cref{lem/def:asg}. 
   As $\ca$ is an \absorption of $s$ and by \Cref{prop:decomposition} we have 
   \begin{equation} \Dsg_s(X) \subset \ca^{\sg} \subset \overline{\Dsg(X)} = \Dsg(\widehat{\co}_s) \oplus \overline{\Dsg_{\Sing(X) \setminus \{s\}}(X)}.\label{eq:inclusions}\end{equation} 
   Moreover, $\Dsg(R_i)$ is idempotent complete for $1 \leq i \leq n$ by \cite[Corollary 2.4]{XWChen}. 
   In particular, $\ca^{\sg}$ is idempotent complete, cf.\ \cite[Lemma 2.2]{KPS19}.
   Hence, passing to the idempotent completion on the left side in \eqref{eq:inclusions} and using \Cref{prop:decomposition}\ref{item:equal} yield
   \[ \Dsg(\widehat{\co}_s) \subset \langle \Dsg(R_1), \dots, \Dsg(R_n) \rangle  \subset \Dsg(\widehat{\co}_s) \oplus \overline{\Dsg_{\Sing(X) \setminus \{s\}}(X)}.\]
   By work of Auslander and Buchweitz, we know that $\Dsg(\smash{\widehat{\co}_s})$ is $d$-Calabi--Yau, see e.g.~\cite[Theorem 4.28]{WemyssLecture}.
   Now \Cref{lem:SODintersection}\ref{item:SODintersectionCY} shows the claim.
\end{proof}

Recall that singularities of type $\SMAL$ and ADE-hypersurface singularities are singular equivalent to hypersurface singularities $\C\llbracket z_0, \ldots, z_d\rrbracket/(f)$, see \Cref{S:AppendixDetailsOnSing}.
Hence, we see by the following proposition that \Cref{prop:reduceabsorption2} applies to these singularities.

\begin{prop}\label{P:Takahashi}
Let $S=\C\llbracket z_0, \ldots, z_d\rrbracket/(f)$ be an isolated singularity. Then the singularity category $\Dsg(S)$ is connected.
\end{prop}
\begin{proof}
$\Dsg(S)$ has no non-trivial thick subcategories, cf.\ e.g.\ \cite[Corollary 6.9]{Takahashi}. In particular, it is a connected triangulated category. 
\end{proof}

\subsection{Nodal singularities}
The following recognition result for nodal singularities can also be shown by arguing along the lines of the proof of \Cref{T:ClassOfSing}.

\begin{lem}\label{L:Yoshino}
Let $S$ be a complete local Gorenstein $\C$-algebra of Krull 
dimension $d$. If $\Dsg(S) \cong  \Dsg(\C\llbracket z_0,\dots,z_3 \rrbracket/(z_0 z_1-z_2 z_3))$, then $S \cong \C\llbracket z_0, \dots, z_d \rrbracket/(z_0^2 + \ldots + z_d^2)$, where $d$ is odd.
\end{lem}
\begin{proof}
As $\ul{\MCM}(S) \cong \Dsg(S) \cong  \Dsg(\C\llbracket z_0,\dots,z_3 \rrbracket/(z_0 z_1 - z_2 z_3))$, we see that $S$ has finite MCM representation type. 
It follows from \cite[Corollary 8.16]{YoshinoBook}, that $S$ is a simple singularity, so it is isomorphic an ADE-hypersurface singularity $\C\llbracket z_0, \dots, z_d\rrbracket/(f)$, by \cite[Theorem 8.8]{YoshinoBook}. The singularity categories of ADE-hypersurface singularities are completely understood, see e.g.\ \cite{YoshinoBook}. 
It follows from this classification that there exists a triangle equivalence $\Dsg(S) \cong  \Dsg(\C\llbracket z_0,\dots ,z_3 \rrbracket/(z_0 z_1 - z_2 z_3))$ if and only if $S \cong \C\llbracket z_0, \dots, z_d\rrbracket/(z_0^2 + \ldots + z_d^2)$ and $d$ is odd.
This completes the proof.
\end{proof}

%
%
%
%
%
%
%
%
%
%
%

\section{Gorenstein rings and their singularity categories}\label{sec:gor+sing}

Recall \Cref{defn:Goring,D:GProj} before reading this section. 
Unless stated otherwise, we assume that $\ca$ is an additive category for the rest of this section.

\subsection{\cSODs of projective Gorenstein varieties}
First we show that all $\C$-algebras appearing in a \absorption absorption of singularities of a projective Gorenstein variety are themselves Gorenstein.
This follows from work of Kuznetsov--Shinder \cite{KS22}.

\begin{prop} \label{P:KuznetsovShinder}
   Let $X$ be a projective Gorenstein variety, %
   $\ca  \subseteq \Db(X)$ be an admissible triangulated subcategory and $\ca = \langle \Db(R_1), \dots, \Db(R_n) \rangle$ be an admissible \SOD for some $\C$-algebras $R_1, \dots, R_n$.
   
   Then $R_1, \dots, R_n$ are finite-dimensional Gorenstein $\C$-algebras.
\end{prop}
\begin{proof}
   First, notice that $R_i$ is finite-dimensional as $\Db(X)$ is $\Hom$-finite.
   Next, $\Db(X)$ is Gorenstein in the sense of \cite[Definition 4.10]{KS22} by \cite[Proposition 6.1(iv)]{KS22}.
   So, $\Db(R_i) \subseteq \ca \subseteq \Db(X)$ is admissible and hence Gorenstein, again in the sense of \cite[Definition 4.10]{KS22}, by \cite[Proposition 4.13(i)]{KS22}.
   Then $R_i$ is Gorenstein as a dg-algebra concentrated in degree $0$, i.e.\ as a $\C$-algebra, by \cite[Proposition 6.9(iv)]{KS22}.
\end{proof}

\subsection{Gorenstein-projective modules -- Auslander's results.}

Our obstructions in \Cref{T:PropertiesDsgR}, used to show \Cref{T:Main}, rely on the homological properties of $\GP(R)$ for finite-dimensional Gorenstein $\C$-algebras $R$.
In order to study those properties it is convenient to treat $\GP(R)$ like a ring and study \emph{modules} over it.

\begin{defn}\label{defn:modA}
   A contravariant functor $F \colon \ca \to \mathrm{Ab}$ is called \emph{finitely presented}, if
   \[F(-) \cong \coker(\ca(-, X) \xrightarrow{\ca(-, f)} \ca(-, Y))\] 
   for a morphism $f\colon X \to Y$.
   For any additive category $\ca$, we define $\mod \, \ca$ to be the category of finitely presented contravariant (additive) functors $\ca \to \mathrm{Ab}$ with natural transformations as morphisms.
\end{defn}

\Cref{defn:modA} does indeed generalise the definition of modules over a ring. 
For example if $\Lambda$ is a Noetherian ring, then $\mod \proj \Lambda \cong \mod \Lambda$ as categories, where the left hand side is the category of finitely presented functors over the category of finitely generated projective $\Lambda$-modules $\proj \Lambda$ and the right hand side is the category of (ordinary) $\Lambda$-modules.

The following remark shows that the category of finitely presented functors over a $\kk$-linear category over a commutative ring $\kk$ form a $\kk$-linear category again.

\begin{rem}\label{rem:k-linear}
Notice, in \Cref{defn:modA}, if $\ca$ is a $\kk$-linear category, for a commutative ring $\kk$, and $F \colon \cA \to \mathrm{Ab}$ is finitely presented then $F(X)$ is canonically a $\kk$-module for all $X \in \cA$ as it is the cokernel of a $\kk$-linear map of $\kk$-modules. Similarly, any map $F(f) \colon F(Y) \to F(X)$ is $\kk$-linear for any morphism $f \in \ca(X,Y)$. Hence, $F$ can be canonically interpreted as a $\kk$-linear functor $\ca \to \Mod \kk$. 

With the above structure one can also easily show that any natural transformation $F \to G$ of finitely presented functors $F$ and $G$ becomes $\kk$-linear with respect to this canonical $\kk$-module structure, using the Yoneda lemma.  It follows that we could have replaced $\mathrm{Ab}$ by $\Mod \kk$ in \Cref{defn:modA} and obtain the same category $\mod \cA$. In particular, the category of finitely presented functors is a $\kk$-linear category.
\end{rem}

Next, we recall important definitions related to Auslander--Reiten theory.

\begin{defn}\label{defn:sink}
    We say a morphism $f \in \ca(X,Y)$ is
    \begin{enumerate}[label={(\alph*)}]
        \item \emph{right minimal} if $fg = f$ for $g \in \End_{\ca}(Y)$ implies that $g$ is an isomorphism, 
        \item \emph{right almost split} if it is not a split epimorphism and any non-split epimorphism $g \in \ca(Z,Y)$ factors through it, and
        \item \emph{a sink morphism for $Y$} if it is right almost split and right minimal.
        \item We say $\ca$ has \emph{enough sink morphisms}  if all $X \in \ind \ca$ admit a sink morphism.
        \end{enumerate}
        The notions of \emph{left minimal}, \emph{left almost split}, \emph{source morphisms} and \emph{enough source morphisms} are defined dually.
\end{defn}

\begin{defn}\label{defn:covering}
   Let $\cm$ be an additive subcategory of $\ca$. 
   \begin{enumerate}[label={{(\alph*)}}]
      \item A \emph{right $\cm$-approximation}\footnote{\label{theproperterms}These term are also called \emph{$\cm$-precover}, \emph{$\cm$-preenvelope}, \emph{precovering} and \emph{preenveloping}.} of an object $X \in \ca$ is a map $f \colon M \to X$ such that the induced map $\ca(N, f) \colon \ca(N, M) \to \ca(N, X)$ is surjective for all $N \in \cm$. 
         Dually, we define \emph{left $\cm$-approximations}\footref{theproperterms}.
      \item We say $\cm$ is \emph{covariantly finite\footref{theproperterms}} in $\ca$ if every $X \in \ca$ admits a right $\cm$-approximation.
          Dually, we define \emph{contravariantly finite}\footref{theproperterms} subcategories of $\ca$.
    \item We say $\cm$ is \emph{functorially finite} in $\ca$ if it is covariantly and contravariantly finite in $\ca$.
   \end{enumerate}
    \end{defn}

The following results go back to ideas of Auslander, Buchweitz, Reiten, and others. 

\begin{lem}\label{lem:GPfunfinite}
   Let $\Lambda$ be an Gorenstein ring.
   Then $\GP(\Lambda) \subseteq \mod \Lambda$ is a covariantly finite subcategory. In particular, $\GP(\Lambda)$ has weak kernels.
\end{lem}
\begin{proof} 
   This first part follows from \cite[Theorem 1.8]{AB}. 
   Then, for any morphism in $\GP(\Lambda)$ a weak kernel in $\GP(\Lambda)$ can by constructed by a right $\GP(\Lambda)$-approximation of a kernel in $\mod \Lambda$.
\end{proof}

\begin{thm}\label{T:Auslander}
   Let $\Lambda$ be a Gorenstein ring. Then the following statements hold:
   \begin{enumerate}[label={{(\alph*)}}]
      \item $\GP(\Lambda)$ has weak kernels and hence $\mod \GP(\Lambda)$ is an abelian category.\label{item:GPab}
      \item $\gldim \mod \GP(\Lambda)<\infty$.\label{item:GPglfin}
      \item If $\Lambda$ is additionally a finite-dimensional $\kk$-algebra over a field $\kk$ then $\GP(\Lambda)$ is a $\kk$-linear, $\Hom$-finite, Krull--Schmidt category with enough sink morphisms.\label{item:GPrest}%
   \end{enumerate}
   Moreover, the same statements are true when replacing $\GP(\Lambda)$ by $\GP(\Lambda)^{\op}$
\end{thm}
\begin{proof}
    Clearly, $\Hom_{\Lambda}(-,\Lambda)$ yields a duality between $\GP(\Lambda)$ and $\GP(\Lambda^{\op})$. As the definition of Gorenstein rings is symmetric this shows the last part, given \ref{item:GPab}-\ref{item:GPrest}.

   \ref{item:GPab}: By \Cref{lem:GPfunfinite}, the category $\GP(\Lambda)$ has weak kernels.
   By \cite[Section III.2]{AuslanderQueen} or \cite[Proposition 2.7]{Enomoto16} we conclude that $\mod \GP(\Lambda)$ is an abelian category

   \ref{item:GPglfin}: By \cite[Corollary 4.13]{Enomoto16}, $\GP(\Lambda)$ has $(n-1)$-kernels, where $n=\injdim_\Lambda \Lambda_\Lambda$. 
   Then the statement follows from \cite[Proposition 4.6]{Enomoto16}.

\ref{item:GPrest}:
   By \Cref{D:GProj}, the subcategory $\GP(\Lambda) \subseteq \mod \Lambda$ is closed under direct summands. Hence, $\GP(\Lambda)$ is idempotent complete, since $\mod \Lambda$  is idempotent complete as it is an abelian category.
   Since $\GP(\Lambda)$ is a subcategory of $\mod \Lambda$ it is clear that $\GP(\Lambda)$ is $\Hom$-finite.
   It follows from \cite[Corollary 4.4]{Kra15} that $\GP(\Lambda)$ is Krull--Schmidt. 

   As $\Lambda$ is a finite-dimensional $\kk$-algebra, $\mod \Lambda$ has Auslander--Reiten sequences and hence enough sink morphisms, where the sink morphism of a indecomposable projective $\Lambda$-module $P$ is given by the inclusion $\rad(P) \to P$.
\Cref{lem:sinksubcat} and \Cref{lem:GPfunfinite} show that $\GP(\Lambda)$ has enough sink morphisms.
\end{proof}

Notice that \Cref{T:Auslander} only depends on the underlying additive structure of $\GP(\Lambda)$ and not on the exact structure of $\GP(\Lambda)$.

\begin{prop}\label{cor:Serre}
    Let $R$ be a finite-dimensional Gorenstein $\kk$-algebra, for a field $\kk$, and $\GP_\cp(R)$ be a Frobenius exact structure on $\GP(R)$.
    
    Then the triangulated category $\underline{\GP}_{\cp}(R)$ has a Serre functor.
\end{prop}
\begin{proof}
    By \Cref{T:Auslander}\ref{item:GPrest} we know that $\GP(R)$ is a $\Hom$-finite Krull--Schmidt category with enough source and sink morphisms. Then, $\GP_\cp(R)$ has almost split extensions, cf.\ \cite[Definition 3.1]{INP}, by \cite[Lemma 3.2]{INP}, and so does $\ul{\GP}_{\cp}(R)$ using \cite[Proposition 5.11]{INP}.
    As $\ul{\GP}_{\cp}(R)$ is a $\Hom$-finite triangulated category it has a Serre functor by \cite[Theorem 3.6]{INP}.
\end{proof}

\section{Auslander--Solberg reduction}\label{sec:aussolred}
\subsection{Setting the stage}
We first recall, that an exact category with enough projective and injective objects is called \emph{Frobenius exact category} if its projective and injective objects coincide.
In the following, we explain a technique, which under certain technical assumptions produces Frobenius exact categories from exact categories by a modification of exact structures on some fixed category $\ce$. 
This is implicitly contained in  a work of Auslander--Solberg \cite{AS93-1}. 
The technique may be used to obtain new Frobenius exact structures on a given Frobenius exact category.
The new Frobenius exact structure has a larger subcategory of projective-injective objects.
For the rest of this section we fix a category $\ce$. 
\begin{Setup}
Let $\kk$ be an algebraically closed field and $\ce$ be a $\kk$-linear category. 
Let $D(-) := \Hom_{\kk}(-,\kk)$ be the $\kk$-duality on $\mod \kk$.
\end{Setup}

For a detailed exposition of exact categories we refer to Bühler \cite{Buehler}.
Recall that for an exact category $(\cx,\cs_{\cp})$ with enough projective objects $\cp$ one can recover the conflations $\cs_{\cp}$ of $(\cx,\cs_\cp)$ from $\cp$, see \cite[Exercise 11.10]{Buehler}.
\begin{Notation}\label{Not:exact}
    When we write that $\cx_\cp$ is an exact category we mean implicitly that $\cx_\cp \coloneqq (\cx, \cs_\cp)$ is an exact category with enough projective objects given by $\cp$.
\end{Notation}

  The following definition is due to Happel, cf.\ \cite{HappelBook}.

\begin{defn}\label{Not:frobenius}
   Let $\cx_{\cp}$ be a Frobenius exact category.
   Then the additive quotient $\cx/\cp$ has a canonical triangulated structure called the \emph{stable category of $\cx_{\cp}$}.
   We denote it by $\ul{\cx}_{\cp} := (\cx/\cp, \smash{\Sigma_{\ul{\cx}_\cp}}, \smash{\Delta_{\ul{\cx}_\cp}})$. Notice, $\smash{\Sigma_{\ul{\cx}_\cp}} = \smash{\Omega^{-1}_{\ul{\cx}_\cp}}$ is the cosyzygy functor.
\end{defn}

\begin{Notation}\label{not:abuse}
   Sometimes, we may denote the underlying additive category of an exact category $\cx_{\cp}$ (or a triangulated category $\cc$) by the same symbol $\cx_{\cp}$ (or $\cc$).
\end{Notation}

Our main interest lies in modifying Frobenius exact structures $\GP_\cp(R)$ on the category $\GP(R)$ of Gorenstein projective modules over a finite-dimensional Gorenstein $\kk$-algebra $R$.
If $\ul{\GP}_{\cp}(R) \cong \cc_1 \oplus \cc_2$ as triangulated categories then we want to obtain a new Frobenius exact structure $\GP_\cm(R)$ with $\ul{\GP}_{\cm}(R) \cong \cc_1$ as triangulated categories.
Moreover, if $\ul{\GP}_{\cp}(R) \cong \Dsg(S)$ for certain singularities $S$ then we want to obtain a new Frobenius exact structure $\GP_\cm(R)$ satisfying $\ul{\GP}_{\cm}(R) \cong \Dsg(S')$ for a simpler singularity $S'$.

\subsection{Modification of exact structures}

The following is the main idea, which we apply in different situations throughout \Cref{sec:aussolred}.

\begin{prop}[cf.~e.g.~{\cite[Proposition 2.16]{KIWY}}]\label{new Frobenius structure} 
Suppose $\ce_\cp = (\ce, \cs_\cp)$ is an exact structure on $\ce$ with enough injective objects $\ci$. 
Assume there is an equivalence $\tau\colon\ce/\cp\to\ce/\ci$ with a functorial isomorphism 
\[\Ext^1_{\ce_\cp}(X,Y)\cong D\Hom_{\ce/\ci}(Y,\tau X)\]
for $X,Y\in\ce$.
Let $\cm$ be a functorially finite, additive subcategory of $\ce$ which is closed under direct summands, contains $\cp$ and $\ci$ and satisfies $\tau(\cm/\cp)=\cm/\ci$. 

Then there exists a Frobenius exact structure $\ce_\cm = (\ce, \cs_\cm)$ on $\ce$ with projective objects $\cm$. 
Furthermore, $\cs_\cm \subset \cs_\cp$ is the class of conflations $Z \rightarrowtail Y \twoheadrightarrow X$ of $\ce_\cp$ whose deflations induce epimorphisms $\ce(M, Y) \to \ce(M, X)$ for all $M \in \cm$.

\end{prop}

If $\ce_{\cp}$ in \Cref{new Frobenius structure} is a Frobenius exact category then $\cp = \ci$. In that an equivalence $\tau$ as in \Cref{new Frobenius structure} is equivalent to the datum of a Serre functor.

\begin{cor}\label{c:changefrob}
Let $\ce_\cp$ be a Frobenius exact structure on $\ce$ such that $\ul{\ce}_{\cp}$ has a Serre-functor $\mathbb{S}$. 
Let $ \cm$ be a functorially finite, additive subcategory of $\ce$ which is closed under direct summands, contains $\cp$, and satisfies $\tau(\cm/\cp) = \cm/\cp$, where $\tau$ is the Auslander--Reiten translation $\tau \coloneqq \mathbb{S} \circ \Sigma_{\ul{\ce}_\cp}^{-1}$ on $\ul{\ce}_\cp$. 

Then \Cref{new Frobenius structure} yields a Frobenius exact structure $\ce_\cm$ on $\ce$ with projective-injective objects $\cm$.
\end{cor}
\begin{proof}
   Since projective and injective objects in $\ce_{\cp}$ coincide, in order to apply Proposition \ref{new Frobenius structure}, it suffices to show functorial isomorphisms
\begin{align}\label{E:AR-duality}
\Ext^1_{\ce_\cp}(X,Y)\cong D\Hom_{\ce/\cp}(Y,\tau X)
\end{align}
for any $X,Y\in\ce$. 
In any Frobenius exact category there is a functorial isomorphism $\Ext^1_{\ce_\cp}(X,Y)\cong \Hom_{\ce/\cp}(\Sigma_{\ul{\ce}_\cp}^{-1}X, Y)$.
Now \eqref{E:AR-duality} follows from Serre duality in $\ul{\ce}_{\cp}$.
\end{proof}

\subsection{Triangulated summands of stable categories}
A direct calculation using the third isomorphism theorem for $\kk$-vectorspaces shows the following useful lemma.
\begin{lem}\label{L:AddIsomorphThm}
     For arbitrary $\kk$-linear categories $\cp \subset \cm \subset \ce$ the canonical map $\ce/\cm \to (\ce/\cp)/(\cm/\cp)$ is an isomorphism of $\kk$-linear categories.
\end{lem}

We are now in the position to achieve our first goal. Indeed, specializing
the following statement  to the category \mbox{$\ce = \GP(R)$} of Gorenstein projective modules over a finite-dimensional Gorenstein $\kk$-algebra $R$ yields \Cref{L:AuslanderSolbergSummand} using \Cref{P:SerreFunctor}.

\begin{lem} \label{C:RemoveFiniteTypeComponents}
Let $\ce_\cp$ be a Frobenius exact structure on $\ce$ such that $\ul{\ce}_\cp = \cc_1 \oplus \cc_2$ is a direct sum of triangulated subcategories. Assume $\ul{\ce}_{\cp}$ has a Serre functor $\mathbb{S}$.

Then, in particular, $\mathbb{S}$ restricts to a Serre functor on $\cc_1$ and $\cc_2$ and there is a Frobenius exact structure $\ce_\cm$ on $\ce$ such that there is a triangle equivalence $\ul{\ce}_{\cm} \cong \cc_1$. 
\end{lem}
\begin{proof}
The Serre functor $\mathbb{S}$ on $\ul{\ce}_{\cp}$ restricts to a Serre functor on $\cc_1$ since for a fixed $X \in \cc_1$ and every $Y \in \cc_2$, we have 
\[0 = \Hom_{\ce/\cp}(X, Y) \cong \sD \Hom_{\ce/\cp}(Y, \mathbb{S}(X)),\] 
so $\mathbb{S}(X)$ has to be in $\cc_1$ as well. 
Similarly, $\mathbb{S}$ restricts to a Serre functor on $\cc_2$.

Define $\cm \coloneqq \Phi^{-1}(\cc_2)$ be the preimage of $\cc_2$ under the quotient $\Phi \colon \ce \to \ce/\cp$. 
By \cite[Lemma A.7]{ACFGS22}, we have that $\cm \subset \ce$ is functorially finite and hence there is a Frobenius exact structure $\ce_\cm$ on $\ce$ with projective objects $\cm$, by \Cref{c:changefrob}. 

To show that $\cc_1$ and $\ul{\ce}_{\cm}$ are equivalent as triangulated categories we construct a diagram
\begin{figure}[H]\begin{tikzcd}
    {\mathcal{E}_{\mathcal{M}}} & {\mathcal{X}_\mathcal{P}} & {\mathcal{E}_{\mathcal{P}}} \\
    {\underline{\mathcal{E}}_{\mathcal{M}}} & {\underline{\mathcal{X}}_{\mathcal{P}}} & %
    {\cc_1 \oplus \cc_2 \cong \mathrlap{ \ul{\ce}_{\cp}}}
    \arrow["\text{incl.}", from=1-2, to=1-3]
    \arrow["\text{incl.}"', from=1-2, to=1-1]
    \arrow[from=1-1, to=2-1]
    \arrow[from=1-2, to=2-2]
    \arrow["\sim"', from=2-2, to=2-1]
    \arrow[from=1-3, to=2-3]
    \arrow[from=2-2, to=2-3]
\end{tikzcd}\end{figure} \noindent where $\cx_\cp$ is a Frobenius exact subcategory of both $\ce_\cp$ and $\ce_\cm$ such that the inclusion $\cx_\cp \to \ce_{\cp}$ induces the inclusion $\cc_1 \to \ul{\ce}_\cp$ of the triangulated subcategory $\cc_1$ and such that the inclusion $\ce_\cm \leftarrow \cx_\cp$ induces a triangle equivalence $\ul{\ce}_{\cm} \xleftarrow{\sim} \ul{\cx}_{\cp}$.

Let $\cx \coloneqq \Phi^{-1}(\cc_1)$ be the preimage of $\cc_1$ under the quotient $\Phi \colon \ce \to \ce/\cp$.
By the same argument as in \cite[Example 2.63]{KalckThesis}, we see that $\cx$ is closed under extensions, kernels of deflations and cokernels of inflations in $\ce_\cp$ and the inherited exact structure is a Frobenius exact structure $\cx_\cp$ with projective objects $\cp$.
It is clear that $\ul{\cx}_\cp = \cc_1$ as triangulated categories.

The way the conflations of $\ce_\cm$ are constructed in \Cref{new Frobenius structure} together with $\Ext^1_{\ce_\cp}(M,X)\cong \Hom_{\ce/\cp}(\Sigma_{\ul{\ce}_\cp}^{-1}M, X) = 0$ for all $M \in \cm$ and $X \in \cx$ imply that any conflation in $\cx_\cp$ is a conflation in $\ce_\cm$. As $\cp \subset \cm$ we have that $\cx_\cp$ is a Frobenius exact subcategory of $\ce_{\cm}$ and hence the inclusion $\ce_\cm \leftarrow \cx_\cp$ induces a triangulated functor $\ul{\ce}_{\cm} \leftarrow \ul{\cx}_{\cp}$ by \cite[page 23]{HappelBook}. 

It is not hard to see that (as an additive functor) this functor agrees with the restriction of the canonical quotient $\ce/\cp \to (\ce/\cp)/(\cm/\cp)$ to $\cx/\cp$. Because $\ce/\cp \cong \cx/\cp  \oplus \cm/\cp$ as additive categories, \Cref{L:AddIsomorphThm} implies that $\ul{\ce}_{\cm} \leftarrow \ul{\cx}_{\cp}$ is a triangle equivalence.
\end{proof}

Consider the special case of \Cref{C:RemoveFiniteTypeComponents} where $\ce_{\cp} = \GP_{\cp}(R)$ is a Frobenius exact structure on the Gorenstein projective modules of a Gorenstein $\kk$-algebra $R$ and $\cm = \add M$ for some $M$.
Then there is an equivalence $\GP_{\cm}(R) \cong \GP_{\proj R'}(R')$ of exact categories by \Cref{T:Auslander} and \cite[Theorem 2.8]{KIWY}, where $R' \coloneqq \End_{R}(M)$ is also a finite-dimensional Gorenstein $\kk$-algebra. In that case $R = eR'e$ for an idempotent $e \in R'$.

The following example shows that there are connected Frobenius exact categories $\ce_\cp$ such that the stable category is disconnected, i.e.\ there is a triangle equivalence $\ul{\ce}_\cp \cong \cc_1 \oplus \cc_2$
for non-trivial triangulated categories $\cc_i$.
\begin{ex}\label{Ex:Conndisconnected}
   Let $R=\kk Q/I$ be given by the following quiver
\[Q \coloneqq \begin{tikzcd}[ampersand replacement=\&, column sep = 2em]
	1 \&\& 2 \&\& 3
	\arrow["{a_1}", bend left = 2em, from=1-1, to=1-3]
	\arrow["{a_2}", bend left = 2em, from=1-3, to=1-5]
	\arrow["{b_2}", bend left = 2em, from=1-5, to=1-3]
	\arrow["{b_1}", bend left = 2em, from=1-3, to=1-1]
\end{tikzcd}\]
and admissible ideal $I$ generated by the relations $a_ib_i$ and $b_ia_i$ for $i=1,2$.
Suppose $\ce_{\cp} \coloneqq \GP_{{\proj R}}(R)$ is the category of Gorenstein projective $R$-modules with the standard Frobenius exact structure.
As $Q$ is connected, the additive category $\GP(R)$ is connected.
Indeed, this follows since $\proj R$ is connected and every $X$ in $\GP(R)$ admits a non-zero morphism $P \to X$ with $P$ projective.

The $\kk$-algebra $R$ is a gentle algebra. 
So \cite[Theorem 2.5(b) and (3.4)]{Kalck15} yield triangle equivalences
\begin{align}\label{E:2nodes}
    \ul{\GP}_{\proj R}(R) \cong \frac{\Db(\kk)}{\Sigma^2} \oplus \frac{\Db(\kk)}{\Sigma^2} \cong \Dsg\left(\frac{\kk \llbracket z_0, z_1\rrbracket}{(z_0 z_1)}\right) \oplus \Dsg\left(\frac{\kk \llbracket z_0, z_1\rrbracket}{(z_0 z_1)}\right)
\end{align}
where $\Db(\kk)/\Sigma^2$ denotes the orbit category of $\Db(\kk)$.
\end{ex}

\begin{rem}
    The $\kk$-algebra $R$ in Example \ref{Ex:Conndisconnected} appears in \KSODs of singular varieties, cf.\ \cite{Burban05} for curves and \cite{ps} for threefolds. 
    Its singularity category splits into the singularity categories of two nodes by \eqref{E:2nodes}. 
    However, we show below that this splitting is \emph{not} induced by an admissible \SOD of $\Db(R)$ of the form
    \begin{align}\label{E:fdcatabs2nodes}
        \Db(R) = \langle \Db(R_1), \Db(R_2) \rangle,
    \end{align}
    where the $R_i$ are finite-dimensional $\kk$-algebras. 
    In other words, $\Db(R)$ provides an \absorption of two nodes, which cannot be refined to obtain \absorptions for each node separately. 
    
    However, if we allow derived categories of dg-algebras as components, then the following \SOD yields separate categorical absorptions of the two nodes
    \begin{align*}
        \Db(R) = \langle \Dfd(\kk[\epsilon]/(\epsilon^2)), \Dfd(\kk[\epsilon]/(\epsilon^2)), \Db(\kk) \rangle,
    \end{align*}
where $\kk[\epsilon]/(\epsilon^2)$ is the dg-algebra with $\deg \epsilon=-1$ and trivial differential. Such decompositions are studied in greater generality in \cite{KS22a}.

    To see that a \SOD as in \eqref{E:fdcatabs2nodes} is impossible, we note that 
    \begin{align}\label{E:SingK0node}
        K_0\left(\Dsg(\kk\llbracket z_0, z_1\rrbracket/(z_0  z_1))\right) \cong \ZZ,
    \end{align}
    cf.\ \cite[Lemma 2.22]{KPS19}.
    For a finite-dimensional $\kk$-algebra with precisely $n$ distinct simple modules the Gothendieck group of its derived category is isomorphic to $\Z^n$, see \cite[Section III.1]{HappelBook}.
    Furthermore, if it is additionally local of $\kk$-dimension $d$ then the Grothendieck group of its singularity category is isomorphic to $\ZZ/d\ZZ$. 
    
    Hence, $K_0(\Db(R)) \cong \ZZ^3$.
    As $K_0$ is additive the admissible \SOD \eqref{E:fdcatabs2nodes} implies $K_0(\Db(R_1)) \cong \ZZ$ or $K_0(\Db(R_2)) \cong \ZZ$ and therefore one of the $R_i$ is local. 
    The admissibility of \eqref{E:fdcatabs2nodes} yields that
    \begin{align}\label{E:fdcatabs2nodes2}
        \Dsg(R) \cong \langle \Dsg(R_1), \Dsg(R_2) \rangle,
    \end{align}
    is admissbile, see \cite[Proposition 1.10]{orl6}.
    Considering Grothendieck groups in \eqref{E:fdcatabs2nodes2} and using  \eqref{E:2nodes} together with Buchweitz's equivalence \Cref{T:BuchweitzIntro} as well as \eqref{E:SingK0node} yields a contradiction.
\end{rem}

\subsection{Standard triangulated categories}

For a detailed exposition of quivers and quiver algebras we refer the reader to \cite[Chapter II]{ASS}.

\begin{Notation}
   Let $Q=(Q_0, Q_1)$ be a quiver.
   We denote the \emph{source} and \emph{target} of $\alpha \in Q_1$ by $s(\alpha)$ and $t(\alpha)$, respectively. 
   For $x,y \in Q_0$ we write $Q_1(x,y) \subset Q_1$ for the set of arrows $\alpha \in Q_1$ with $s(\alpha) = x$ and $t(\alpha) = y$. 
   If $Q_1(x,y) \neq \emptyset$ we say that $y$ is a \emph{direct successor} of $x$ and $x$ is a \emph{direct predecessor} of $y$.
   We denote by $x^+$ and $x^-$ the set of all direct successors and direct predecessors of $x \in Q_0$, respectively.
\end{Notation}

\begin{defn}
   Let $Q = (Q_0, Q_1)$ be a quiver.
   \begin{enumerate}[label={(\alph*)}]
      \item We say $Q$ has \emph{no double arrows} if $|Q_1(x, y)| \leq 1$ for any vertices $x,y \in Q_0$.
      \item We call $Q$ \emph{locally finite} if $x^+$ and $x^-$ are finite for all $x \in Q_0$.
      \item A \emph{valuation} on $Q$ is a map $a \colon Q_1 \to \mathbb{N}_{+}$.
            We call $(Q,a)$ a \emph{valued quiver}.
   \end{enumerate}
\end{defn}

The following was originally introduced by Riedtmann for quivers without loops, see \cite[pages 200 and 205]{Riedtmann}.
However, e.g.\ \cite[Definition 2.1.1]{Amiotthesis} allows the underlying quiver to have loops. 

\begin{defn}\label{defn:transquiver}
    A \emph{stable translation quiver} is a pair $(Q, \tau)$ consisting of a locally finite quiver $Q = (Q_0, Q_1)$ without double arrows and a bijection $\tau \colon Q_0 \to Q_0$ which satisfies $(\tau x)^+ = x^-$.
    The bijection $\sigma \colon Q_1 \to Q_1$ with $\sigma Q_1(x, y) = Q_1(\tau y, x)$ is called the \emph{polarisation of $(Q,\tau)$}.
\end{defn}

Notice, since a stable translation quiver does not have double arrows, the polarisation of a stable translation quiver is unique, compare \cite[page 51]{RingelBook}.

The following generalisation is due to Happel--Preiser--Ringel, see \cite[page 289]{Happel-Preiser-Ringel}. See also \cite[Definition 2.1.2]{Amiotthesis}.

\begin{defn}
   Let $(Q,\tau)$ be a stable translation quiver with polarisation $\sigma$.
   A \emph{valued stable translation quiver} is a triplet $(Q,\tau,a)$ where $a$ is a valuation on $Q$ satisfying $a(\sigma \alpha) = a(\alpha)$ for any $\alpha \in Q_1$. 
\end{defn}

\begin{defn}\label{defn:quiver}
    Given a $\kk$-linear, $\Hom$-finite, idempotent complete category $\ca$ one can define a quiver $Q = (Q_0, Q_1)$ without double arrows, called the \emph{quiver of $\ca$}, and a valued quiver $(Q,a)$, called the \emph{valued quiver of $\ca$}.
    \begin{enumerate}[label={(\alph*)}]
        \item The vertices of $Q_0$ are isomorphism classes $[X]$ of $X \in \ind \ca$,
        \item there is an arrow $\alpha \in Q_1([X],[Y])$ precisely if $\rad_{\ca}(X,Y)/\rad_{\ca}^2(X,Y) \neq 0$ for any $X,Y \in \ind \ca$ and in this case $a(\alpha) = \dim_{\kk} \rad_{\ca}(X,Y)/\rad_{\ca}^2(X,Y)$.
    \end{enumerate}
\end{defn}

\begin{defn}\label{defn:ARquiver}
    Let $\cc$ be a $\kk$-linear, $\Hom$-finite, idempotent complete, triangulated category.
    Suppose $\cc$ admits a Serre functor $\mathbb{S} \colon \cc \to \cc$.
    Let $(Q,a)$ be the valued quiver of $\cc$ and $\tau \colon Q_0 \to Q_0, \, [X] \mapsto [\Sigma^{-1}\mathbb{S}(X)]$ the \emph{Auslander--Reiten translation}.
    We call $(Q, \tau, a)$ the \emph{Auslander--Reiten quiver (AR-quiver)} of $\cc$.
\end{defn}

It is well known that the AR-quiver of $\cc$ as above is indeed a valued stable translation quiver. 
Recall also the following well-known fact.

\begin{rem}\label{rem:unique}
    Under the notation and assumptions of \Cref{defn:ARquiver} assume the quiver $Q$ of $\cc$ is connected. Then the action of the Auslander--Reiten translation $\tau$ on objects and hence the AR-quiver $(Q,a,\tau)$ of $\cc$ does only depend on the underlying $\kk$-linear structure.
    Indeed, the existence of a Serre-functor is independent of the triangulation and by \cite[Proposition I.2.4]{ReitenVandenBergh} any triangulation on $\cc$ has Auslander--Reiten triangles. We may assume that $Q$ has more than one vertex because otherwise the action of $\tau$ on objects is unique anyway. Using that $(Q,\tau)$ is a connected stable translation quiver, one then easily shows that at each vertex of $Q$ starts and ends at least one arrow and hence the Auslander--Reiten triangles of $\cc$ are non-degenerate. Applying \cite[Theorem 4.9]{shah} shows the claim.  %
\end{rem}

We will be especially interested in Auslander--Solberg reduction on triangulated categories with an additive generator.

\begin{defn}
    A triangulated category $\cc$ is called \emph{additively finite} if it has an additive generator, i.e.\ if $\cc = \add X$ for some $X \in \cc$.
\end{defn}

\begin{rem}\label{rem:exists}
    Any $\Hom$-finite, idempotent complete, additively finite, triangulated category has a Serre functor, cf.\ \cite[Theorem 1.1]{Amiotthesis}.
\end{rem}

Let $(Q, \tau, a)$ be a valued translation quiver with underlying quiver $Q = (Q_0, Q_1)$ and polarisation $\sigma$.
Define a quiver $\overline{Q} := (\overline{Q}_0, \overline{Q}_1)$ by replacing $\alpha \in Q_1$ by $a(\alpha)$ many arrows, that is $\overline{Q}_0 = Q_0$ and 
\[\overline{Q}_1 = \set{(\alpha,k) \in Q_1 \times \mathbb{N}}{\text{$\alpha \in Q_1$ and $1 \leq k \leq a(\alpha)$}}.\]
Furthermore, define $\overline{\sigma} \colon \overline{Q}_1 \to \overline{Q}_1,\ (\alpha, k) \mapsto (\sigma \alpha, k)$.
Recall also the following, cf.\ e.g.\ \cite[Section 2.5 and Definition 5.1]{Gabriel} or \cite[Section 5]{Amiotthesis}. 
\begin{defn}\label{defn:meshalgebra}
    Suppose $(Q,\tau,a)$ is a finite valued stable translation quiver with polarisation $\sigma$.
    Let $\overline{Q} = (\overline{Q}_0, \overline{Q}_1)$ and $\overline{\sigma}$ be constructed as above.
    \begin{enumerate}[label={(\alph*)}]
       \item The ideal\footnote{We follow the same convention as \cite{ASS} for the composition of arrows.} $I(Q,\tau,a) \coloneqq \big(\smash{\sum_{\alpha \in \overline{Q}_1}} \alpha \overline{\sigma}(\alpha) \big)$ of $\kk \smash{\overline{Q}}^{\op}$ is called the \emph{mesh ideal}.
       \item The \emph{mesh algebra} is the quotient $\kk (Q, \tau, a) := \kk \smash{\overline{Q}}^{\op}/I(Q,\tau,a)$.
       \item A $\kk$-linear, additively finite\footnote{{To define standardness for triangulated categories that are not additively finite }one needs to replace $\kk \smash{\overline{Q}}^{\op}$ by the path category of $Q$ and adopt the definition of $I(Q,\tau,a)$ appropriately.} triangulated category $\cc$ is called \emph{standard} if it is $\Hom$-finite, idempotent complete and $\cc \cong \proj \kk (Q,\tau,a)$ as additive categories, where $(Q,\tau,a)$ is the AR-quiver of $\cc$, cf.\ \Cref{rem:exists,rem:unique}.
    \end{enumerate}
\end{defn}

Regarding the definition of the mesh algebra notice, \cite[page 47]{RingelBook} allows a stable translation quiver $(Q,\tau)$ to have double arrows. This makes the polarisation $\sigma$ of $(Q,\tau)$ in \Cref{defn:transquiver} and hence the mesh relations subject to a choice, compare \cite[page 51]{RingelBook}. In our definition of the mesh relations this choice is hidden in the definition of the map $\overline{\sigma}$ on $(\overline{Q}, \tau)$.
We can derive the following result.
\begin{cor} \label{C:Frob}
   Let $\ce_\cp$ be a $\kk$-linear Frobenius exact category and $\ul{\ce}_{\cp}$ be additively finite and standard in the sense of \Cref{defn:meshalgebra}.

   If removing $\tau$-orbits from the AR-quiver of $\smash{\ul{\ce}_{\cp}}$ yields a connected quiver which is isomorphic to the AR-quiver of a standard triangulated category $\smash{\cc}$ then there is a Frobenius exact structure $\smash{\ce_{\cm}}$ on $\ce$ such that $\smash{\ul{\ce}_{\cm} = \cc}$ as triangulated categories.
\end{cor}
\begin{proof}
   We denote the $\tau$-orbits in the AR-quiver of $\ul{\ce}_{\cm}$ that we want to remove by $X_i, \tau(X_i), \ldots, \tau^{n_i}(X_i)$ for $1 \leq i \leq t$, set
\[
   X \coloneqq \bigoplus_{i=1}^t \bigoplus_{k=0}^{n_i} \tau^k(X_i)
\]
and define $\cm$ to be the preimage of $\add X$ under the canonical quotient $\ce \to \ce/\cp$.
Then $\cm \subseteq \ce$ is functorially finite by \cite[Lemma A.7]{ACFGS22} as $\add X =\cm/\cp \subseteq \ce/\cp$ has an additive generator and is therefore functorially finite.
Now, $\ce$ admits a Frobenius exact structure $\ce_\cm$ with projective-injective objects $\cm$, by \Cref{c:changefrob}.
By construction of $\cm$, we have an equivalence $\ce/\cm \cong (\ce/\cp)/\add X$ as additive categories, using Lemma \ref{L:AddIsomorphThm}. 
As $\add X = \cm / \cp$ is functorially finite in $\ce/\cp$ the conditions of \mbox{\cite[Set-up 3.1]{Jor0705}} are satisfied. 
By \Cref{rem:unique} and \cite[Theorem 4.2(ii)]{Jor0705} the \mbox{AR-quiver} of $\ul{\ce}_{\cm}$ is obtained from the AR-quiver of $\ul{\ce}_{\cp}$ by removing the vertices corresponding to the summands in $X$ and all adjacent arrows.
Further, $\ul{\ce}_{\cm}$ is standard by \cite[Theorem 4.2(iii)]{Jor0705} since $\ul{\ce}_{\cp}$ is standard and the AR-quiver of $\ul{\ce}_{\cm}$ is isomorphic to the AR-quiver of $\cc$.
Additively finite, standard triangulated categories with isomorphic, connected AR-quivers are triangle equivalent by \mbox{\cite[Corollary 2]{KellerAmiot}} (see also \cite{KalckYang16}), so there is a triangle equivalence $\ul{\ce}_{\cm} \cong \cc$.\end{proof}

\begin{cor} \label{C:ReduceType}
Let $d$ be odd and $S$ be a complete local $d$-dimensional ADE-hypersurface singularity such that there is a finite-dimensional Gorenstein $\C$-algebra $R$ and a Frobenius exact structure $\GP_\cp(R)$ which satisfies $\underline{\GP}_{\cp}(R) \cong \Dsg(S)$.

Then there is a Frobenius exact structure $\GP_\cm(R)$ on $\GP(R)$ such that there is a triangle equivalence $\underline{\GP}_{\cm}(R) \cong \Dsg(S')$, where $S'$ is a complete local $3$-dimensional ADE-hypersurface singularity
\begin{enumerate}[label={(\alph*)}]
\item of type $A_3$, if $S$ is of type $D_{2m+1}$ (with $m >1$) or of type $E_6$
\item of type $D_4$, if $S$ is of type $E_7$ of $E_8$.
\end{enumerate}
In particular, $\Dsg(S')$ is $2$-Calabi--Yau and $\Spec S'$ has a small resolution, see for example \cite[Theorems 5.7 and 6.2(a)]{BIKR}.
\end{cor}
\begin{proof}
The category $\Dsg(S)$ and $\Dsg(S')$ are standard, cf.\ e.g.\ \cite[Theorem 8.7]{ErdmannSkowronski}.
Therefore, by Corollary \ref{C:Frob}, it is enough to show that we can remove $\tau$-orbits from the AR-quivers of $\underline{\GP}_{\cp}(R) = \Dsg(S)$ to obtain the AR-quiver of $\Dsg(S')$ in each case. 
The AR-quivers of $\Dsg(S)$ are known by \cite{DieterichWiedemann86} combined with \cite{Knoerrer}. 
We suppress indicating the valuations on the AR-quivers, as all arrows evaluate to $1$.
For type $D_{2m+1}$ the AR-quiver is given as follows
\begin{center}
    \includegraphics{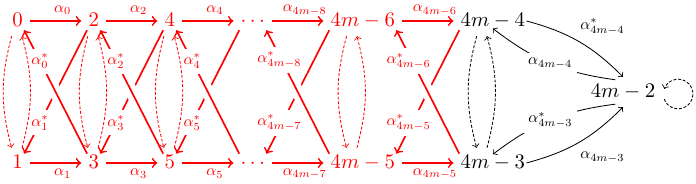}
\end{center}
where the dashed arrows indicate the action of the Auslander--Reiten translation. 
Removing the $\tau$-orbits
$(0,1), (2, 3), \ldots, (4m-6, 4m-5)$ and all adjacent arrows, which are colored in red, we obtain the black subquiver
which is the AR-quiver of type $A_3$.
For type $E_6$ the AR-quiver is given as follows
\begin{center}
    \includegraphics{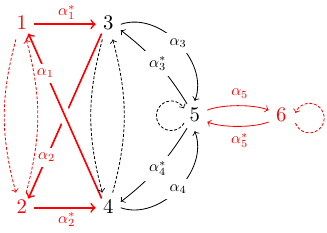}
\end{center}
and removing the $\tau$-orbits
$(1,2)$ and $(6)$ and all adjacent arrows, which are colored in red, we obtain the black subquiver which is again the AR-quiver of type $A_3$. 

We treat the cases $E_7$ and $E_8$ together.
For type $E_8$ the AR-quiver is given by
\begin{center}
    \includegraphics{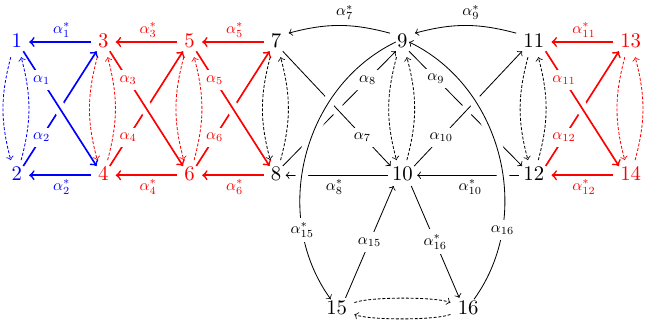}    
\end{center}
and removing the $\tau$-orbit $(1,2)$ and all adjacent arrows, which are colored in blue, gives the AR-quiver of type $E_7$.
Then removing the $\tau$-orbits $(3,4)$, $(5,6)$ and $(13,14)$ and all adjacent arrows, which are colored in red, yields the black subquiver
which is the AR-quiver of $D_4$.
\end{proof}

\section{From small resolutions to cluster-tilting objects via NCCRs}\label{sec:fromsmallres}

\begin{Setup} \label{Set:VdB}
    Let $S \cong \C\llbracket z_1, \ldots, z_n \rrbracket/I$ be a complete local Gorenstein threefold with an isolated singularity. 
\end{Setup}

In particular, \Cref{Set:VdB} implies that $S$ is normal, since it is Cohen--Macaulay and regular in codimension one.

We explain the connections between the existence of a special type of 
commutative resolutions of singularities $Y \to \Spec S$, certain noncommutative resolutions of singularities $\Lambda$ of $S$ and a special type of generators (called \emph{cluster-tilting objects}, see \Cref{defn:ctsmall}) in the singularity category $\Dsg(S)$ of $S$. These are well-known results due to Van den Bergh \cite{VandenBergh04}, Iyama \cite{Iyama07, Iyama07a} and Iyama--Reiten \cite{IyamaReiten} and hold in greater generality. For simplicity, we restrict to the setup needed in this article.  

\begin{defn}\label{D:Small}
A projective morphism $\pi\colon Y \to \Spec S$, such that $Y$ is {regular} {and the induced morphism $\pi \colon Y \setminus \pi^{-1} ( \mathfrak{m} ) \to \Spec(S) \setminus \{ \mathfrak{m} \}$ is an isomorphism}, is called a \emph{resolution of singularities}.
It is called \emph{small}, if additionally all fibers are at most $1$-dimensional.
\end{defn}

\begin{rem}\label{rem:loc-free-sheaves}
    We note that any projective morphism $Y \to \Spec ( A)$ for any commutative noetherian ring $A$ 
    has \emph{enough locally free sheaves}, meaning that any coherent sheaf on $Y$ is a quotient of a locally free sheaf on $Y$.
    Indeed, the fact that $Y$ has an ample line bundle implies that there are enough locally free sheaves on $Y$ by a famous theorem of Serre (see e.g.\ \cite[Chapter II, Theorem 5.17]{hartshorne}). 
\end{rem}

\begin{rem}\label{rem:crepant-res}
    A small resolution as in \Cref{D:Small} is \emph{crepant}, i.e.\ $\pi^*\co_S \cong \omega_Y$, where $\omega_Y$ is the canonical sheaf on $Y$ and $\co_S \cong \omega_S$ is the canonical sheaf on $\Spec(S)$ since $S$ is a 
complete local Gorenstein algebra.
We note that, as a crepant resolution of $\Spec(S)$, $Y$ is Calabi--Yau, i.e.\ $\co_Y \cong \pi^*\co_S \cong \omega_Y$.

     Indeed, set $U = Y \setminus \pi^{-1}(\mathfrak{m})$. The sheaf $\pi^*\omega_S \cong  \pi^*\co_S$ is a line bundle since $\co_S$ is a line bundle. %
     As %
     the restriction %
     $\pi\colon U \to \Spec (S) \setminus \{\mathfrak{m}\}$ is an isomorphism, we have 
    \begin{align*}
        \pi^*\omega_S|_U \cong \omega_U \cong \omega_Y|_U .
    \end{align*}
    As the resolution is small, we have $\mathrm{codim}_Y(Y \setminus U)\geq 2$, this can be used to show that an extension of $\pi^*\omega_S|_U$ to a line bundle on $Y$ is unique, 
    using the %
    isomorphism between the divisor class groups of $Y$ and $U$ (cf.\ e.g.\ \cite[Chapter II, Proposition 6.5(b)]{hartshorne}). Hence, $\pi^*\omega_S \cong \omega_Y$ showing crepancy. 
\end{rem}

Our goal is to explain \Cref{P:ClusterTiltingObject}, which we recall here:

\begin{thm}\label{P:ClusterTiltingObject2}
   Let $S \cong \C\llbracket z_1, \ldots, z_n \rrbracket/I$ be a $3$-dimensional Gorenstein \mbox{$\C$-algebra} with an isolated singularity and assume that $\Spec S$ has a small resolution. 

   Then $\Dsg(S)$ contains a cluster-tilting object. Moreover, if the quiver of every cluster-tilting object in $\Dsg(S)$ contains no loops 
   or $2$-cycles, then $S$ is an $A_1$-singularity.
\end{thm}

The proof proceeds in several steps.
We first show the existence of an NCCR. Using this we get a cluster-tilting object in $\Dsg(S)$. Next, we describe the essential parts of its quiver. Finally, we use this to `characterize' $A_1$-singularities among threefolds with small resolutions. 

\subsection{From small resolutions to NCCRs}

The following result is due to Van den Bergh \cite{VandenBergh04}.
\begin{thm} \label{T:VdB} 
Let $S \cong \C\llbracket z_1, \ldots, z_n \rrbracket/I$ be a $3$-dimensional Gorenstein \mbox{$\C$-algebra} with an isolated singularity. If $\pi\colon Y \to \Spec S$ is a \emph{small} resolution of singularities, then there is a tilting bundle $\ct \in \Db(Y)$. In particular, there is a triangle
equivalence
\begin{align}\label{E:TVdB}
\Db(Y) \cong \Db(\End_Y(\ct)).
\end{align}
\end{thm}
\begin{proof}
    For the convenience of the reader, we include a sketch of the argument.

    We first produce a generator $\cg$ of $\Db(Y)$ and then `(partially) mutate' it, to get the tilting bundle $\ct$. 
    
    Let $\cl$ be an ample line bundle on $Y$. Since $\pi$ has fibers of dimension at most $1$, work of Bondal--Van den Bergh (cf.\ \cite[Lemma 3.2.2]{VandenBergh04}) shows that
    \begin{align*}
        \cg=\co_Y \oplus \cl
    \end{align*}
    generates $\Db(Y)$. We next study the graded Ext-algebra of $\cg$. Combining $\co_Y \cong \omega_Y $ (cf.\ Remark \ref{rem:crepant-res}), with the Grauert--Riemenschneider vanishing theorem shows that 
    \begin{align*}
        \mathsf{R}^q\pi_*\co_Y \cong \mathsf{R}^q\pi_*\omega_Y =0
    \end{align*}
    for $q>0$ which, since $\Spec(S)$ is affine, implies
    \begin{align*}
        \Ext^q_Y(\co_Y, \co_Y) \cong H^q(Y, \co_Y) \cong H^0(\Spec(S), \mathsf{R}^q\pi_*\co_Y) =0.  
    \end{align*}
    Since tensoring with a line bundle is an autoequivalence of $\Db(Y)$, we also deduce
    \begin{align*}
        \Ext^q_Y(\cl, \cl) =0
    \end{align*}
    for all $q>0$. Now we claim that $\Ext^q_Y(\co_Y, \cl)=0$ for all $q>0$.
    Since $\cl$ is ample, it is generated by global sections and we get a short exact sequence
    \begin{align*}
        0 \to \ck \to \co_Y^l \to \cl \to 0.
    \end{align*}
    Applying $\Hom_Y(\co_Y, -)$ yields a long exact sequence and using\footnote{Since we have $\mathsf{R}^{>1}\pi_*\cf = 0$ for all coherent sheaves $\cf$ by our assumption on the fibers of $\pi$, \cite[Chapter III, Corollary 11.2]{hartshorne}.} 
    \begin{align} \label{E:ExtVanish}
        \Ext^{>1}_Y(\co_Y, \cf) \cong H^{>1}(Y, \cf) \cong H^0(\Spec(S), \mathsf{R}^{>1}\pi_*\cf) =0
    \end{align}
    for $\cf=\ck$ and $\cf= \cl$ shows the claim.

    Since $\Ext^1_Y(\cl, \co_Y) \neq 0$ in general, $\cg$ is not tilting. However, there is a well-known construction yielding a tilting bundle in this situation. Consider the universal extension
    \begin{align} \label{E:univExt}
        0 \to \co_Y^k \to \cm \to \cl \to 0,
    \end{align}
    constructed from a minimal set of generators $\epsilon_1, \ldots, \epsilon_k$ of the finitely generated $R$-module $\Ext^1_Y(\cl, \co_Y)$. In particular, this means that the map 
    \begin{align*}
        \Hom_Y(\co_Y^k, \co_Y) \to \Ext^1_Y(\cl, \co_Y) \to 0
    \end{align*}
    induced by \eqref{E:univExt} is surjective. In combination with the Ext-vanishing between $\cl$ and $\co_Y$ that we have established above and with 
    \begin{align*}
        \Ext^{>1}_Y(\cl, \co_Y) \cong \Ext^{>1}_Y(\co_Y, \cl^{-1}) = 0
    \end{align*} (by \eqref{E:ExtVanish}),  this implies that 
    \begin{align*}
        \Ext^{>0}_Y(\co_Y \oplus \cm, \co_Y \oplus \cm) =0,
    \end{align*}
    cf.\ e.g.\ \cite[Lemma 3.1]{HillePerling}. Hence, $\ct \cong \co_Y \oplus \cm$ is tilting, completing the proof.
\end{proof}

\begin{rem}
    The bundle $\cm$ in \eqref{E:univExt} will be not be indecomposable in general.
\end{rem}

 The derived equivalence \eqref{E:TVdB} suggests to view $\End_Y(\ct)$ as a noncommutative analogue of a crepant resolution of $S$. The following definition of Van den Bergh axiomatizes such noncommutative resolutions, cf.\ \mbox{\cite[Definition 4.1 and Lemma 4.2]{NCCR}.}

\begin{defn}\label{D:NCCR} Let $S$ be as in \Cref{Set:VdB}. The $\C$-algebra $\Lambda=\End_S(M)$ is a  \emph{noncommutative crepant resolution (NCCR)} of $S$ if
\begin{enumerate}[label={(\alph*)}]
    \item $M$ is a reflexive\footnote{Since $S$ is Gorenstein, all maximal Cohen--Macaulay $S$-modules are reflexive.} $S$-module.
    \item The $\C$-algebra $\Lambda$ is maximal Cohen--Macaulay as an $S$-module.
    \item $\gldim \Lambda < \infty$.
\end{enumerate}
\end{defn}

Van den Bergh showed that tilting bundles as in Theorem \ref{T:VdB} give rise to NCCRs in the following way.

\begin{thm} \label{T:VdB2} 
We keep the assumptions from \Cref{T:VdB}. The pushforward $\pi_*\ct$ of the tilting bundle $\ct$ on $Y$ is a maximal Cohen--Macaulay $S$-module. There is an isomorphism of $\C$-algebras
\begin{align} \label{E:IsoPushf}
\End_Y(\ct) \cong \End_S(\pi_*\ct) 
\end{align}
and these $\C$-algebras are NCCRs of $S$.
\end{thm}
\begin{proof}
Set $\Lambda=\End_S(\pi_*\ct)$. In view of \cite[Proposition 3.2.10]{VandenBergh04} and its proof it remains to show that $\gldim \Lambda$ is finite.
Since $Y$ is {regular and since $Y$ has enough locally free sheaves, see Remark \ref{rem:loc-free-sheaves}}, we have 
\begin{align*}
\Perf(Y) = \Db(Y)
\end{align*}
and, using the tilting equivalence \Cref{T:VdB} and \eqref{E:IsoPushf}, we deduce
\begin{align*}
\Perf(\Lambda) = \Db(\Lambda).
\end{align*}
In particular, every finitely generated $\Lambda$-module has finite projective dimension.
Since $S$ is complete, $\Lambda$ is semiperfect -- in particular, there are only finitely many isomorphism classes of simple $\Lambda$-modules. Then the global dimension of $\Lambda$ is the maximum of the projective dimensions of the simples, which completes the proof.
\end{proof}

\subsection{From NCCRs to cluster-tilting objects}

\begin{thm}\label{T:Iyama} Let $S$ be as in \Cref{Set:VdB}. 

If $M \in \MCM(S)$ is a maximal Cohen--Macaulay module such that $\Lambda=\End_S(M)$ is an NCCR of $S$, then $M$ is a cluster-tilting object in $\Dsg(S)$.

\end{thm}
\begin{proof}
This follows from Theorem 8.18(1) in Iyama--Reiten \cite{IyamaReiten}, for the convenience of the reader, we translate their statement into our setting. 

Firstly, we take $\Lambda=S$ in their statement -- we note that $S$ is $3$-$CY^-$ by \cite[Proposition 3.10]{IyamaReiten}, since $S$ is local Gorenstein and of Krull dimension $3$.

Now, \cite[Lemma 2.23]{IyamaWemyss14} shows that condition (2) in the definition of an NCCR in \cite[page 1138]{IyamaReiten} is satisfied by NCCRs as in Definition \ref{D:NCCR}. Their condition (1) holds since for Gorenstein rings $S$ maximal Cohen--Macaulay modules $M$ are reflexive. Moreover, we show that $M_\mathfrak{p}$ is free for all height one prime ideals $\mathfrak{p}\subseteq S$, in particular, it is a generator. Firstly, $M_\mathfrak{p}$ is maximal Cohen--Macaulay for any prime ideal $\mathfrak{p} \subseteq S$. Since $S$ has isolated singularities $S_\mathfrak{p}$ is regular for all height one prime ideals $\mathfrak{p}$. As the only maximal Cohen--Macaulay modules over regular local rings are free, we have shown that NCCRs in our sense are NCCRs in the sense of \cite[page 1138]{IyamaReiten}.

Thus \cite[Theorem 8.18(1)]{IyamaReiten} shows that $M$ is a maximal $1$-orthogonal $S$-module in $\MCM(S)$, cf.\ \cite[page 1143]{IyamaReiten}. Using Buchweitz's equivalence $\ul{\MCM}(S) \cong \Dsg(S)$ and the fact that $\ul{\Hom}_S(\Omega(X), Y) \cong \Ext^1_S(X, Y)$ for $X, Y \in \MCM(S)$ shows that $M$ defines a cluster-tilting object in $\Dsg(S)$. 
\end{proof}

\begin{cor}\label{C:ClusterTiltingObject}
In the setup of \Cref{T:VdB2}, $\pi_*\ct$ is a cluster-tilting object in $\Dsg(S)$.
\end{cor}
\begin{proof}
    This follows from \Cref{T:Iyama} together with \Cref{T:VdB2}. 
\end{proof}

\subsection{The quiver of the cluster-tilting object $\pi_*\ct$ in $\Dsg(S)$}
We describe parts of the quiver of the finite-dimensional $\C$-algebra $\ul{A}=\ul{\End}_S(\pi_*\ct)$. These results are well-known, cf.\ \cite[Theorem 2.15]{WemyssFlops}.

Recall that $\pi\colon Y \to \Spec(S)$ is a small resolution. Let $C=\pi^{-1}(\mathfrak{m})$ be the exceptional curve and let $C^{\mathrm{red}}$ be the corresponding reduced scheme. We decompose $C^{\mathrm{red}}= \bigsqcup C_i$ into irreducible components $C_i$. 
Then $C_i \cong \P^1$ for all $i$ and the intersection of $C_i$ and $C_j$ for $i\neq j$ is either empty or transversal, cf.\ \cite[Lemma 3.4.1]{VandenBergh04}.
\begin{prop}\label{P:QuiverClusterT}
There is an isomorphism of $\C$-algebras
\begin{align*}
    \ul{\End}_S(\pi_*\ct) \cong \C Q/I,
\end{align*}
where $Q$ is a quiver with vertices labelled by the irreducible curves $C_i$ and $I \subseteq \C Q$ is an admissible ideal.

If $C_i \neq C_j$ intersect, then $Q$ contains a subquiver\footnote{This subquiver is not necessarily \emph{full}, that is there might be loops or additional arrows between $C_i$ and $C_j$.}
\begin{equation}\label{E:2Cyc}
\begin{tikzpicture}[description/.style={fill=white,inner sep=2pt}, scale=0.75, baseline=(current  bounding  box.center)]
\matrix (n) [matrix of math nodes, row sep=2em,
                 column sep=2em, text height=1.5ex, text depth=0.25ex,
                 inner sep=0pt, nodes={inner xsep=0.3333em, inner
ysep=0.3333em}] at (0, 0)
    {   C_i && C_j \\ };
\draw[->] ($(n-1-1.east)+(0,0.4)$) to [bend left=25] ($(n-1-3.west) +(0,0.4)$);
\draw[<-] ($(n-1-1.east)+(0,-0.4)$) to [bend right=25] ($(n-1-3.west) +(0,-0.4)$);
\end{tikzpicture}
\end{equation}
\end{prop}
\begin{proof}
    Since $\ul{A}:=\ul{\End}_S(\pi_*\ct)$ is a finite-dimensional $\C$-algebra, there is an isomorphism
    \begin{align*}
    \ul{\End}_S(\pi_*\ct) \cong \C Q/I,
\end{align*}
where $Q$ is a finite quiver and $I \subseteq \C Q$ is an admissible ideal. The vertices of $Q$ are in bijection with the simple $\ul{A}$-modules. The number of arrows from vertex $i$ to vertex $j$ is equal to the dimension of $D\Ext^1_{\ul{A}}(S_i, S_j)$ for simple $\ul{A}$-modules $S_i$, $S_j$. By definition, 
\begin{align*}
    \ul{A} \cong \End_S(\pi_*\ct)/(e_S)
\end{align*}
where $e_S$ is the idempotent endomorphism of $\pi_\ast \ct$ which is given by the projection $\pi_*\ct \to S \to \pi_*\ct$. By \Cref{L:SerreSub}, the simple $\ul{A}$-modules correspond to simple $A=\End_S(\pi_*\ct)$-modules that are annihilated by $e_S$ and it suffices to understand $\Ext^1_A(S_i, S_j)$ for these simples. Van den Bergh shows that under the equivalence \eqref{E:TVdB} the simple $A$-modules $S_i$ annihilated by $e_S$ correspond to shifts of sheaves $\Sigma\co_{C_i}(-1)$ on the resolution $Y$, where, as above, the $C_i$ are the irreducible components of $C$, see \cite[Proposition 3.5.7]{VandenBergh04}. Now, if $C_i \neq C_j$ intersect, Proposition \ref{P:Neb} and \eqref{E:TVdB} show 
\begin{align*}
   0  \neq  \Hom_{\Db(Y)}(\Sigma\co_{C_i}(-1), \Sigma^2\co_{C_j}(-1)) \cong \Ext^1_A(S_i, S_j). 
\end{align*}
By symmetry of the components $C_i$ and the discussion above, this completes the proof.
\end{proof}

\begin{prop}\label{P:Neb}
   In the notation of \Cref{P:QuiverClusterT}, assume that $C_i \neq C_j$ 
   intersect. Then 
\begin{align*}
   \Hom_{\Db(Y)}(\Sigma\co_{C_i}(-1), \Sigma^2\co_{C_j}(-1)) \neq  0. 
\end{align*}
\end{prop}
\begin{proof}
     Consider the commutative diagram
\begin{equation*}
\xymatrix{
    \{ p \}\ar@{->}[d]\ar@{->}[r] & C_i \ar@{->}[d]^{\iota_i}\\
 C_j \ar@{->}[r]^{\iota_j} & Y} 
\end{equation*}
of regular embeddings.
    By \cite[Theorem A.1]{CKS} there is a spectral sequence
\[
E_2^{p,q} = \Ext^{p} ( \co_{ \{p \} } , \mathcal{E}^{\vee}|_{\{p\}}\otimes \mathcal{F}|_{\{p\}} \otimes \wedge^1\mathcal{N}_{\{p\} / C_j } \otimes \wedge^{1-q} E ) \Rightarrow \Ext^{p+q}_Y ({\iota_i}_\ast\mathcal{E} , {\iota_j}_\ast\mathcal{F} ),
\]    
where $\mathcal{E}$ and $\mathcal{F}$ are vector bundles on $C_i$ and $C_j$, respectively, $\mathcal{N}_{\{p\} / C_j }$ is the normal bundle of $\{p\}\hookrightarrow C_j$ and $E$ is the excess bundle
\[
E = \frac{\mathcal{T}_Y|_{p}}{\mathcal{T}_{C_i}|_{p}\oplus \mathcal{T}_{C_j}|_{p}}.
\]
In this case, it is clear that $\mathcal{N}_{\{p\} / C_j } \simeq \co_{\{p \}}$ and that $E \simeq \co_{\{p \}}$.
Plugging this in into the spectral sequence and substituting $\mathcal{E}=\co_{C_i}$ and $\mathcal{F}=\co_{C_j}$, we get
\[
E_2^{p,q} = \Ext^{p} ( \co_{ \{p \} } ,  \wedge^{1-q} \co_{\{p\}} ) \Rightarrow \Ext^{p+q}_Y (\co_{C_i} , \co_{C_j} ).
\]
Thus, we read off that $\Ext^{1}_Y (\co_{C_i} , \co_{C_j} ) \simeq \C$. In particular, 
\begin{align*}
   \Hom_{\Db(Y)}(\Sigma\co_{C_i}(-1), \Sigma^2\co_{C_j}(-1)) \cong \Ext^{1}_Y (\co_{C_i} , \co_{C_j} ) \neq  0. \tag*{\qedhere}
\end{align*}%
\end{proof}

\begin{cor}\label{C:SingleVertex} In the setup and notation of \Cref{P:QuiverClusterT}.
    If the quiver $Q$ contains no loops or $2$-cycles as in \eqref{E:2Cyc}, then $Q$ is just a single vertex $C=C_1$ without loops -- in particular, $C$ is irreducible.
\end{cor}
\begin{proof}
    If $Q$ contains at least $2$ vertices, then the exceptional curve $C$ has at least $2$ irreducible components. Since $C$ is connected, \Cref{P:QuiverClusterT} and the discussion preceding it show that $Q$ contains a $2$-cycle. It follows that $Q$ has only one vertex without loops.
\end{proof}

For completeness, we include the following statement.
\begin{lem}\label{L:SerreSub}
    Let $\Lambda$ be a (semi-perfect) ring and let $e \in \Lambda$ be an idempotent.

    The natural inclusion $\mod \Lambda/(e) \to \mod \Lambda$ induces a bijection
    \[
    \{\text{Simple $\Lambda/(e)$-modules}\} \ \xleftrightarrow{\text{$1$:$1$}} \ 
    \{\text{Simple $\Lambda$-modules $S$ with $Se=0$} \}
    \]
    up to isomorphism. Moreover, since $\mod \Lambda/(e) \subseteq \mod \Lambda$ is extension closed 
    \[
    \Ext^1_{\Lambda/(e)}(M, N) \cong \Ext^1_{\Lambda}(M, N)
    \]
    for all $\Lambda/(e)$-modules $M$ and $N$.
\end{lem}

\subsection{Proof of \Cref{P:ClusterTiltingObject}}
We give a `purely algebraic' argument and comment on an alternative more geometric proof in \Cref{rem:geom-argument}.

We have seen in \Cref{C:ClusterTiltingObject} that $\Dsg(S)$ has a specific cluster-tilting object $T=\pi_*\ct$. Our assumption that there are no loops and $2$-cycles in the quiver $Q$ of $T$ and \Cref{C:SingleVertex} imply that $Q$ consists of a single vertex with no loops.
In particular, there are $\C$-algebra isomorphisms
\begin{align*}
    \End_{\Dsg(S)}(T) \cong \ul{\End}_S(T) \cong \C Q \cong \C.
\end{align*}
Summing up, $\Dsg(S)$ is a $2$ Calabi--Yau category  (cf.\ \cite[Theorem 4.28]{WemyssLecture}) with a cluster-tilting object $T$ with \emph{hereditary} endomorphism algebra. As a very special case of Keller--Reiten's recognition theorem \cite[Theorem in Section 2.1]{KellerReiten08} for such categories, we obtain a triangle equivalence
\begin{align}\label{E:KR1}
    \Dsg(S) \cong \cc(\C),
\end{align}
where the latter denotes the cluster category of the hereditary $\C$-algebra $\C$. The singularity category of $S' \coloneqq \C\llbracket z_0, \dots, z_3 \rrbracket/(z_0z_1 - z_2 z_3)$ also has a cluster-tilting object $T'$ with endomorphism algebra $\C$, see \cite[Proposition 2.4]{BIKR} together with Kn\"orrer periodicity \cite{Knoerrer} to reduce from threefolds to curves.
Therefore, applying \cite[Theorem in Section 2.1]{KellerReiten08}
to $\Dsg(S')$ and using \eqref{E:KR1} yields triangle equivalences
\begin{align*}
    \Dsg(S) \cong \Dsg(S').
\end{align*}
Therefore $S \cong S'$ is nodal by Lemma \ref{L:Yoshino}.
\hfill \qedsymbol

\begin{rem}\label{rem:geom-argument}
Here is a more geometric argument for parts of the proof of \Cref{P:ClusterTiltingObject}. If the quiver $Q$ of the cluster-tilting object $\pi_*\ct$ consists of a single vertex, then the small resolution has an irreducible exceptional curve $C$, by \Cref{C:SingleVertex}. 
If this vertex does not have loops, then the  normal bundle of $C$ is isomorphic to $\mathcal{O}(-1) \oplus \mathcal{O}(-1)$ by \cite[2.15]{WemyssFlops}. By Reid's \cite{Reid83} this implies that $S$ is an $A_1$-singularity.
\end{rem}
%
%
%
%
%
%
%
%
%
%
%
%
%
%
%
%
%
%
%
%
%
%
%
%
%
%
%
%
%
%
%
%
%
%
%
%
%
%
%
%
%
%

\section{A categorical `No-loop Theorem' and consequences for cluster-tilting subcategories}\label{sec:homresults}

\subsection{Setting the stage}

We suggest to recall \Cref{defn:modA,defn:sink,defn:covering,defn:quiver,defn:ARquiver} before reading \Cref{sec:homresults}.
We consider sufficiently nice categories, satisfying a subset of the following conditions.

\begin{Setup}\label{setup:algebraicsetup}
   Let $\kk$ be an algebraically closed field and $\ca$ be a $\kk$-linear category and $\mod \ca$ be the category of finitely presented contravariant functors $\cA \to \mathrm{Ab}$, see \Cref{defn:modA} and \Cref{rem:k-linear}.
   We consider the following additional \mbox{conditions on $\ca$.}
   \begin{enumerate}[label={{(\alph*)}}]
      \item $\ca$ is a Krull--Schmidt category.\label{ks}
      \item $\ca$ is $\Hom$-finite, i.e.\ $\dim_\kk \ca(X,Y) < \infty$ for all $X,Y \in \cA$.\label{hf}
      \item $\mod \ca$ is an abelian category.\label{ab}
      \item $\mod \ca$ has finite global dimension.\label{fg}
   \end{enumerate}
\end{Setup}

Notice, \Cref{setup:algebraicsetup}\ref{ab} is equivalent to $\cA$ having weak kernels, cf.\ \cite[Section III.2]{AuslanderQueen}.

For our applications the most relevant example is the category of Gorenstein projective modules over a finite-dimensional Gorenstein $\kk$-algebra (cf.\ \cref{defn:Goring,D:GProj}).
The following is a direct consequence of \Cref{T:Auslander}.

\begin{cor}\label{cor:GPisreallynice}
   For a finite-dimensional Gorenstein $\kk$-algebra $R$ the category $\GP(R)$ of Gorenstein projective modules over $R$ satisfies \Cref{setup:algebraicsetup}\ref{ks}-\ref{fg}.
\end{cor}

\subsection{Quivers of $\Hom$-finite Krull--Schmidt categories} 
Assume that $\ca$ is a category satisfying \Cref{setup:algebraicsetup}\ref{ks}-\ref{ab}.
We present and recall generalisations of results which are well-known for $\mod A$ where $A$ is a finite-dimensional $\kk$-algebra.

\begin{defn}\label{defn:simplefunc}
The contravariant functor $S_X(-) \coloneqq \ca(-, X)/\rad_{\ca}(-, X)$ for an indecomposable object $X \in \cA$ is called the \emph{simple functor at $X$}.
\end{defn}

Recall the definition of sink morphisms from \Cref{defn:sink}.

\begin{rem}\label{rem:sink}
Under the assumptions of \Cref{setup:algebraicsetup}\ref{ks}, one can show that a morphism $f \colon Y \to X$ is a sink morphism for $X$ if and only if 
\begin{enumerate}[label={{(\alph*)}}]
    \item $X$ is indecomposable, 
    \item $\im \ca(-,f) = \rad_{\ca}(-,X) \subset \ca(-,X)$ and 
    \item $\ker \ca(-,f) \subset \rad_{\ca}(-,Y)$.
\end{enumerate}
Furthermore, if $S_X \in \mod \cA$ and $X_1 \xrightarrow{\smash{d_1}} X_0 \xrightarrow{\varepsilon} S_X \to 0$ is the start of a projective resolution as in \Cref{cor:minres} then $X_0 \cong X$ and $d_1$ is a sink morphism for $X$. Hence, an object $X \in \ind \ca$ admits a sink morphism if and only if $S_X \in \mod \ca$.
\end{rem}

\begin{rem}\label{rem:simple}
   Suppose $X \in \ind \ca$.
   Then $S_X(\ca \setminus \add X) = 0$ holds by \Cref{lem:radical}, and
   $S_X(X) \cong \kk$ as $S_X(X) = \End_\ca (X)/\rad \End_{\ca} (X)$ is a finite-dimensional division $\kk$-algebra, where $\kk$ is algebraically closed.
   Hence, for $Y \in \ca$ the $\kk$-dimension of $S_X(Y) = \ca(Y,X)/\rad_\ca(Y,X)$ is the multiplicity of $X$ as a direct summand of $Y$. 
   Similarly, one can show that for $Y \in \ca$ the $\kk$-dimension of $\ca(X,Y)/\rad_\ca(X,Y)$ and the multiplicity of $X$ as a direct summand of $Y$ agree.
\end{rem}

The following generalises \cite[Lemma III.2.12]{ASS} to functor categories.

\begin{lem}\label{lem:extar}
    Suppose $X \in \ind \ca$ admits a sink morphism \mbox{$d_1 \colon X_1 \to X$}.
    Then the following numbers agree and are finite for any $Y \in \ind \cA$.
    \begin{enumerate}[label={(\alph*)}]
        \item The multiplicity of $Y$ as a summand of $X_1$.\label{item:sink}
        \item The $\kk$-dimension of $\rad_\ca(Y,X)/\rad_\ca^2(Y,X)$.\label{item:ar}
    \end{enumerate}
    Moreover, if $Y$ admits a sink morphism then the above numbers agree with 
    \begin{enumerate}[label={(\alph*)}, resume]
        \item the $\kk$-dimension of $\Ext_{\mod \ca}^1(S_X, S_Y)$. \label{item:ext}
    \end{enumerate}
\end{lem}
\begin{proof}
    Since $\ca$ is Krull-Schmidt, \ref{item:sink} is finite. 
    As sink morphisms are unique up to isomorphism, \ref{item:sink} independent of the choice of the sink morphism.
    That \ref{item:sink} and \ref{item:ar} agree follows from \cite[Corollary 3.7]{bautista}.
    
    We show that \ref{item:ext} and \ref{item:sink} agree.
    Pick the start of a projective resolution
    \[\begin{tikzcd}[ampersand replacement=\&]
        {\ca(-,X_2)} \& {\ca(-,X_1)} \& {\ca(-,X)} \&[-.3cm] {S_X} \&[-.5cm] 0
        \arrow["{d_1\cdot-}", from=1-2, to=1-3]
        \arrow["\varepsilon", from=1-3, to=1-4]
        \arrow[from=1-4, to=1-5]
        \arrow["{d_2\cdot-}", from=1-1, to=1-2]
    \end{tikzcd}\]
    with $d_1 \in \rad_{\ca}(X_1, X)$ and $d_2 \in \rad_{\ca}(X_2, X_1)$, see \Cref{cor:minres}.
    Then $d_1$ is a sink morphism for $X$ by \Cref{rem:sink}.
    By the Yoneda lemma, there is a commutative diagram
    \[\scalebox{.95}{\begin{tikzcd}[ampersand replacement=\&]
        {\Hom_{\mod \ca}(\ca(-,X), S_Y)} \&[-.3cm] {\Hom_{\mod \ca}(\ca(-,X_1), S_Y)} \&[-.3cm]     {\Hom_{\mod \ca}(\ca(-,X_2), S_Y)} \\
        {S_Y(X)} \& {S_Y(X_1)} \& {S_Y(X_2)}
        \arrow[from=1-3, to=2-3, "\cong"]
        \arrow[from=1-2, to=2-2, "\cong"]
        \arrow[from=1-1, to=2-1, "\cong"]
        \arrow["{S_Y(d_2)}", from=2-2, to=2-3]
        \arrow[from=1-2, to=1-3, "d_2^\ast"]
        \arrow[from=1-1, to=1-2, "d_1^\ast"]
        \arrow["{S_Y(d_1)}", from=2-1, to=2-2]
    \end{tikzcd}}\]
    where $d_k^\ast = \Hom_{\mod \ca}(\ca(-,d_k),S_Y)$ for $k=1,2$.
    Then $S_Y(d_k) = 0$ for $k= 1,2$ as $d_k$ is in the radical. 
    Hence, $d_k^\ast = 0$ for $k = 1,2$ and the $\kk$-dimension of $\Ext^1_{\mod \cA}(S_X, S_Y)$ and the $\kk$-dimension of $S_Y(X_1)$ agree. 
    The latter is the multiplicity of $Y$ as a summand of $X_1$ by \Cref{rem:simple}.
    This shows that \ref{item:ext} and \ref{item:sink} agree.
 \end{proof}

Recall the definition of the quiver of $\ca$, see \Cref{defn:quiver}. We have the following.
\begin{lem}\label{lem:cyclequot}
     Suppose $\Phi \colon \ca \to \cb$ is a full, essentially surjective additive functor and let $X,Y \in \Phi^{-1}(\ind \cb) \cap \ind \ca$. 
     If there is no arrow $X \to Y$ in the quiver of $\ca$ then there is no arrow $\Phi(X) \to \Phi(Y)$ in the quiver of $\cb$.
\end{lem}
\begin{proof}
    The functor $\Phi \colon \ca \to \cb$ induces a $\kk$-vector space epimorphism 
    \[\rad_{\ca}(X,Y)/\rad_{\ca}^2(X,Y) \to \rad_{\cb}(\Phi(X),\Phi(Y))/\rad^2_{\cb}(\Phi(X),\Phi(Y))\] 
    for $X,Y \in \ind \ca$ using \Cref{lem:radicalfunctor}\ref{item:radicaliso}.
\end{proof}

\subsection{Restrictions of functor categories}\label{sec:locglo}
In this section we assume that $\ca$ is an idempotent complete category satisfying \Cref{setup:algebraicsetup}\ref{ab}, but we can safely replace $\kk$ by just a commutative ring.
We start with a special case of the Yoneda lemma.

\begin{lem}\label{lem:yoneda}
    Let $P \in \ca$ and $A := \End_\ca(P)$. 
    Given $G \in \mod \ca$ consider the family
    \begin{align*}
        \rho_{X} \colon \Hom_{\mod \ca}(\ca(-, X), G(-)) &\to \Hom_{A}(\ca(P, X), G(P)),\, \alpha \mapsto \alpha_P
    \intertext{of morphisms for $X \in \add P$. This family yields a natural isomorphism} 
    \rho \colon \Hom_{\mod \ca}(\ca(-, -|_{\add P}), G(-)) &\xrightarrow{\simeq} \Hom_{A}(\ca(P, -|_{\add P}), G(P))
    \end{align*}
    of contravariant functors $\add P \to \Mod \kk$.
\end{lem}
\begin{proof}
    The defined transformation is clearly $\kk$-linear and natural on $\add P \subset \ca$.
    Hence, to check that $\rho$ is a natural isomorphism we only have to show that $\rho_P$ is an isomorphism, using $\kk$-linearity of both functors. However, $\rho_P$ composed with the canonical isomorphism $\operatorname{eval}_{\id} \colon \Hom_{A}(A, G(P)) \to G(P),\, f \mapsto f(\id)$ is the Yoneda isomorphism, showing that $\rho_P$ is itself an isomorphism. 
\end{proof}

The following lemma allows us to restrict extension groups of functor categories to module categories over finite-dimensional $\kk$-algebras.

\begin{lem}\label{lem:localise}
    For any $F \in \mod \ca$ with $\prdim_{\mod \ca} F < \infty$ there exists an object $P_F \in \ca$ such that for any $P \in \ca$ with $\add P \supset \add P_F$ we have 
    \begin{enumerate}[label={(\alph*)}]
        \item $F(P)$ has finite projective dimension as $A$-module and \label{item:finprdim}
        \item $\Ext^{i}_{\mod \ca}(F, G) \cong \Ext^i_{A}(F(P), G(P))$ for every $G \in \mod \ca$ and all $i \in \N$,\label{item:extiso}
    \end{enumerate}
    where $A \coloneqq \End_{\ca}(P)$. 
\end{lem}
\begin{proof}
    Because $\prdim F < \infty$ we can find a finite projective resolution of $F$ in $\mod \ca$. 
    Using the Yoneda lemma we may assume that this finite projective resolution of $F$ is of the form $\ca(-, X_{\bullet})$ for a bounded complex $X_{\bullet}$ with terms in $\ca$.
    
    Let $P_F \coloneqq \bigoplus_{k\in \N} X_{k}$, which is well-defined since $X_\bullet$ is bounded. 
    Suppose we are given $P \in \ca$ with $\add P \supseteq \add P_F$ and $A \coloneqq \End_{\ca}(P)$.
    Since $X_{k}$ is in $\add P$ for $k \in \N$ we know that $\ca(P,X_{k})$ is a projective $A$-module. 
    This shows that $\ca(P, X_{\bullet})$ is a finite projective resolution of $F(P)$ in $\mod A$ and therefore $F(P)$ is of finite projective dimension as $A$-module.
    This establishes part \ref{item:finprdim}.
   
   Let $G \in \mod \ca$. Using that $X_\bullet$ has terms in $\add P$ and \Cref{lem:yoneda} the complexes $\Hom_{\mod \ca}(\ca(-, X_{\bullet}), G(-))$ and $\Hom_{A}(\ca(P, X_{\bullet}), G(P))$ are isomorphic.
   Hence, the two complexes have isomorphic homology, i.e.\ $\Ext_{\mod \ca}^i(F, G) \cong \Ext^i_{A}(F(P),G(P))$ for all $i \in \N$. 
   This establishes part \ref{item:extiso}.
\end{proof}

\subsection{The categorical `no-loop theorem'}\label{sec:noloop}
Recall the definition of the (AR-)quiver of $\ca$, see \Cref{defn:quiver,defn:ARquiver}.
To obtain our obstruction for loops in the AR-quiver of $\ul{\GP}_\cp(R)$ in \Cref{T:PropertiesDsgR}\ref{item:noloops} we use the following modification of the strong no-loop theorem, see Igusa--Liu--Paquette \cite[Theorem 2.2]{IgusaLiuPaquette}.
\begin{thm}\label{thm:noloopfun}
    Suppose $\ca$ is a category satisfying \Cref{setup:algebraicsetup}\ref{ks}-\ref{ab}.
    If $X \in \ind \ca$ admits a sink morphism and satisfies $\prdim_{\mod \ca} S_X < \infty$ then
    \begin{enumerate}[label={{(\alph*)}}]
       \item $\Ext^1_{\mod \cA}(S_X, S_X) = 0$.\label{item:noloopfun1}
       \item Equivalently, the quiver of $\ca$ has no loop at $X$.\label{item:noloopfun2}
    \end{enumerate}
\end{thm}
\begin{proof}
   Using \Cref{lem:localise} we can choose a basic object $P$ with $X \in \add P$ so that $\prdim_A S_X(P) < \infty$, where $A \coloneqq \End_{\ca}(P)$ is a finite-dimensional $\kk$-algebra, and $\Ext^1_{\mod \ca}(S_X,S_X) \cong \Ext^1_{A}(S_X(P), S_X(P))$. 
   Since $P$ is basic, $X$ appears with multiplicity one in $P$ and, by \Cref{rem:simple}, the module $S_X(P)$ has $\kk$-dimension $1$.
   Hence, the $A$-module $S_X(P)$ is simple.
   So, $\Ext^1_{\mod \ca}(S_X,S_X) \cong \Ext^1_A(S_X(P), S_X(P)) = 0$, by \cite[Theorem 2.2]{IgusaLiuPaquette}.
    This shows part \ref{item:noloopfun1} and part \ref{item:noloopfun2} follows by \Cref{lem:extar}.
\end{proof}

Recall also our conventions from \Cref{Not:exact,Not:frobenius,not:abuse,defn:quiver,defn:ARquiver}, used in the following lemma and the following two corollaries. 
We want to combine \Cref{thm:noloopfun} with the following lemma 

\begin{lem}\label{lem:exactpd2}
    Let $\ce_\cp$ be an Frobenius exact category satisfying \Cref{setup:algebraicsetup}\ref{ks}, \ref{ab}.
    \begin{enumerate}[label={{(\alph*)}}]
        \item If $X \in \ind \ce \setminus \ind \cp$ admits a sink morphism in $\ul{\ce}_{\cp}$ then it admits a sink morphism in $\ce$ and $\prdim_{\mod \ce} S_X \leq 2$.\label{item:exactpd2}
        \item If $\ul{\ce}_{\cp}$ has a Serre-functor then every $X \in \ind \ce \setminus \ind \cp$ admits a sink morphism in $\ce$ and $\prdim_{\mod \ce} S_X \leq 2$.\label{item:exactpd2serre}
    \end{enumerate}

\end{lem}
\begin{proof}
    \ref{item:exactpd2}: By \Cref{cor:sinklift} there is a sink morphism $f \colon Y \to X$ in $\ce$. There also exists a deflation $g \colon P \to X$ with $P \in \cp$, which is not a split epimorphism as $X \in \ind \ce \setminus \ind \cp$.
    As $f$ is a sink morphism, $g$ factors through $f$ which is therefore a deflation by \cite[Proposition 7.6]{Buehler}.
    Hence, there is a conflation 
    \begin{center}
        \begin{tikzcd} 
            K \ar[r, rightarrowtail] & Y \ar[r, twoheadrightarrow, "f"] & X 
        \end{tikzcd}
    \end{center}
     in $\ce_\cp$. Applying $\ce(-,-)$ to this conflation yields a projective resolution of $S_X$.
     
     \ref{item:exactpd2serre}: If $\ul{\ce}_{\cp}$ has a Serre-functor, then $\ul{\ce}_\cp$ has Auslander--Reiten triangles, in the sense of \cite[Section I.4]{HappelBook}, by \cite[Theorem A]{ReitenVandenBergh} and hence enough sink morphisms. The claim follows from \ref{item:exactpd2}.
\end{proof}

\begin{cor}\label{cor:noloops3}
   Suppose $\ce$ satisfies \Cref{setup:algebraicsetup}\ref{ks}-\ref{ab} and $\ce_\cp$ is a Frobenius exact structure on $\ce$ such that $\ul{\ce}_{\cp}$ has a Serre-functor. 
   
   Then the AR-quiver of $\ul{\ce}_{\cp}$ has no loops.
\end{cor}
\begin{proof}
    By \Cref{thm:noloopfun} and \Cref{lem:exactpd2} there are no loops at the non-projective vertices of the quiver of $\ce$. By \Cref{lem:cyclequot,lem:radicalfunctor} there are no loops at any vertices in the AR-quiver of $\ul{\ce}_{\cp}$.
\end{proof}

Also recall the notions related to Gorenstein rings from \Cref{defn:Goring,D:GProj}.

\begin{cor}\label{T:ARnoLoops}
   Let $R$ be a finite-dimensional Gorenstein $\kk$-algebra and $\GP_\cp(R)$ be a Frobenius exact structure on $\GP(R)$.
   
   Then the AR-quiver of $\underline{\GP}_{\cp}(R)$ has no loops.
\end{cor}
\begin{proof}
   By \Cref{cor:GPisreallynice}, the category $\GP(R)$ satisfies \Cref{setup:algebraicsetup}\ref{ks}-\ref{ab} and has a Serre-functor by \Cref{P:SerreFunctor}. 
   Hence, we can apply \Cref{cor:noloops3}.
\end{proof}

\subsection{The categorical `no-2-cycle theorem'}
In this section, we show a result similar to \Cref{thm:noloopfun}, but for $2$-cycles, by using the same trick of applying \Cref{lem:localise}.
First, we generalize Geiß--Leclerc--Schröer's \mbox{\cite[Proposition 3.11]{GLSInv}}, which builds on earlier work of Lenzing and Bongartz. We build on work of Igusa--Liu--Paquette \cite{IgusaLiuPaquette}.

\begin{lem}\label{lem:glsgeneralized}
   Let $A \coloneqq \kk Q/I$ be a finite-dimensional $\kk$-algebra, where $I \lhd \kk Q$ is an admissible ideal. 
   Assume that there are distinct vertices $i$ and $j$ in $Q$ such that $\prdim_A (S_i \oplus S_j) < \infty$ and such that $\Ext_A^2(S_i, S_i) = 0$ and $\Ext_A^2(S_j, S_j) = 0$. 
   Then 
   \begin{enumerate}[label={(\alph*)}]
      \item $\Ext^1_{A}(S_i, S_j) = 0$ or $\Ext_A^1(S_j, S_i) = 0$.\label{item:glsgeneralized1}
      \item Equivalently, $Q$ contains no $2$-cycle $i \leftrightarrows j$.\label{item:glsgeneralized2}
   \end{enumerate}
\end{lem}
\begin{proof}
    Suppose the contrary of \ref{item:glsgeneralized2}, i.e.\ there are arrows $\alpha \colon i \to j$ and $\beta \colon j \to i$ in $Q$. Let $\sigma_1 \coloneqq \alpha \beta$ and $\sigma_2 \coloneqq \beta \alpha$ as well as $e \coloneqq e_i + e_j$. 
    
    We show that $\sigma_1$ and $\sigma_2$ are cyclically free, in the sense of \cite[page 2739]{IgusaLiuPaquette}, that is $\sigma_1$ and $\sigma_2$ are not summands of a \emph{minimal relation} $\rho := \sum_{l=1}^n \lambda_l p_l \in I$ of $A$, where the $\lambda_l \in \kk$ are non-zero scalars, the $p_l$ are distinct paths in $Q$ from $i$ to $i$ and $\sum_{l \in \Omega} \lambda_l p_l \notin I$ for all $\emptyset \subsetneq \Omega \subsetneq \{1,\dots,n\}$. 
    Indeed, if $p_l = \sigma_1$ for some $1 \leq l \leq n$ then $\rho / \lambda_l$ would be of the form $\alpha \beta + c$ where $c \neq - \alpha \beta$ is a sum of paths from $i$ to $i$. This would imply $\Ext^2_A(S_i, S_i) \neq 0$ by \cite[Lemma 3.10]{GLSInv}. Similarly, $p_l \neq \sigma_2$ for $1 \leq l \leq n$. This shows that $\sigma_1$ and $\sigma_2$ are cyclically free.
    
    Since $e = e_i + e_j$ is the support of $\sigma_1$ and $\sigma_2$ we can apply \cite[Theorem 2.3]{IgusaLiuPaquette} and obtain $\prdim (S_i \oplus S_j) = \infty$, which is a contradiction. Hence, part \ref{item:glsgeneralized2} is true and part \ref{item:glsgeneralized1} follows from part \ref{item:glsgeneralized2} by \Cref{lem:extar}, cf.\ also \cite[Lemma III.2.12]{ASS}.
\end{proof}

Recall the definition of the quiver of a category from \Cref{defn:quiver}.
\begin{thm}\label{lem:glsfunctor}
    Suppose $\ca$ is a category satisfying \Cref{setup:algebraicsetup}\ref{ks}-\ref{ab}. 
    If $X,Y \in \ind \ca$ admit sink morphisms and satisfy $\Ext^2_{\mod \ca}(S_{X}, S_{X}) = 0$ and $\Ext^2_{\mod \ca}(S_Y, S_Y) = 0$ as well as $\prdim_{\mod \ca} (S_{X} \oplus S_{Y}) < \infty$ then 
    \begin{enumerate}[label={(\alph*)}]
       \item $\Ext^1_{\mod \ca}(S_X, S_Y) = 0$ or $\Ext^1_{\mod \ca}(S_Y, S_X) =0$.\label{item:glsfunctor1}
       \item Equivalently, the quiver of $\ca$ has no $2$-cylce $X \leftrightarrows Y$.\label{item:glsfunctor2}
    \end{enumerate}
\end{thm}
\begin{proof}
   Using \Cref{lem:localise} there are objects $P_{S_Z}$ for $Z \in \{X,Y\}$ so that for every $P$ with $\add P_{S_Z} \subseteq \add P$ the module $S_{Z}(P)$ has finite projective dimension over $\End_{\ca}(P)$ and $\Ext^i_{\mod \ca}(S_{Z},G) \cong \Ext^i_{\End_{\ca}(P)}(S_{Z}(P), G(P))$ for $G \in \mod \ca$ and $i \in \mathbb{N}$.
   Let $P$ be a basic additive generator of $\add(X \oplus Y \oplus P_{S_X} \oplus P_{S_Y})$ and $A$ be the basic finite-dimensional $\kk$-algebra $\End_{\ca}(P)$.
    Notice, $A$ is the quotient of a quiver algebra $\kk Q$ by an admissible ideal $I \lhd \kk Q$, by \cite[Theorem II.3.7]{ASS}.
    
    By the above, we know that $S_{X}(P) \oplus S_{Y}(P)$ has finite projective dimension over $A$ and $\Ext_{\mod \ca}^i(S_{Z}, S_{Z'}) \cong \Ext_A^i(S_{Z}(P), S_{Z'}(P))$ for $i = 1,2$ and $Z,Z' \in \{X,Y\}$.
    Using \Cref{rem:simple} and $P$ being basic with $X, Y \in \add P$, we obtain that $S_{Z}(P)$ has $\kk$-dimension $1$ for $Z \in \{X,Y\}$ and is hence simple.
    By \Cref{lem:glsgeneralized}, we have $\Ext_A^1(S_X(P), S_Y(P)) = 0$ or $\Ext_A^1(S_X(P),S_Y(P)) = 0$. 
    This shows part \ref{item:glsfunctor1} and part \ref{item:glsfunctor2} follows from part \ref{item:glsfunctor1} by \Cref{lem:extar}.
\end{proof}

\subsection{Obstructions from cluster-tilting objects}
   The aim of this section is to generalize \cite[Proposition II.1.11]{BIRSc} to show \Cref{T:PropertiesDsgR}\ref{item:noloops2cycles}.
   Indeed, we use the same strategy as in loc.\ cit., we just have to take care of more technicalities. 
   
   Recall our conventions regarding exact categories and the stable category of Frobenius exact categories, see \Cref{sec:aussolred}. 
   In particular, recall \Cref{Not:exact,not:abuse,Not:frobenius}.
   Throughout this section we assume that $\ce$ is a category satisfying \Cref{setup:algebraicsetup}\ref{ks} and that $\ce_\cp$ is a Frobenius exact structure on $\ce$ such that its stable category $\ul{\ce}_\cp = (\ce/\cp, \smash{\Sigma_{\ul{\ce}_\cp}}, \smash{\Delta_{\ul{\ce}_\cp}})$ is $2$-Calabi--Yau.

\begin{defn}\label{defn:ct}
   A subcategory $\ct$ of $\ce$ (or $\ce/\cp$) is called a \emph{cluster-tilting subcategory}\footnote{This is sometimes called $2$-cluster-tilting subcategory.}\footnote{In a non-Calabi--Yau setting the definition in \Cref{remark:proj in ct}\ref{item:otherdef} is different and more common.} of $\ce_{\cp}$ (or $\ul{\ce}_{\cp}$) if the following conditions hold. 
\begin{enumerate}[label={{(\alph*)}}]
   \item $\Ext^1(T, T')=0$ for all $T,T' \in \ct$.
   \item If $X$ is an object such that $\Ext^1(T, X)=0$ for all $T \in \ct$, then $X \in \ct$.
   \item $\ct \subseteq \ce$ (or $\ct \subseteq \ce/\cp$) is functorially finite.
\end{enumerate}
Here, $\Ext^1(-,-) := \Ext^1_{\ce_{\cp}}(-,-)$ (or $\Ext^1(-,-) := {\Hom_{\ce/\cp}}(-,\Sigma_{\ul{\ce}_\cp}-)$).
\end{defn}

Cluster-tilting subcategories $\ct$ of $\ul{\ce}_{\cp}$ with an additive generator are precisely the subcategories $\ct = \add T$ for a cluster-tilting object $T$ in the sense of \cref{defn:ctsmall}. Notice, because $\ul{\ce}_{\cp}$ is $\Hom$-finite the condition that $\add T$ is functorially finite is automatic for any (cluster-tilting) object $T \in \ce/\cp$.

\begin{rem} Suppose we are in the setting of \Cref{defn:ct}.
\begin{enumerate}[label={{(\alph*)}}]\label{remark:proj in ct}
   \item By definition any cluster-tilting subcategory is an additive subcategory which is closed under direct summands and isomorphisms.
   \item Suppose $\ct$ is a cluster-tilting subcategory of $\ce_{\cp}$.
         Then $\cp \subseteq \ct$ by condition (b) and the Frobenius property of $\ce_{\cp}$\label{item:proj in ct}.
         This implies that any left (right) $\ct$-approximation is an inflation (deflation) by a similar argument as in the proof of \Cref{lem:exactpd2}\ref{item:exactpd2}.
   \item\label{item:otherdef} One can check that $\ct$ is a cluster-tilting subcategory of $\mathcal{E}_{\cp}$ (or $\ul{\ce}_{\cp}$) if and only if it is functorially finite in $\ce$ (or $\ce/\cp$) and
   \begin{align*}
      \ct &= \set{X }{\text{$\Ext^1_{}(T, X)=0$ for all $T \in \ct$}} \\
          &= \set{X }{\text{$\Ext^1_{}(X, T)=0$ for all $T \in \ct$}}.
   \end{align*}
   The latter definition is used in \cite{BIKR} in our setting. 
   The second equality follows from the $2$-Calabi--Yau property, see \cite[Proposition I.1.1]{BIRSc} for $\ce_\cp$.
   This definition is also used in \cite[Definition 4.13]{Jasso}.
\end{enumerate}
\end{rem}

The next lemma shows that one can pass between cluster-tilting subcategories of $\ce_\cp$ and $\ul{\ce}_\cp$, see \cite[Lemma I.1.3.]{BIRSc}.

\begin{lem}\label{LIRSc}
   Let $\ct$ be an additive subcategory of $\ce$ containing all objects in $\cp$. 
   
   Then $\ct$ is a cluster-tilting subcategory of $\mathcal{E}_{\cp}$ if and only if $\ct/\cp$ is a cluster-tilting subcategory of $\ul{\ce}_{\cp}$.
\end{lem}

The following is an easy modification of \cite[Proposition II.1.11(a)]{BIRSc} to our more general setting concerning cluster-tilting \emph{subcategories} instead of \emph{objects}.
    Notice also that a similar statement on the level of stable categories is not true in general, cf.\ \cite[Corollary on page 128]{KellerReiten07}.

\begin{lem}\label{lem:fingldim}
    Suppose $\ct$ is a cluster-tilting subcategory of $\ce_{\cp}$. 
    Then 
    \begin{enumerate}[label={(\alph*)}]
       \item $\prdim_{\mod \ct} \ce(-|_{\ct},X) \leq 1$ for $X \in \ce$ and\label{item:projdim1}
       \item $\gldim \mod \ct \leq \gldim \mod \ce + 1$.\label{item:fingldim}
    \end{enumerate}
\end{lem}
\begin{proof}
    Because $\ct$ is a cluster-tilting subcategory in $\ce_\cp$ there are $T_0, T_1 \in \ct$ and a conflation $T_0 \rightarrowtail T_1 \twoheadrightarrow X$ of $\ce_\cp$, cf.\ \cite[Proposition 4.15]{Jasso} or \cite[Proposition II.1.7]{BIRSc}. 
    Applying $\ce(-|_{\ct}, -)$ to this conflation yields an exact sequence
    \[\begin{tikzcd}[ampersand replacement=\&]
        0 \& {\ct(-,T_0)} \& {\ct(-,T_1)} \& {\ce(-|_{\ct},X)} \& {\Ext^1_{\ce_\cp}(-|_{\ct},T_0)}.
        \arrow[from=1-2, to=1-3]
        \arrow[from=1-3, to=1-4]
        \arrow[from=1-4, to=1-5]
        \arrow[from=1-1, to=1-2]
    \end{tikzcd}\]
    The cluster-tilting property of $\ct$ implies $\Ext^1_{\ce_\cp}(-|_{\ct}, T_0) = 0$ and hence part \ref{item:projdim1}.
    
    Let $F \in \mod \ct$ be arbitrary. 
    Then there are $T_0, T_1 \in \ct$ and $f_1 \colon T_1 \to T_0$, so that 
    \[\begin{tikzcd}[ampersand replacement=\&]
        {\ct(-,T_1)} \& {\ct(-,T_0)} \&[-.3cm] F \&[-.5cm] 0
        \arrow["{f_1\cdot-}", from=1-1, to=1-2]
        \arrow["\varepsilon", from=1-2, to=1-3]
        \arrow[from=1-3, to=1-4]
    \end{tikzcd}\]
    is exact in $\mod \ct$. 
    Consider $G(-) \coloneqq \cok(\ce(-,T_1) \smash{\xrightarrow{f_1\cdot-}} \ce(-,T_0))$ in $\mod \ce$.
    Clearly, $G(-|_\ct) = F(-)$.
    Restring a projective resolution of $G$ over $\mod \ce$ yields a resolution of $F$ over $\mod \ct$ by $\prdim_{\mod \ce} G + 1$  objects of the form $\ce(-|_\ct,X)$, where $X \in \ce$.
    Using part \ref{item:projdim1}, and the well-known fact that
    \[\prdim_{\mod \ct} H'' \leq \max\{\prdim_{\mod \ct} H, \prdim_{\mod \ct} H'\} + 1\] 
    for any exact sequence $0 \to H' \to H \to H'' \to 0$ in $\mod \ct$, inductively on this resolution of $F$ yields $\prdim_{\mod \ct} F \leq \prdim_{\mod \ce} G+1$.
\end{proof}

\begin{lem}\label{lem:inheritsetup}
    Suppose that $\ct$ is a cluster-tilting subcategory of $\ce_\cp$.   
    If $\ce$ satisfies \Cref{setup:algebraicsetup}\ref{ks}-\ref{fg} then so does $\ct$.
\end{lem}
\begin{proof}
    If $f \colon T \to T'$ is a morphism in $\ct$ with weak kernel $g \colon X \to T$ in $\ce$ then for any right $\ct$-approximation $h \colon T'' \to X$ the morphism $gh$ is a weak kernel of $f$ in $\ct$. As $\ct \subset \ce$ is functorially finite this shows that \ref{ab} is inherited.
    
    \Cref{setup:algebraicsetup}\ref{ks} and \ref{hf} are inherited by additive subcategories which are closed under direct summands and that \ref{fg} is inherited follows from \Cref{lem:fingldim}\ref{item:fingldim}.
\end{proof}

\Cref{lem:glsfunctor} allows us to modify the proof of \cite[Proposition II.1.11]{BIRSc} to fit our more general setting. Recall the definition of the quiver of a category from \Cref{defn:quiver}.

\begin{thm}\label{thm:noloopsand2cycles}
    Let $\ct$ be a cluster-tilting subcategory of $\ce_{\cp}$ and suppose that $\ce$ satisfies \Cref{setup:algebraicsetup}\ref{ks}-\ref{fg}.
    
    Then the quiver of $\ct/\cp$ has no loops and no $2$-cycles.
\end{thm}
\begin{proof}
    \Cref{lem:inheritsetup} shows that $\ct$ satisfies \Cref{setup:algebraicsetup}\ref{ks}-\ref{fg}.
    As $\ul{\ce}_{\cp}$ has a Serre-functor, it follows from \cite[Theorem A]{ReitenVandenBergh} that $\ul{\ce}_\cp$ has Auslander--Reiten triangles and hence enough sink morphisms.
    By \Cref{cor:sinklift} every object in $\ind \ct \setminus \ind \cp$ admits a sink morphism in $\ce$ and hence, by \Cref{lem:sinksubcat}, in $\ct$.
    Using this and \Cref{setup:algebraicsetup}\ref{fg} together with \Cref{thm:noloopfun} this shows that the quiver of $\ct$ has no loops at the objects in $\ind \ct \setminus \ind \cp$.
    Hence, there are no loops in the quiver of $\ct/\cp$, by \Cref{lem:cyclequot,lem:radicalfunctor}.
    
    Let $T \in \ind \ct \setminus \ind \cp$. We claim that $\Ext^2_{\mod \ct}(S_T,S_T) = 0$. Let $\cu \coloneqq \ct \setminus \add T$.
    By \Cref{cor:IyamaYoshino} there are conflations
    \[\text{\begin{tikzcd}[ampersand replacement=\&]
    {T^\ast} \& U \& T
    \arrow["f", tail, from=1-1, to=1-2]
    \arrow["g", two heads, from=1-2, to=1-3]
    \end{tikzcd} and \begin{tikzcd}[ampersand replacement=\&]
    T \& V \& {T^\ast}
    \arrow["{i}", tail, from=1-1, to=1-2]
    \arrow["{j}", two heads, from=1-2, to=1-3]
    \end{tikzcd}}
    \]
    in $\ce_\cp$ such that $f$ and $i$ are left minimal left $\cu$-approximations and $g$ and $j$ are right minimal right $\cu$-approximations.
    As $T$ is non-projective the quiver of $\ct$ has no loops at $T$ by the above.
    Using \Cref{lem:sinkiscover} we see that $g$ is a sink morphism for $T$ in $\ct$.
    Hence, applying the left exact functor $\ce(-|_{\ct},-)$ to the first sequence gives us an exact sequence
    \[\begin{tikzcd}[ampersand replacement=\&]
    0 \& \ce(-|_{\ct},{T^\ast}) \& \ct(-,U) \& \ct(-,T) \& S_T \& 0.
    \arrow["f\cdot-", from=1-2, to=1-3]
    \arrow["g\cdot-", from=1-3, to=1-4]
    \arrow[from=1-5, to=1-6]
    \arrow[from=1-4, to=1-5]
    \arrow[from=1-1, to=1-2]
    \end{tikzcd}\]
    Applying the same functor to the second sequence gives us an exact sequence
    \[\begin{tikzcd}[ampersand replacement=\&]
    0 \& \ct(-,{T}) \& \ct(-,V) \& \ce(-|_{\ct},T^{\ast}) \& \Ext^1_{\ce_{\cp}}(-|_{\ct}, T)
    \arrow["i\cdot-", from=1-2, to=1-3]
    \arrow["j\cdot-", from=1-3, to=1-4]
    \arrow[from=1-4, to=1-5]
    \arrow[from=1-1, to=1-2]
    \end{tikzcd}\]
    where $\Ext^1_{\ce_{\cp}}(-|_{\ct},T) = 0$ as $T \in \ct$ and $\ct$ is cluster-tilting subcategory of $\ce_{\cp}$.
    Hence, gluing of the two sequences yields a projective resolution
    \[\begin{tikzcd}[ampersand replacement=\&]
    0 \&[-.5em] \ct(-,T) \& \ct(-,V) \& \ct(-,U) \& \ct(-,T) \& S_T \&[-.5em] 0
    \arrow["fj\cdot-", from=1-3, to=1-4]
    \arrow["g\cdot-", from=1-4, to=1-5]
    \arrow[from=1-5, to=1-6]
    \arrow[from=1-6, to=1-7]
    \arrow[from=1-2, to=1-3, "i\cdot-"]
    \arrow[from=1-1, to=1-2]
    \end{tikzcd}\]
    of $S_T$ in $\mod \ct$.
    Applying $\Hom_{\mod \ct}(-, S_T)$ to this resolution and using the Yoneda lemma, we see that $\Ext_{\mod \ct}^2(S_T,S_T)$ is isomorphic to the homology of
    \[\begin{tikzcd}[ampersand replacement=\&, column sep = 1.5cm]
        {S_T(U)} \& {S_T(V)} \& {S_T(T)}
        \arrow["{S_T(fj)}", from=1-1, to=1-2]
        \arrow["{S_T(i)}", from=1-2, to=1-3]
    \end{tikzcd}\]
    at $S_T(V)$.
    However, $S_T(V) = 0$ as $V \in \cu = \ct \setminus \add T$, so $\Ext^2_{\mod \ct}(S_T, S_T) = 0$.
    
    Now, there are no $2$-cycles $T \leftrightarrows T'$ between objects $T,T' \in \ind \ct \setminus \ind \cp$ using \Cref{setup:algebraicsetup}\ref{fg} together with \Cref{lem:glsfunctor}.
    Hence, the quiver of $\ct/\cp$ has no $2$-cycles by \Cref{lem:cyclequot,lem:radicalfunctor}.
\end{proof}

Recall the notions related to Gorenstein rings from \Cref{D:GProj,defn:Goring} and the definition of the AR-quiver from \Cref{defn:ARquiver}.
Applying the above \namecref{thm:noloopsand2cycles} in our setting shows the following corollary. 
In particular, this shows Theorem \ref{T:PropertiesDsgR}\ref{item:noloops2cycles}.

\begin{cor}\label{C:Noloops2cycles}
      Let $R$ be a finite-dimensional Gorenstein $\kk$-algebra and $\GP_{\cp}(R)$ be a Frobenius exact structure on $\GP(R)$, such that $\ul{\GP}_{\cp}(R)$ is $2$-Calabi--Yau.
      
      Then there are neither loops nor $2$-cycles in the quiver of any cluster-tilting subcategory $\ct $ of $\ul{\GP}_{\cp}(R)$.
\end{cor}
\begin{proof}
   Any cluster-tilting subcategory of $\ul{\GP}_{\cp}(R)$ arises from a cluster-tilting subcategory of $\GP_{\cp}(R)$ by \Cref{LIRSc}.
   Now we can apply \Cref{thm:noloopsand2cycles} because of \Cref{cor:GPisreallynice}.
\end{proof}

\appendix \section{Background on (functor) categories}

\subsection{Krull--Schmidt categories and the categorical radical.}
We recall some well-known facts about the radical of Krull--Schmidt categories, cf.\ e.g.\ \cite{Kra15}.
In this section, we assume that $\ca$ is a $\kk$-linear category over a commutative ring $\kk$.

\begin{defn}\label{defn:KS}
    We say that $\cA$ is \emph{Krull--Schmidt} if any $X \in \cA$ admits a decomposition $X \cong X_1 \oplus \cdots \oplus X_n$ with $\End_\cA (X_i)$ local for $i = 1,\dots,n$.
\end{defn}
Equivalently, $\ca$ is idempotent complete and $\End_\cA (X)$ is semiperfect for $X \in \cA$, see \cite[Corollary 4.4]{Kra15}.
By \cite[Theorem 4.2]{Kra15} any object in $\ca$ admits a \emph{unique} decomposition into indecomposable objects.
In particular, the following is true.
\begin{lem}\label{lem:unique}
    Suppose $\ca$ is Krull--Schmidt and $X,Y \in \ca$.
    Then $\add X \cap \add Y \neq 0$ if and only if $X$ and $Y$ share an (indecomposable) direct summand.
\end{lem}

A very useful tool when studying $\kk$-linear categories is the radical, which generalises the Jacobson radical from $\kk$-algebras to $\kk$-linear categories. 
\begin{defn}
    We say that $f \in \ca(X,Y)$ is in the \emph{radical $\rad_{\ca}(X,Y)$ of $\ca$} if $\id_X - gf$ is invertible for all $g \in \ca(Y,X)$. 
    Inductively, for $n \geq 2$ we define $\rad_{\ca}^n(X,Y)$ as the set of all $f \in \ca(X,Y)$ which are of the form $f = hg$ for some $Z \in \ca$ and $g \in \rad^{n-1}_{\ca}(X,Z)$ as well as $h \in \rad_{\ca}(Z,Y)$.
\end{defn}

It is well-known that a morphism $f \in \ca(X,Y)$ belongs to $\rad_{\ca}(X,Y)$ if and only if $\id_{Z} - gfh$ is invertible for all $g \in \ca(Y,Z)$ and $h \in \ca(Z,X)$.
This is also equivalent to $\id_{Y} - fg$ being invertible for all $g \in \ca(Y,X)$, see for example \cite[Corollary 2.10]{Kra15}.
It follows that $\rad_{\ca}$ is an ideal in $\ca$, cf.\ \cite[Proposition 2.9]{Kra15}, and $f \in \rad_{\ca}(X,Y)$ if and only if $f \in \rad_{\ca^{\op}}(Y,X)$. One also easily checks $\rad_{\ca}(X,X) = \rad \End_{\cA}(X)$, where $ \rad \End_{\cA}(X)$ is the Jacobson radical of $\End_{\cA}(X)$.

The radical of a Krull--Schmidt category has a particularly nice characterization.

\begin{lem}\label{lem:bestchar}
    If $\cA$ is Krull--Schmidt then $f \in \ca(X,Y)$ is in the radical precisely if there are no $g \in \ca(Y, Z)$ and $h \in \ca(Z,X)$ such that $gfh$ is an automorphism of some non-zero object $Z \in \ca$.
\end{lem}
\begin{proof}
    We show there is an automorphism factoring through $f$ if it is not in the radical.
    Suppose
    \[f = \left[\begin{smallmatrix} f_1 & f_2 \end{smallmatrix}\right] \colon X_1 \oplus X_2 \to Y\]
    is a decomposition then one of $f_1$ and $f_2$ is not in the radical, as it is an ideal.
    Furthermore, if a non-trivial automorphism factors through $f_1$ or $f_2$ then it factors through $f$.
    As $\cA$ has the descending chain condition on direct summands, this argument and its dual shows that we can assume that $X$ and $Y$ are indecomposable.
    As $f$ is not in the radical we have $g \in \ca(Y,X)$ such that $\id_X - gf$ is not invertible.
    As $\End_{\ca} (X)$ is local that means that $\id_X - gf$ is in $\rad \End_{\ca}( X)$.
    So $gf = \id_X - (\id_X -gf)$ is invertible and hence an automorphism.

    The converse follows as radical is an ideal not containing any isomorphism.
\end{proof}

\begin{lem}\label{lem:radical}
    Suppose $\cA$ is Krull--Schmidt.
    For $X,Y \in \cA$ indecomposable any morphism $f \in \ca(X, Y)$ is either in $\rad_\ca(X,Y)$ or an isomorphism.
    For $X,Y \in \ca$ not necessarily indecomposable the three statements
    \begin{align*}
        &\rad_{\ca}(X, Y) = \ca(X, Y) \text{,}\\
        &\rad_{\ca}(Y, X) = \ca(Y, X) \text{,}\\
        &\add (X) \cap \add (Y) =0 
    \end{align*}
    are equivalent.
\end{lem}
\begin{proof}
    A non-trivial automorphism factors through a morphism $f \in \cA(X,Y)$ if and only if $f$ induces an isomorphism between a non-trivial common direct summand of $X$ and $Y$.
    If $X$ and $Y$ are indecomposable this is only possible if $f$ is itself an isomorphism.
    The first part follows now from \Cref{lem:bestchar}. 

    Notice, a non-trivial common direct summand between $X$ and $Y$ yields a non-trivial automorphism factoring through a morphism in $\ca(X,Y)$.
    Hence, there is an automorphism factoring through a morphism in $\ca(X,Y)$ if and only if $X$ and $Y$ share a non-trivial direct summand.
    Again, \Cref{lem:bestchar} and \Cref{lem:unique} show equivalence of the statements in the second part.
\end{proof}

We show that the Krull--Schmidt property is inherited by quotients and that the quotient functor induces an epimorphism on the radical.
First, recall the following.
\begin{lem}\label{lem:wellknownrad}
    Let $\Phi \colon \ca \to \cb$ be a full $\kk$-linear functor.
    Then $\Phi(\rad_{\ca}(X,Y))$ is a subset of $\rad_{\cb}(\Phi(X),\Phi(Y))$ for any $X, Y \in \ca$.
\end{lem}
\begin{proof}
    If $f \in \rad_{\ca}(X,Y)$ then $\id_X - gf$ is invertible in $\End_{\ca} (X)$ for all $g \in \ca(Y,X)$.
    Therefore, $\Phi(\id_X - gf) = \id_{ \Phi(X)} - \Phi(g)\Phi(f)$ is invertible in $\End_{\cb} (\Phi(X))$ for every morphism $\Phi(g) \in \Phi(\ca(Y,X))$.
    As $\Phi$ is full, $\Phi(f) \in \rad_{\cb}(\Phi(X),\Phi(Y))$ follows.
\end{proof}

\begin{lem}\label{lem:localepi}
   Let $\phi \colon \Lambda \to \Lambda/I$ be the quotient of a local $\kk$-algebra $\Lambda$ by an ideal $I \unlhd \Lambda$.
   Then $\phi(\rad \Lambda) = \rad \Lambda/I$ and if $I \lhd \Lambda$ is proper then $\Lambda/I$ is local. 
\end{lem}
\begin{proof}
    The statement for $I = \Lambda$ is clear, so we assume that $I \lhd \Lambda$ is proper.
    The maximal ideals of $\Lambda/I$ are via $\phi^{-1}$ in bijection with the maximal ideals of $\Lambda$ containing $I$.
    As $I \lhd \Lambda$ is proper, it is contained in the unique maximal ideal $\rad \Lambda$ of $\Lambda$.
    Hence, $\Lambda/I$ has a unique maximal ideal $\rad \Lambda/I$ which satisfies $\phi^{-1} (\rad \Lambda/I) = \rad \Lambda$.
    We have $\phi(\rad \Lambda) = \phi(\phi^{-1}(\rad \Lambda/I)) = \rad \Lambda/I$ as $\phi$ is surjective.
\end{proof}

If $\Lambda$ is not semiperfect $\phi(\rad \Lambda) = \rad \Lambda/I$ does not have to be true.
For example $\phi \colon \kk[z] \to \kk[z]/(z^2)$ satisfies $\rad \kk[z] = 0$ but $\rad \kk[z]/(z^2) \neq 0$.
This can be used to show that part \ref{item:radicaliso} of the following \namecref{lem:radicalfunctor} may fail if $\cA$ is not Krull--Schmidt.

\begin{lem}\label{lem:radicalfunctor}
    Suppose $\Phi \colon \ca \to \cb$ is a full, essentially surjective $\kk$-linear functor and that $\ca$ is a Krull--Schmidt category.
    Then
    \begin{enumerate}[label={(\alph*)}]
        \item $\cb$ is a Krull--Schmidt category,\label{item:ks}
        \item every $X \in \ca$ admits a decomposition $X \cong X_1 \oplus \cdots \oplus X_n$ and such that for each $1 \leq i \leq n$ either $X_i \in \Phi^{-1}(\ind \cb) \cap \ind \ca$ or $X_{i} \in \Phi^{-1}(0) \cap \ind \ca$, and\label{item:kspreimage}
        \item $\Phi(\rad^n_{\ca}(X,Y)) = \rad^n_{\cb}(\Phi(X), \Phi(Y))$ for $n \geq 1$ and $X,Y \in \cA$.\label{item:radicaliso}
    \end{enumerate}
    In particular, the essential image of $\Phi^{-1}(\ind \cb) \cap \ind \ca$ under $\Phi$ is $\ind \cb$.
\end{lem}
\begin{proof}
    If $X \in \ind \ca$ then $\Phi(X)$ has local endomorphism ring or vanishes in $\cb$ by \Cref{lem:localepi}.
    If $X \in \ca$ is arbitrary then $X \cong X_1 \oplus \cdots \oplus X_n$ in $\ca$ for some $X_1, \dots, X_n \in \ind \ca$. After reordering we may assume that $\Phi(X_1), \dots, \Phi(X_k)$ have local endomorphism ring and $X_{k+1}, \dots, X_{n}$ map to zero. This establishes \ref{item:kspreimage}.
    We have an isomorphism $\Phi(X) \cong \Phi(X_1) \oplus \cdots \oplus \Phi(X_k)$ in $\cb$ with $\End_{\cb} (\Phi(X_i))$ local by the above. As $\Phi$ is essentially surjective this establishes \ref{item:ks}.

    Using that $\Phi$ is essentially surjective it suffices to show \ref{item:radicaliso} for the case $n = 1$ by induction on $n$ and the definition of the higher powers of the radical.
    As $\Phi$ is a $\kk$-linear functor between Krull--Schmidt categories and as the radicals are ideals it suffices to consider the case where $X,Y$ are indecomposable in $\cA$.
    If $X \cong Y$ in $\ca$ then without loss of generality $X = Y$ and the lemma follows from \Cref{lem:localepi}.
    If $X \not \cong Y$ in $\cA$ then 
    \[\Phi(\rad_{\ca}(X,Y)) = \Phi(\ca(X,Y)) = \cb(\Phi(X),\Phi(Y)) \supset \rad_{\cb}(\Phi(X),\Phi(Y))\] 
    by \Cref{lem:radical} and $\Phi$ being full.
    \Cref{lem:wellknownrad} completes the proof of \ref{item:radicaliso}.

    To see the last part, notice that every object in $\ind \cb$ has a preimage $X$ in $\ca$.
    However, if we pick a decomposition of $X$ as in \ref{item:kspreimage} then $\Phi(X) \cong \Phi(X_1)$ as $\Phi(X)$ is indecomposable.
\end{proof}

\subsection{Functor categories}
In this section we assume that $\ca$ is a category satisfying \Cref{setup:algebraicsetup}\ref{ks}, \ref{ab}, but we can safely replace $\kk$ by just a commutative ring.
Then $\mod \ca$ is an abelian category with enough projective objects given precisely by the $\Hom$-functors $\ca(-,X)$ for $X \in \ca$. 
For $X,Y \in \ca$ the Yoneda lemma yields that
\[\ca(X,Y) \to \Hom_{\mod \ca}(\ca(-,X),\ca(-,Y)),\, f \mapsto \tau_f \coloneqq \ca(-,f)\]
is a natural $\kk$-linear isomorphism. 
Notice, $\tau_f$ is just left multiplication by $f$.
We use this isomorphism to freely pass between morphisms in $\ca$ and morphisms of projective $\ca$-modules. We start with the following well-known lemmas.
\begin{lem}\label{lem:mincover}
    For every $F \in \mod \ca$ there exists a short exact sequence
    \[\begin{tikzcd}[ampersand replacement=\&]
        0 \& G \& {\ca(-,X_0)} \& F \& 0
        \arrow[from=1-1, to=1-2]
        \arrow[from=1-2, to=1-3, "k"]
        \arrow["\varepsilon", from=1-3, to=1-4]
        \arrow[from=1-4, to=1-5]
    \end{tikzcd}\]
    in $\mod \ca$ with $\varepsilon$ right minimal.
\end{lem}
\begin{proof}
    By definition there is an exact sequence     
    \[\begin{tikzcd}[ampersand replacement=\&]
        \ca(-,X_1'') \& {\ca(-,X''_0)} \& F \& 0
        \arrow["d_1'' \cdot -", from=1-1, to=1-2]
        \arrow["\varepsilon''", from=1-2, to=1-3]
        \arrow[from=1-3, to=1-4]
    \end{tikzcd}\]
    in $\mod \ca$.
    By the dual of \cite[Corollary 1.4]{KS98}, we have
    \[\varepsilon'' = \left[\begin{smallmatrix}\varepsilon & 0\end{smallmatrix}\right] \colon \ca(-, X_0) \oplus \ca(-, X_0') \to F\]
    with $\varepsilon$ being right minimal for some decomposition $X_0'' \cong X_0 \oplus X_0'$.  
    As $\mod \ca$ is abelian, a kernel $k \colon G \to \ca(-,X_0)$ of $\varepsilon$ exists in $\mod \ca$.
\end{proof}

\begin{lem}\label{lem:minradical}
    Suppose we are given an exact sequence
    \[\begin{tikzcd}[ampersand replacement=\&]
        \ca(-,X_1) \& {\ca(-,X_0)} \& F
        \arrow["d_1\cdot -", from=1-1, to=1-2]
        \arrow["\varepsilon", from=1-2, to=1-3]
    \end{tikzcd}\]
    in $\mod \ca$. 
    Then $d_1 \in \rad_{\ca}(X_1,X_0)$ if and only if $\varepsilon$ is right minimal.
\end{lem}
\begin{proof}
    Let $\varepsilon$ be right minimal.
    Let $g \colon X_0 \to X_1$ be an arbitrary morphism.
    Then $\varepsilon \tau_{\id_{X_0} - d_1 g} = \varepsilon (\tau_{\id_{X_0}} - \tau_{d_1} \tau_{g}) = \varepsilon$ as $\varepsilon \tau_{d_1} = 0$ and so $\tau_{\id_{X_0}-d_1g}$ is invertible. Thus, $\id_{X_0} - d_1 g$ is invertible so $d_1 \in \rad_{\ca}(X_1,X_0)$ as $g$ was arbitrary. 
    
    On the other side, suppose $d_1 \in \rad_{\ca}(X_1, X_0)$. 
    Let $\tau_h \colon \ca(-,X_0) \to \ca(-,X_0)$ satisfy $\varepsilon \tau_h  =\varepsilon$. 
    Then $\varepsilon_{X_0}(\id_{X_0} - h) = (\varepsilon \tau_{\id_{X_0} - h})_{X_0}(\id_{X_0}) = (\varepsilon (\tau_{\id_{X_0} }- \tau_{h}))_{X_0}(\id_{X_0}) = 0$. By exactness there is a $g \in \ca( X_0, X_1)$ satisfying $d_1 g = \id_{X_0} - h$.
    So, $h = \id_{X_0} - d_1 g$ is invertible and so is $\tau_{h}$. This shows that $\varepsilon$ is right-minimal.
\end{proof}

This allows us to construct minimal projective resolutions in $\mod \cA$.

\begin{cor}\label{cor:minres}
 Every functor $F \in \mod \ca$ admits a projective resolution
\[\begin{tikzcd}
    \cdots & {\ca(-,X_2)} & {\ca(-,X_1)} & {\ca(-,X_0)} & F & 0
    \arrow["{d_2} \cdot -", from=1-2, to=1-3]
    \arrow["{d_1} \cdot -", from=1-3, to=1-4]
    \arrow["\varepsilon", from=1-4, to=1-5]
    \arrow["{d_3} \cdot -", from=1-1, to=1-2]
    \arrow[from=1-5, to=1-6]
\end{tikzcd}\]
with $d_i \in \rad_{\ca}(X_i,X_{i-1})$ for $i > 0$.
\end{cor}
\begin{proof}
    This follows from repeated application of \Cref{lem:mincover} and \ref{lem:minradical}.
\end{proof}

Lastly, the following is an immediate consequence of the Yoneda lemma.

\begin{lem}\label{lem:minimalfun}
     A morphism $f \colon X \to Y$ is right minimal if and only if its induced morphism $\tau_f \colon  \ca(-,X) \to \ca(-,Y)$ is right minimal.
\end{lem}

\subsection{Approximations}
In this section we assume that $\ca$ satisfies \Cref{setup:algebraicsetup}\ref{ks}-\ref{ab}.
In particular, $\kk$ is an algebraically closed field.

\begin{rem}\label{rem:krause}
Because $\ca$ is Krull-Schmidt it follows from \cite[Corollary 1.4]{KS98} that every morphism in any additive, idempotent complete subcategory $\cb$ of $\ca$ admits a right minimal version. 
In particular, any object in $\ca$ that admits a right almost split morphism also admits a sink morphism and any object in $\ca$ that admits a right $\cb$-approximation admits a right minimal right $\cb$-approximation. 
\end{rem}

\begin{lem}\label{lem:sinksubcat}
    Suppose $\cb \subseteq \ca$ is an additive, idempotent complete and covariantly finite subcategory and let $X \in \cb$.
    If $X$ admits a sink morphism in $\ca$ then it admits a sink morphism in $\cb$.
\end{lem}
\begin{proof}
    Let $f \colon Y \to X$ be a sink morphism for $X$ in $\ca$ and $g \colon Z \to Y$ be a right $\cb$-approximation.
    Using the lifting properties of $f$ and $g$ it is easy to see that $fg$ is a right almost split morphism in $\cb$.
    The lemma follows from \Cref{rem:krause}.
\end{proof}

In the following two lemmas we use the quiver of a category, see \Cref{defn:quiver}, and sink morphisms, see \Cref{defn:sink}.

\begin{lem}\label{lem:sinkiscover}
    Suppose the quiver of $\ca$ has no loops at an object $X \in \ind \ca$.
    Then a morphism $f \colon Y \to X$ is a sink morphism for $X$ if and only if it is a right minimal right $(\ca \setminus \add X)$-approximation of $X$.
\end{lem}
\begin{proof}
    Let $f \colon Y \to X$ be a sink morphism.
    We have $Y \in \ca \setminus \add X$ by \Cref{lem:extar} as the quiver of $\ca$ has no loops at $X$.
    Any morphism $h \colon Z \to X$ with $Z \in \ca \setminus \add X$ cannot be a split epimorphism and hence factors through $f$.
    This shows that $f$ is a right $(\ca \setminus \add X)$-approximation and $f$ is right minimal by definition.

    Conversely, suppose that $f$ is a right minimal right $(\ca \setminus \add X)$-approximation.
    If we can show that $X$ admits a sink morphism then $f$ is a sink morphism for $X$ as any sink morphism for $X$ is a right minimal right $(\ca \setminus \add X)$-approximation of $X$ by the first part and hence isomorphic to $f$.
    To this end we pick a $\kk$-basis $f_1, \dots, f_n$ of $\rad \End_\ca (X)$.
    Then any non-split epimorphism $Z \to X$ factors through \[g \coloneqq \left[ \begin{smallmatrix} f & f_1 \cdots & f_n \end{smallmatrix} \right] \colon Y \oplus X^n \to X.\]
    Indeed, it suffices to show this for $Z$ indecomposable and this follows from considering $Z \cong X$ and $Z \in \ind \ca \setminus \{X\}$ separately.
    Hence, $g$ is a right almost split morphism for $X$ and hence $X$ admits a sink morphism by \Cref{rem:krause}.
\end{proof}

\begin{lem} \label{lem:sinkcover}
    Suppose that $X \in \ind \cA$ admits a sink morphism.
    Then $\ca \setminus \add X$ is covariantly finite in $\ca$.
 \end{lem}
\begin{proof}
    First notice, at every object of $\add X$ ends a radical morphism through which any other radical morphism ending at the same object factors.
    In fact, such a morphism can be obtained as a direct sum of sink morphisms ending at $X$.
    We call such a morphism a \emph{right $\rad_{\ca}$-approximation}. 

    It suffices to show that $X$ has a right $(\ca \setminus \add X)$-approximation.
    Let $X_0 = X$ and pick for $i \in \N$ inductively a right $\rad_{\ca}$-approximation 
    \begin{align}\label{eq:radapproxmation}\left[\begin{smallmatrix} f_{i+1} & g_{i+1} \end{smallmatrix}\right] \colon X_{i+1} \oplus Y_{i+1} \to X_{i},\end{align}
    where the domain is decomposed such that $X_{i+1} \in \add X$ and $\add Y_{i+1} \cap \add X = 0$.
    Since $\End_{\ca}(X)$ is a finite-dimensional $\kk$-algebra there exists an $n \in \N$ such that $\rad^n \End_{\ca}(X) = 0$. 
    Let $h_1 = g_1$ and $h_i = f_1 \cdots f_{i-1} g_i$ for $i \geq 2$.
    We claim that
    \[h \coloneqq \left[\begin{smallmatrix} h_{1} & \cdots & h_n \end{smallmatrix}\right] \colon Y_{1} \oplus \cdots \oplus Y_n \to X\]
    is a right $(\ca \setminus \add X)$-approximation.
    
    Let $Z \in \ca \setminus \add X$ and $k_0 \colon Z \to X$ be arbitrary.
    We have $\ca(Z, X_i) = \rad_{\ca}(Z, X_i)$ for $i \in \mathbb{N}$ by \Cref{lem:radical}.
    We can pick morphisms
    \[\left[\begin{smallmatrix} k_{i+1} \\ l_{i+1} \end{smallmatrix}\right] \colon Z \to X_{i+1} \oplus Y_{i+1}\] 
    inductively such that $\label{eq:firstfactor}k_i = f_{i+1}k_{i+1} + g_{i+1}l_{i+1}$
    for $i \in \N$, where we used that \eqref{eq:radapproxmation} is a right $\rad_{\ca}$-approximation.
    By induction one shows
    \begin{align}\label{eq:magicfactorization} k_0 = f_1 \cdots f_{n} k_{n} + \sum_{i=1}^{n} h_i l_i = f_1 \cdots f_{n} k_{n} + \left[\begin{smallmatrix} h_{1} & \cdots & h_n \end{smallmatrix}\right] \left[\begin{smallmatrix} l_{1} \\ \vdots \\ l_n \end{smallmatrix}\right]. \end{align}
    Notice, $f_1 \cdots f_{n} = 0$ because $\rad^n \End_{\ca} (X) = 0$ and $X_i \in \add X$ for $i \in \N$.
    Thus \eqref{eq:magicfactorization} shows that $k_0$ factors through $h$.
\end{proof}

\subsection{Frobenius exact categories}
In this section we assume that $\ce_\cp$ is a Frobenius exact category with stable category $\ul{\ce}_{\cp} \coloneqq (\ce/\cp, \smash{\Sigma_{\ul{\ce}_\cp}}, \smash{\Delta_{\ul{\ce}_\cp}})$, c.f.\ \Cref{Not:frobenius}.
For any morphism $f \in \ce(X,Y)$ we denote by $\ul{f} \in \Hom_{\ce/\cp}(X,Y)$ its image under the additive quotient functor $\ce \to \ce/\cp$.

\begin{lem}\label{lem:lifttriangle}
Suppose $X,Y,Z \in \ce$.
Any triangle
\begin{align}\begin{tikzcd}[ampersand replacement=\&]
   X \& Y \& Z \& {\Sigma_{\ul{\ce}_\cp} X}
    \arrow["", from=1-1, to=1-2]
    \arrow["", from=1-2, to=1-3]
    \arrow["", from=1-3, to=1-4]
\end{tikzcd}\label{eqn:generic_triangle}\end{align}
in $\ul{\ce}_{\cp}$ is isomorphic to a triangle induced by a conflation in $\ce_\cp$ of the form
\begin{align*}\begin{tikzcd}[ampersand replacement=\&]
    X \& {Y''} \& Z. \& {}
    \arrow["", tail, from=1-1, to=1-2]
    \arrow["", two heads, from=1-2, to=1-3]
\end{tikzcd}\label{eqn:genericonflation}\end{align*}
\end{lem}
\begin{proof}
   It is well-known that the triangle (\ref{eqn:generic_triangle}) is isomorphic to a triangle which is induced by a conflation $X \rightarrowtail Y' \twoheadrightarrow Z'$ in $\ce_\cp$, cf.\ e.g.\ \cite[Appendix A.1(vii)]{JY22} in a different setup.
   In particular, there is a morphism in $f \in \ce(Z, Z')$ which becomes an isomorphism in $\ce/\cp$.
   Through the pullback of $X \rightarrowtail Y' \twoheadrightarrow Z'$ along $f$ we obtain a conflation $X \rightarrowtail Y'' \twoheadrightarrow Z$ in $\ce_\cp$ which also induces a triangle isomorphic to (\ref{eqn:generic_triangle}). 
\end{proof}

\begin{lem}\label{lem:jenny}
     Suppose $f \colon Y \to X$ is a deflation in $\ce_\cp$ and $g \colon Z \to X$ is a morphism. 
     If $\ul{g}$ factors through $\ul{f}$ in $\ce/\cp$ then $g$ factors through $f$ in $\ce$.
\end{lem}
\begin{proof}
    By assumption there is $h \colon Z \to Y$ with $\ul{g} = \ul{f} \ul{h}$.
    Then $fh - g$ factors through a projective object, say $fh - g = qp$ for $p \colon Z \to P$ and $q \colon P \to X$ where $P \in \cp$.
    Since $f$ is a deflation in $\ce_\cp$ and $P$ is projective, the morphism $q$ factors through $f$, say $q = fr$ with $r \colon P \to Y$.
    This shows $g = f(h-rp)$.
\end{proof}

\begin{cor}\label{cor:sinklift}
    Assume $\ce$ satisfies \Cref{setup:algebraicsetup}\ref{ks}.
    If $X \in \ind \ce \setminus \ind \cp$ has a sink morphism in $\ce/\cp$ then it has a sink morphism in $\ce$.
\end{cor}
\begin{proof}
    Suppose $\ul{f} \colon Y \to X$ is a sink morphism in ${\ce}/\cp$ and $g \colon P \to X$ is a deflation in $\ce_\cp$ with $P \in \cp$. 
    Then $g \in \rad_{\ce}(P,X)$ by \Cref{lem:radical} and $\ul{f} \in \rad_{{\ce}/\cp}(Y, X)$ by assumption. 
    We may assume that $f \in \rad_{\ce}(Y,X)$ by \Cref{lem:radicalfunctor}\ref{item:radicaliso}.
    Then 
    \[h \coloneqq \left[\begin{smallmatrix} f & g \end{smallmatrix}\right] \colon Y \oplus P \to X\]
    is deflation in $\ce_\cp$, cf.\ e.g.\ \cite[Lemma A.1(i)]{JY22}, and a radical morphism.
    
    If $k \in \ce(Z,X)$ is not a split epimorphism then $k \in \rad_{\ce}(Z,X)$ and therefore $\ul{k} \in \rad_{\ce/\cp}(Z, X)$ by \Cref{lem:radicalfunctor}\ref{item:radicaliso}. 
    By \Cref{lem:jenny}, $k$ factors through $h$ as $\ul{k}$ factors through $\ul{f} \cong \ul{h}$.
    Hence $h$ is a right almost split morphism ending at $X$.
    It follows from \Cref{setup:algebraicsetup}\ref{ks} that $X$ has a sink morphism in $\ce$, cf.\ e.g.\ \Cref{rem:krause}.
\end{proof}

\begin{cor}\label{cor:jenny}
    Let $\ca \subseteq \ce$ be an additive subcategory which is closed under direct summands and contains $\cp$.
    If $f \colon Y \to X$ is a deflation in $\ce_\cp$ and $\ul{f} \colon Y \to X$ is a right $(\ca/\cp)$-approximation, then $f$ is a right $\ca$-approximation.
\end{cor}
\begin{proof}
    We have $Y \in \ca$ because $\ca$ is closed under direct summands and contains $\cp$.
    The corollary now follows from \Cref{lem:jenny}.
\end{proof}

\subsection{Iyama--Yoshino mutation}
The following is well-known to the experts. 
We give a proof for convenience of the reader.

\begin{thm}[{Iyama--Yoshino mutation, cf.\ \cite{IyamaYoshino}}]\label{thm:IyamaYoshino}
    Let $\cc$ be a $2$-Calabi--Yau triangulated category satisfying \Cref{setup:algebraicsetup}\ref{ks}, \ref{hf}.
    Suppose $\ct \subseteq \cc$ is a cluster-tilting subcategory and $\cu \coloneqq \ct \setminus \add T$ for some $T \in \ind \ct$. 
    Then there is $T^\ast \in \ind \cc \setminus \ind \ct$ such that
    \begin{enumerate}[label=(\alph*)]
        \item $\ct^\ast \coloneqq \add( \cu \cup \{T^\ast\})$ is a cluster-tilting subcategory of $\cc$ and
        \item there are triangles
        \[\text{\begin{tikzcd}[ampersand replacement=\&]
        {T^\ast} \& U \& T \& {\Sigma T^\ast}
        \arrow["f", from=1-1, to=1-2]
        \arrow["g", from=1-2, to=1-3]
        \arrow["h", from=1-3, to=1-4]
        \end{tikzcd} and \begin{tikzcd}[ampersand replacement=\&]
        T \& V \& {T^\ast} \& {\Sigma T}
        \arrow["{i}", from=1-1, to=1-2]
        \arrow["{j}", from=1-2, to=1-3]
        \arrow["{k}", from=1-3, to=1-4]
        \end{tikzcd}}
        \]
        in $\cc$, where $g$ and $j$ are right minimal right $\cu$-approximations and $f$ and $i$ are left minimal left $\cu$-approximations.
    \end{enumerate}
\end{thm}
\begin{proof}
    Notice, $\cc$ has enough source and sink morphisms as it has a Serre functor, see \cite[Theorem A]{ReitenVandenBergh}.
    By \Cref{lem:sinksubcat} and its dual, $\ct$ has enough source and sink morphisms.
    Now, $\cu$ is functorially finite in $\ct$ by \Cref{lem:sinkcover} and hence $\cu$ is also functorially finite in $\cc$ as $\ct$ is functorially finite in $\cc$. 
    This implies that $\cu$ is an almost complete cluster-tilting subcategory in the sense of \cite[Definition 5.2]{IyamaYoshino}. 
    By \cite[Theorem 5.3]{IyamaYoshino} there are exactly two cluster-tilting subcategories $\ct$ and $\ct^\ast$ of $\cc$ with $\cu \subseteq \ct, \ct^\ast$. 
    By the same theorem $(\ct,\ct^\ast)$ is an $\cu$-mutation pair in the sense of \cite[Definition 2.5]{IyamaYoshino}.
    This $\cu$-mutation pair gives us an equivalence $\GG \colon \ct/\cu \to \ct^\ast/\cu$ as in \cite[Proposition 2.6]{IyamaYoshino}.
    As $T \in \ind (\ct/\cu)$ is an additive generator of $\ct/\cu$ its image $\GG(T) \in \ind (\ct^\ast/\cu)$ is an additive generator of $\ct^\ast/\cu$.
    As $\cc$ is Krull--Schmidt we can decompose $\GG(T) \cong T^\ast \oplus V^*$ in $\ct^\ast$ such that $T^\ast \in \ind \ct^\ast \setminus \ind \cu$ and $V^* \in \cu$ and as $\GG(T)$ is a additive generator of $\ct^\ast/\cu$ we have $\ct^\ast = \add(\cu \cup \{T^\ast\})$, compare \Cref{lem:radicalfunctor}\ref{item:kspreimage}.
    By definition of $\GG$, see \cite[Proposition 2.6]{IyamaYoshino}, there is a triangle
    \begin{equation}\begin{tikzcd}[ampersand replacement=\&, column sep = .9cm]
        {T} \& {V'} \& {T^\ast \oplus V^*} \& {\Sigma T}
        \arrow["", from=1-1, to=1-2]
        \arrow["", from=1-2, to=1-3]
        \arrow["", from=1-3, to=1-4]
    \end{tikzcd}\label{eqn:triangle}\end{equation}
    in $\cc$ with $V' \in \cu$.
    Since $V^\ast, T \in \ct$ we have $\cc(V^\ast,\Sigma T) = 0$, so the canonical inclusion $V^\ast \to T^\ast \oplus V^\ast$ factors through $V' \to T^\ast \oplus V^\ast$. 
    A standard argument shows that there is a direct summand of the triangle (\ref{eqn:triangle}), which is of the shape
    \[\begin{tikzcd}[ampersand replacement=\&, column sep = .9cm]
        {T} \& {V} \& {T^\ast} \& {\Sigma T}
        \arrow["i", from=1-1, to=1-2]
        \arrow["j", from=1-2, to=1-3]
        \arrow["k", from=1-3, to=1-4]
    \end{tikzcd}\]
    with $V \oplus V^{\ast} = V' \in \cu$. 
    It follows from \cite[page 126]{IyamaYoshino} that $j$ is a left $\cu$-approximation and $i$ is a right $\cu$-approximation. 
    Since $T \notin \add \cu$ we have $i \in \rad_{\cc}(T,V)$. 
    Hence, $j$ is right-minimal using that the Yoneda-embedding of $\cc$ maps triangles to exact sequences and \Cref{lem:minradical,lem:minimalfun}.
    This implies that $j$ is a right minimal right $\cu$-approximation. 
    Dually, $i$ is a left minimal left $\cu$-approximation.
    
    Notice, $\cu = \ct^\ast \setminus \add T^\ast$. 
    Hence, the other triangle is obtained by reversing the roles of $T$ and $T^\ast$.
\end{proof}
\begin{cor}\label{cor:IyamaYoshino}
    Let $\ce_\cp$ be a Frobenius exact category such that its stable category $\ul{\ce}_\cp$ is $2$-Calabi--Yau and $\ct$ be a cluster-tiling subcategory of $\ce_\cp$. 
    Suppose $\ce$ satisfies \Cref{setup:algebraicsetup}\ref{ks}.
    Let $\cu \coloneqq \ct \setminus \add T$ for some object $T \in \ind \ct \setminus \ind \cp$.
    Then we have an object $T^\ast \in \ind \ce \setminus \ind \ct$ and conflations
        \begin{equation}\label{eqn:holyconflations}\text{\begin{tikzcd}[ampersand replacement=\&]
        {T^\ast} \& U \& T
        \arrow["f", from=1-1, to=1-2, tail]
        \arrow["g", from=1-2, to=1-3, two heads]
        \end{tikzcd} and \begin{tikzcd}[ampersand replacement=\&]
        T \& V \& {T^\ast}
        \arrow["{i}", from=1-1, to=1-2, tail]
        \arrow["{j}", from=1-2, to=1-3, two heads]
        \end{tikzcd}}
        \end{equation}
        in $\ce_\cp$, where $g$ and $j$ are right minimal right $\cu$-approximations and $f$ and $j$ are left minimal left $\cu$-approximations.
\end{cor}
\begin{proof}
    By \Cref{LIRSc} the subcategory $\ct/\cp$ is cluster-tilting in $\ul{\ce}_\cp$.
    By \Cref{lem:radicalfunctor}, $\ul{\ce}_{\cp}$ inherits \Cref{setup:algebraicsetup}\ref{ks} from $\ce$.
    There are $T^\ast \in \ind ({\ce}/\cp) \setminus \ind (\ct/\cp)$ and triangles
    \[\text{\begin{tikzcd}[ampersand replacement=\&]
        {{T}^\ast} \& {U} \& {T} \& {\Sigma_{\ul{\ce}_\cp} {T}^\ast}
        \arrow["\ul{f}", from=1-1, to=1-2]
        \arrow["\ul{g}", from=1-2, to=1-3]
        \arrow["\ul{h}", from=1-3, to=1-4]
    \end{tikzcd} and \begin{tikzcd}[ampersand replacement=\&]
        {T} \& {V} \& {{T}^\ast} \& {\Sigma_{\ul{\ce}_\cp} T}
        \arrow["{\ul{i}}", from=1-1, to=1-2]
        \arrow["{\ul{j}}", from=1-2, to=1-3]
        \arrow["{\ul{k}}", from=1-3, to=1-4]
    \end{tikzcd}}\]
    in $\ul{\ce}_\cp$ where $\ul{g}$ and $\ul{j}$ are right minimal right $(\cu/\cp)$-approximations and $\ul{f}$ and $\ul{i}$ are left minimal left $(\cu/\cp)$-approximations, by \Cref{thm:IyamaYoshino}.
    We may assume that $T^\ast \in \ind \ce \setminus \ind \ct$, by \Cref{lem:radicalfunctor}, and that these triangles are induced by conflations as in \eqref{eqn:holyconflations}, by \Cref{lem:lifttriangle}.
    By \Cref{cor:jenny} and its dual $f$ and $i$ are left $\cu$-approximations and $g$ and $j$ are right $\cu$-approximations.
    Since $T,T^\ast \notin \add \cu$ we have that $f$ and $i$ are in the radical, by \Cref{lem:radical}.
    Hence, $g$ and $j$ are right minimal using that the Yoneda-embedding of $\ce_{\cp}$ is left exact and \Cref{lem:minradical,lem:minimalfun}.
    Dually, one shows that $f$ and $i$ are left minimal.
\end{proof}

%
%
%

%
%
%
%
%
%
%
%
%
%
%
%

%
%
%
%
%
%
%
%
%
%
%
%
%
%
%
%
%
%
%
%
%
%
%
%
%
%
%
%
%
%
%
%
%
%
%
%
%

%
%
%
%
%
%
%
%
%
%
%
%
%
%
%
%
%
%
%
%
%
%
%
%
%
%

\section{More details on the singularities in \Cref{T:Main}}\label{S:AppendixDetailsOnSing}

We use the following notation throughout this section
\begin{Setup}
    Let $P_d \coloneqq \C \llbracket z_0, \dots, z_d\rrbracket$ for $d \in \N$.
\end{Setup}
There are two (non-disjoint) classes of odd-dimensional singularities appearing in the statement of our main result \Cref{T:Main}.

Firstly, the odd-dimensional ADE-hypersurface singularities\footnote{In dimension $2$, the ADE-hypersurface singularities are also known as Du Val singularities, Kleinian singularities, simple surface singularities, canonical singularities or rational double points. They play a key role in the definition of cDV-singularities, cf.\ \Cref{Reid}. } $P_d/(f)$ listed in \Cref{rem:explicit}. 
Secondly, the singularities of type $\SMAL$ (cf.\ \Cref{D:frakS}), which are generalizations of ADE-hypersurface singularities, see \Cref{R:BIKR}. Our description of these singularities in \Cref{T:ClassOfSing} is as explicit as the description of isolated Gorenstein threefold singularities admitting a small resolution of singularities.

\begin{thm}[Reid \cite{Reid83}]\label{Reid}
Let $S \in \SMAL^{(3)}$, i.e.\ $S$ is an isolated Gorenstein complete local threefold singularity over $\C$ that admits a small resolution of singularities. Then $S$ is a \emph{compound Du Val (cDV)} singularity, i.e.\ there is an isomorphism  
\begin{align*}
    S \cong P_3/(g + z_3 h),
\end{align*}
 where $g \in P_2$ is of type ADE and $h \in P_3$. In particular, $S$ is a hypersurface singularity.
\end{thm}

The following result follows for example  from \cite[Theorems 5.7 and 6.2(c)]{BIKR}. 

\begin{prop}\label{P:BIKR}
Let $g \in P_1$ such that $P_3/(g + z_2^2 + z_3^2) \in \SMAL^{(3)}$. Then \begin{align}\label{E:PrimeFact}
  g=f_1 \cdots f_n,  
\end{align} where $f_i \in (z_0, z_1) \setminus (z_0, z_1)^2$ are irreducible and mutually prime. Conversely, for every $g$ as in \eqref{E:PrimeFact} the singularity $P_3/(g + z_2^2 + z_3^2)$ is in $\SMAL^{(3)}$.
\end{prop}

\begin{rem}\label{R:BIKR}
    Threefolds of the form $P_3/(g + z_2^2 + z_3^2) $ belong to the class of $cA_n$-singularities, see e.g.\ \cite[Proposition 6.1(e)]{BIKR}.
    In particular, all ADE-hypersurface singularities in dimension $3$ are $cA_n$-singularities for $n=1$ or $2$, see for example \cite[Proposition 6.1(a), (b)]{BIKR}.
    Conversely, isolated $cA_n $-singularities are of the form $P_3/(g + z_2^2 + z_3^2) $ for some $g \in P_1$, which follows from Morse Lemma (see for example \cite[Section 11.1]{Morse}).
\end{rem}

\begin{thm} \label{T:ClassOfSing}
Let $S \in \SMAL$. 
Then $d:=\dim \, S$ is odd and there exists a $f \in P_3$ such that the following holds.
\begin{enumerate}[label={(\alph*)}]
\item The threefold $\Spec(P_3/(f))$ admits a small resolution of singularities.\label{item:threefoldsmallres}
\item If $d \geq 3$ there is an $\C$-algebra isomorphism
\[
S \cong P_d/(f + z_{4}^2 + \cdots + z_d^2).
\]
\item If $d =1$, then there exists $g=f_1 \cdots f_n \in P_1$ such that 
\begin{enumerate}[label={(\arabic*)}]
    \item $f_i \in (z_0, z_1) \setminus (z_0, z_1)^2$ are irreducible and mutually prime,
    \item $P_3/(f)\cong P_3/(g + z_2^2 + z_3^2)$,
    \item $S \cong P_1/(g)$.
\end{enumerate}

\end{enumerate}
\end{thm}
\begin{proof}
By definition of $\SMAL$, there exists a threefold $S'$ such that $\Spec(S')$ admits a small resolution of singularities and a triangle equivalence
\begin{align} \label{E:singuleq}
\Dsg(S) \cong \Dsg(S').
\end{align}
Since $\Spec(S')$ admits a small resolution, $S' \cong P_3/(f)$ is a hypersurface singularity by \Cref{Reid} and  $\Dsg(S')$ has a cluster-tilting object by \Cref{P:ClusterTiltingObject}. As observed by Keller, the triangulated Auslander--Iyama correspondence of Jasso--Muro \cite{JassoMuro} implies that $\Dsg(S')$ has a unique dg-enhancement. Hence, the equivalence \eqref{E:singuleq} lifts to the dg-level. Now, the main result of \cite{K21b} (combined with \Cref{P:BIKR} in the case $d=1$) completes the proof.
\end{proof}

    In view of our obstruction \Cref{T:PropertiesDsgR}\ref{item:threefoldsmallres},
    it is natural to wonder whether there are further isolated Gorenstein singularities $S$ such that the AR-quiver of $\Dsg(S)$ contains loops.
    The following result gives a partial negative answer.

    \begin{thm}[Wiedemann] \label{T:Wiedemann}
        Let $S \coloneqq P_d/I$ be a Gorenstein domain with an isolated singularity, such that one of the following conditions is satisfied:
        \begin{enumerate}[label={(\alph*)}]
            \item $\dim \, S =1$,\label{item:case1}
            \item $S = P_{2l+1}/(g(z_0, z_1) + z_2^2 + \cdots + z_{2l+1}^2),$ where $g \in P_1$ is irreducible.\label{item:case2}
        \end{enumerate}
        If the AR-quiver of $\Dsg(S)$ contains a loop, then $S$ is an $A_{2n}$-singularity. 
    \end{thm}
    \begin{proof}
        By Knörrer's periodicity \cite{Knoerrer}, part \ref{item:case2} is a consequence of part \ref{item:case1}, so we can assume $\dim \, S=1$. 
        A special case of Wiedemann's \mbox{\cite[Corollary on page 354]{Wiedemann}} shows that the AR-quiver of $\ul{\MCM}(S) \cong \Dsg(S)$ equals the AR-quiver of $\Dsg(A_{2n})$ and in particular, $S$ has finite MCM-representation type. 
        
        For the convenience of the reader, we first translate our setting into Wiedemann's and then sketch his argument. By a complete local version of Noether normalization (cf.\ \cite[\href{https://stacks.math.columbia.edu/tag/0323}{Section 0323, Lemma 10.160.11.}]{stacks-project}), there is an embedding of $\C$-algebras $R:=\C\llbracket t \rrbracket \subseteq S$ such that $S$ is a finitely generated $R$-module. Since $S$ is a domain, $S$ is torsion-free as an $R$-module and hence free, since $R$ is a DVR. This shows that $S$ is an $R$-order as in \cite{Wiedemann}. Moreover, in our setting, $S$-lattices coincide with maximal Cohen--Macaulay $S$-modules, so the AR-quiver in \cite{Wiedemann} is indeed the AR-quiver of $\ul{\MCM}(S)$. Let $K$ be the quotient field of $R$, then $KS:=K \otimes_R S$ is a field (as it is a domain which is finite-dimensional as a $K$-vectorspace). Every $S$-module $M$ yields a $KS$-module $KM:=K \otimes_R M$. Now, Wiedemann defines the $R$-rank of $M$ as $\dim_K(KM)$. For completeness, we note that the $R$-rank depends on the fixed embedding $R \subseteq S$ -- this does not affect Wiedemann's arguments. Since $KS$ is a field, the minimal possible $R$-rank of an $S$-module is $\dim_K(KS)$.

        Next, Wiedemann constructs the AR-component containing an indecomposable MCM $S$-module $M_0$ with a loop, see \mbox{\cite[Proof of Theorem 2]{Wiedemann2}}. 
        By construction, all indecomposable modules in this component have minimal possible $R$-rank, since this is true for $M_0$ by \cite[Part (i) of Proposition on page 353]{Wiedemann} and since the $R$-rank is additive on short exact sequences. We note that the key point is that all the MCM modules $M_i$ in Wiedemann's construction have to be indecomposable as they have minimal possible $R$-rank. 
        
        Summing up, there is a AR-component where the $R$-rank is bounded and hence \cite[Theorem 2]{Wiedemann} shows that the AR-quiver of $\Dsg(S)$ is finite and consists only of the component containing the loop. 
        By construction this component equals the AR-quiver of $\Dsg(A_{2n})$, cf.~\mbox{\cite[Corollary on page 354]{Wiedemann}}.
        
        Now, by \cite[Corollary 8.16]{YoshinoBook}, $S$ is an ADE-hypersurface singularity. The ADE-hypersurface singularities in any fixed dimension are completely determined by the AR-quiver of their singularity categories \cite{YoshinoBook}. This completes the proof.
    \end{proof}

\emph{Acknowledgements.} 
We thank Evgeny Shinder for many interesting discussions. 
We thank Wassilij Gnedin for discussions about Theorem \ref{T:Wiedemann} and Norihiro Hanihara, Yujiro Kawamata, Sasha Kuznetsov, Steffen König, Duco van Straten and Fei Xie for their interest and comments on an earlier version.

Martin Kalck was partially funded by the Deutsche Forschungsgemeinschaft (DFG,
German Research Foundation) – Projektnummern 496500943; 201167725.
He would like to thank his family, Wolfgang Soergel and the Mathematical Institute in Freiburg for giving him the opportunity to start working on this paper. 
He is also very grateful to Dong Yang and Michael Wemyss for answering his questions and to the GK1821 in Freiburg for (partial) financial support.
He would like to thank Peter Jørgensen for an invitation to Aarhus, which can be seen as the starting point of the collaboration between the first and second named author.

Carlo Klapproth was supported by Aarhus University Research Foundation, grant no.\ AUFF-F-2020-7-16, and the EliteForsk Rejsestipendium from the Independent Research Fund Denmark, grant no.\ 2083-00072B. He would like to thank his supervisor Peter Jørgensen for guidance and support. 

Crucial progress was made during the conference "Silting Theory, Algebras and Representations" in Prague in September 2022. The first and second named author are very grateful to Jan Šťovíček and the DFG network on "Silting Theory" for organizing this conference, for providing excellent working conditions and for financial support.

\addtocontents{toc}{\protect\setcounter{tocdepth}{0}}

\end{document}